\documentclass[oneside,english]{amsart}
\usepackage[T1]{fontenc}
\usepackage[latin9]{inputenc}
\usepackage{prettyref}
\usepackage{amsthm}
\usepackage{amssymb}

\makeatletter
\numberwithin{equation}{section}
\numberwithin{figure}{section}
\theoremstyle{plain}
\newtheorem{thm}{\protect\theoremname}
  \theoremstyle{definition}
  \newtheorem{defn}[thm]{\protect\definitionname}
  \theoremstyle{plain}
  \newtheorem{lem}[thm]{\protect\lemmaname}
  \theoremstyle{remark}
  \newtheorem{claim}[thm]{\protect\claimname}
  \theoremstyle{remark}
  \newtheorem{rem}[thm]{\protect\remarkname}
  \theoremstyle{plain}
  \newtheorem{cor}[thm]{\protect\corollaryname}
  \theoremstyle{plain}
  \newtheorem{prop}[thm]{\protect\propositionname}
  \theoremstyle{definition}
  \newtheorem{example}[thm]{\protect\examplename}

\usepackage{amsthm, amssymb, amsmath, color, marvosym,
rotating, graphicx, mathrsfs}

\usepackage{epsfig}
\usepackage{amscd}

\makeatletter
\newcommand*\leftdash{\rotatebox[origin=c]{-45}{$\dabar@\dabar@\dabar@$}}
\newcommand*\rightdash{\rotatebox[origin=c]{45}{$\dabar@\dabar@\dabar@$}}
\makeatother

\makeatother

\usepackage{babel}
  \providecommand{\claimname}{Claim}
  \providecommand{\corollaryname}{Corollary}
  \providecommand{\definitionname}{Definition}
  \providecommand{\examplename}{Example}
  \providecommand{\lemmaname}{Lemma}
  \providecommand{\propositionname}{Proposition}
  \providecommand{\remarkname}{Remark}
\providecommand{\theoremname}{Theorem}

\begin{document}

\title{Equivariant Coherent Sheaves, Soergel Bimodules, and Categorification
of Affine Hecke Algebras}

\author{Christopher Dodd}
\begin{abstract}
We give a description of certain categories of equivariant coherent
sheaves on Grothendieck's resolution in terms of the categorical affine
Hecke algebra of Soergel. As an application, we deduce a relationship
of these coherent sheaf categories to the categories of perverse sheaves
considered in \cite{key-12}, generalizing results of \cite{key-1}.
In addition, we deduce that the weak braid group action on sheaves
of \cite{key-11} can be upgraded to a strict braid group action.
\end{abstract}
\maketitle
\tableofcontents{}

\section{Introduction. }

\subsection{Kazhdan-Lusztig Equivalence}

Categories of equivariant coherent sheaves play a crucial role in
geometric representation theory dating back at least to the work of
Kazhdan and Lusztig in the 1980's. The subject of their seminal work
\cite{key-23} was a description of irreducible representations of
the affine Hecke algebra associated to a given root datum. In order
to accomplish this, they first gave a geometric construction of the
affine Hecke algebra, which we shall now describe. 

Let $G$ be a complex reductive algebraic group with simply connected
derived group, and let $\check{G}$ be its (complex reductive) Langlands
dual group. In particular, let us fix a pinning of $G$, consisting
of the choice of a Borel subgroup $B$, and a maximal torus $T$.
To this pinning there is associated the root datum $(\mathbb{X},\mathbb{Y},\Phi,\check{\Phi})$,
and the Weyl group $W$, with $S$ its set of simple reflections.
Then, the root datum $(\mathbb{Y},\mathbb{X},\check{\Phi},\Phi)$
also comes from an algebraic group, denoted $\check{G}$, with its
pinning $\check{B}$ and $\check{T}$ . We let $\mathfrak{g}$ and
$\check{\mathfrak{g}}$ denote the lie algebras of $G$ and $\check{G}$,
respectively, and $\mathfrak{b}$, $\check{\mathfrak{b}}$, $\mathfrak{h}$,
$\check{\mathfrak{h}}$ the lie algebras of our fixed Borel and Cartan
subalgebras.

Now, we let $H_{\mathbb{X}}$ denote the affine Hecke algebra associated
to $\mathbb{X}$. Let us recall Bernstein's presentation of this algebra: 
\begin{defn}
$H_{\mathbb{Y}}$ is the free $\mathbb{Z}[q,q^{-1}]$-module with
basis $\{e^{\lambda}T_{w}|\lambda\in\mathbb{X},w\in W\}$ satisfying 

1) The $T_{w}$ span a $\mathbb{Z}[q,q^{-1}]$-subalgebra isomorphic
to the finite Hecke algebra $H_{W}$. 

2) The $\{e^{\lambda}\}$ (by which we mean $\{T_{e}e^{\lambda}\}$)
satisfy $e^{\lambda_{1}}e^{\lambda_{2}}=e^{\lambda_{1}+\lambda_{2}}$). 

3) Let $\alpha$ be a simple root and $s_{\alpha}$ its simple reflection.
Then for $\lambda\in\mathbb{X}$ 
\[
T_{s_{\alpha}}e^{s_{\alpha}(\lambda)}-e^{\lambda}T_{s_{\alpha}}=(1-q)\frac{e^{\lambda}-e^{s_{\alpha}(\lambda)}}{1-e^{-\alpha}}
\]

\end{defn}
Kazhdan and Lusztig have constructed $H_{\mathbb{X}}$ entirely in
terms of the geometry of the group $G$. We first recall the flag
variety of $G$, denoted $\mathcal{B}$, which we shall regard as
the variety of all Borel subalgebras of $\mathfrak{g}$. Of course,
given the pinning chosen in the previous paragraph, we can identify
$\mathcal{B}\tilde{=}G/B$ as homogeneous spaces. We next recall that
there is associated to $G$ a morphism $\tilde{\mathcal{N}}\to\mathcal{N}$
called the springer resolution. Here $\mathcal{N}$ denotes the nilpotent
cone associated to $G$, which is defined as follows: we first define
\[
\mathcal{N}^{*}=\{x\in\mathfrak{g}|ad(x)^{dim\mathfrak{g}}=0\}
\]
 and then we define $\mathcal{N}\subset\mathfrak{g}^{*}$ as the transport
of $\mathcal{N}^{*}$ under the natural isomorphism (given by the
killing form) $\mathfrak{g}\tilde{=}\mathfrak{g}^{*}$. 

Next, $\tilde{\mathcal{N}}$ denotes the incidence variety defined
by 
\[
\tilde{\mathcal{N}}=\{(x,\mathfrak{b})\in\mathfrak{g}^{*}\times\mathcal{B}|x|_{\mathfrak{b}}=0\}
\]
Then the morphism $\tilde{\mathcal{N}}\to\mathcal{N}$ is given by
the first projection.

Finally, we can construct from here the Steinberg variety, defined
as 
\[
St_{G}=\tilde{\mathcal{N}}\times_{\mathcal{N}}\tilde{\mathcal{N}}
\]

We see immediately from the definition that $St_{G}$ has an action
by $G\times\mathbb{C}^{*}$, where $G$ acts via its obvious action
on $\tilde{\mathcal{N}}$ and $\mathcal{N}$, and the action of $\mathbb{C}^{*}$
is given by dilation of the first coordinate in $\tilde{\mathcal{N}}$.
Therefore we can consider the (complexified) $K$ group $K^{G\times\mathbb{C}^{*}}(St_{G})$,
which is naturally a $\mathbb{C}[q,q^{-1}]$-module, where the parameter
$q$ acts by shifting the grading induced by the $\mathbb{C}^{*}$-action. 

Further, this module also has the structure of an algebra, where the
product is given via the exterior product of sheaves: let $\mathcal{F}$
and $\mathcal{G}$ be two $G\times\mathbb{C}^{*}$-equivariant sheaves
on $St_{G}$. Consider the variety $\tilde{\mathcal{N}}\times_{\mathcal{N}}\tilde{\mathcal{N}}\times_{\mathcal{N}}\tilde{\mathcal{N}}$,
which has three projections to $St_{G}$, which we shall denote $p_{i}$
where $i=\{1,2,3\}$ is the omitted factor. Then we can consider the
complex in $D^{b,G\times\mathbb{C}^{*}}(Coh(St_{G}))$ (the bounded
derived category of equivariant coherent sheaves; full details about
such objects are recalled in section 2 below) 
\[
p_{2*}(p_{3}^{*}(\mathcal{F})\otimes p_{1}^{*}(\mathcal{G}))
\]
 where all the functors are taken to be derived. This complex defines
a class in $K$ theory, which is then the product of $[\mathcal{F}]$
and $[\mathcal{G}]$. 

With all of this is hand, we can state the theorem of Kazhdan and
Lusztig (actually the slightly stronger version proved in \cite{key-13}),
which says that there is an isomorphism: 
\[
H_{\mathbb{X}}\tilde{=}K^{G\times\mathbb{C}^{*}}(St_{G})
\]
 Let us explain how this isomorphism works, in terms of the presentation
of $H_{\mathbb{X}}$ given above. We shall follow the explanation
of Riche \cite{key-28}. First, given $\lambda\in\mathbb{X}$, we
associate the bundle $O(\lambda)$ on the flag variety $\mathcal{B}$,
and then via pullback, a bundle $O_{\tilde{\mathcal{N}}}(\lambda)$.
Now, we have the diagonal embedding $\Delta\tilde{\mathcal{N}}\to St_{G}$.
So we can consider the equivariant coherent sheaf $O_{\Delta\tilde{\mathcal{N}}}(\lambda)$
on $St_{G}$; the map will send $e^{\lambda}$ to the class $[O_{\Delta\tilde{\mathcal{N}}}(\lambda)]\in K^{G\times\mathbb{C}^{*}}(St_{G})$. 

Next, consider $s\in S$. We associate to $s$ a partial flag variety
of $G$, denoted $\mathcal{P}_{s}=G/P_{s}$, where $P_{s}$ is the
standard parabolic in $G$ of type $s$ containing $B$. Then we can
form the variety $\mathcal{B}\times_{\mathcal{P}_{s}}\mathcal{B}$,
which is naturally a closed subscheme of $\mathcal{B}\times\mathcal{B}$.
From here we define the variety 
\[
S_{\alpha}^{'}=\{(X,g_{1}B,g_{2}B)\in\mathfrak{g}^{*}\times\mathcal{B}\times_{\mathcal{P}_{\alpha}}\mathcal{B}|\, X|_{g_{1}\mathfrak{b}+g_{2}\mathfrak{b}}=0\}
\]
 It is easy to see that this is a subvariety of $St_{G}$, which is
$G\times\mathbb{C}^{*}$ equivariant. Then we have our map send $T_{s}$
to $-q^{-1}[O_{S_{\alpha}^{'}}]$. The fact that this really is an
isomorphism of algebras is checked in \cite{key-13}, c.f. also the
main result of Riche in \cite{key-28} (shown in full generality in
\cite{key-11}). We shall discuss the results of these papers in more
detail later. 

For our purposes in this work, the interest in this result lies in
the fact that the Steinberg variety can now be viewed as a {}``categorification''
of the affine Hecke algebra.

\subsection{Soergel Bimodules}

In this section, we discuss a different categorification of the affine
Hecke algebra, which can be found in the works of Wolfgang Soergel
\cite{key-30,key-31,key-32} and Rafael Rouquier \cite{key-29}. For
this categorification, it is appropriate to consider the general case
where $(W,S)$ is any Coxeter system for which the generating set
of involutions $S$ is finite.

Let us recall that in this generality, the Hecke algebra is defined
as the $\mathbb{Z}[q,q^{-1}]$-algebra generated by symbols $\{T_{s}\}_{s\in S}$,
which satisfy the braid relations for $S$, and also the additional
relation $(T_{s}+1)(T_{s}-q)=0$. 

Then, over an algebraically closed field of sufficiently large characteristic
(including zero), $k$, we have the geometric representation of $W$,
on a finite dimensional vector space $V$. Then there is a categorification
which can be constructed from this purely combinatorial set-up. We
shall follow closely the notation and constructions in \cite{key-29}. 

So, we let $\{e_{s}\}_{s\in S}$ be the natural basis of $V$, and
we let $\{\alpha_{s}\}_{s\in S}$ be its dual basis. Therefore we
have that 
\[
ker(\alpha_{s})=ker(s-id)
\]
for all $s\in S$. We define $A$ to be $Sym(V^{*})$, i.e., the algebra
of polynomial functions on $V$, and we let $A^{en}=A\otimes_{k}A$.
We consider both of these as graded rings by putting $V^{*}$ in degree
$2$. We shall work with certain graded $A^{en}$-modules which are
known as Soergel bimodules (they first appeared in the paper \cite{key-31}). 

To define these bimodules, we first note that for any $s\in S$, we
have the subalgebra $A^{s}$ of elements fixed under the action of
$s\in S$. We have a decomposition 
\[
A=A^{s}\oplus A^{s}\alpha_{s}
\]
as $A^{s}$-modules (this follows immediately from the fact that $s$
has order two). 

Now, let $w\in W$ be any element. We can write $w=\Pi_{i=1}^{m}s_{i}$,
a minimal length decomposition (where $s_{i}\in S$). We associate
to any such decomposition the graded $A^{en}$-module 
\[
A\otimes_{A^{S_{1}}}A\otimes_{A^{s_{2}}}\otimes_{A^{s_{3}}}\cdot\cdot\cdot\otimes_{A^{s_{m}}}A
\]
 We note that the element $e\in W$ gets the {}``diagonal'' bimodule
$A$ by definition. 

Then, as shown in the works of Soergel, each of these bimodules admits
a decomposition as a direct sum of indecomposable graded $A-A$ bimodules,
called the Soergel bimodules. In fact, it is even shown that to each
such $w\in W$, one can associate a unique $B_{w}$ which occurs as
a summand (with multiplicity one) in $A\otimes_{A^{S_{1}}}A\otimes_{A^{s_{2}}}\otimes_{A^{s_{3}}}\cdot\cdot\cdot\otimes_{A^{s_{m}}}A$
and which occurs in no bimodule $A\otimes_{A^{S_{1}^{'}}}A\otimes_{A^{s_{2}^{'}}}\otimes_{A^{s_{3}^{'}}}\cdot\cdot\cdot\otimes_{A^{s_{n}^{'}}}A$
which corresponds to a lower element in the Bruhat ordering.

Therefore, one makes the following 
\begin{defn}
\cite{key-32} The category $\mathcal{H}(W)$ is the smallest category
of $A^{en}$ bimodules containing the Soergel bimodules and closed
under direct sums, summands, and tensor product. This is an additive
category, which has a monoidal structure via 
\[
(M,N)\to M\otimes_{A}N
\]

\end{defn}
Next, we consider the (complexified) split Grothendieck group of this
category, which is a $\mathbb{C}[q,q^{-1}]$-algebra, where the parameter
$q$ acts by shifting the grading, and the multiplication is the image
in $K$-theory of the above monoidal structure. We state the main
result of \cite{key-32} as 
\[
K(\mathcal{H}(W))\tilde{=}H(W)
\]
 as $\mathbb{C}[q,q^{-1}]$ algebras. The map $H(W)\to K(\mathcal{H}(W))$
is given by $T_{s}\to[A\otimes_{A^{s}}A]-1$, and $q\to A[1]$. 

Thus the category $\mathcal{H}(W)$ provides another categorification
of the Hecke algebra of a Coxeter group, of which the affine Hecke
algebra is a particular example (where $W$ is the affine Weyl group
of a given root datum).

\subsection{Perverse Sheaves}

We now discuss a third categorical realization of the affine Hecke
algebra- the one that, in fact, is closest to the actual definition
of $H_{aff}$. This realization follows from applying Grothendieck's
{}``sheaves-functions'' correspondence to the affine flag variety. 

In particular, for our given reductive group $G$, we consider the
formal loop group $G((t))$, and the Iwahori subgroup $I$ (see section
5). 

Since the affine Hecke algebra, by definition, is the algebra of $I$-equivariant
functions on $\mathcal{F}l:=G((t))/I$, the sheaves-functions philosophy
would predict that it can be categorified as the category of $I$-equivariant
mixed perverse sheaves on $\mathcal{F}l$, and this is indeed a theorem
of Kazhdan and Lusztig \cite{key-24}. So in particular we have an
isomorphism of $\mathbb{C}[q,q^{-1}]$-algebras 
\[
K(Perv_{I}^{mix}(G((t))/I)\tilde{=}H_{\mathbb{Y}}
\]
 where the group on the left is the complexified $K$-group, and the
graded element $q$ acts by a the Tate twist of mixed sheaves (c.f.,
e.g.,\cite{key-5}). Note that the Hecke algebra on the right is the
one associated to the dual group of $G$. 

We explain how the map works, in the case that $\mathbb{Y}$ is the
root lattice $\mathbb{Z}\Phi$ (in the general case the variety $\mathcal{F}l$
is a finite cover of this one). By the standard Bruhat decomposition,
the $I$-orbits on $G((t))/I$ are parametrized by the group $W_{aff}$.
As is well known, each of these orbits is isomorphic to an affine
space. For $w\in W_{aff}$, let $Z_{w}$ denote its orbit. Then we
can consider the trivial sheaf $\bar{\mathbb{Q}}_{l,Z_{w}}$. If we
let $\bar{Z_{w}}$ denote the closure of $Z_{w}$ in $G((t))/I$,
and $j_{w}:Z_{w}\to\bar{Z_{w}}$ the natural inclusion, then we can
also consider $j_{w!}(\bar{\mathbb{Q}}_{l,Z_{w}}[dimZ_{w}])$. These
classes of these objects in $K(Perv_{I}^{mix}(G((t))/I)$ form a natural
$\mathbb{C}[q,q^{-1}]$-basis, and the map is then given by $T_{w}\to q^{-l(w)}[j_{w!}(\bar{\mathbb{Q}}_{l,Z_{w}}[dimZ_{w}])]$
(here $l(w)$ denotes the length).

\subsection{Description of the Main Theorems}

In order to explain our goal, we first explain a variant of the above
results, which is the categorification of the standard (or aspherical)
module for a given affine Hecke algebra, $H_{\mathbb{X}}$. This module,
denoted $M_{asp}$, is defined by the induction 
\[
M_{asp}=H\otimes_{H_{fin}}sgn
\]
 where $H_{fin}\subseteq H$ denotes the inclusion of the finite Hecke
algebra, and $sgn$ is the one-dimensional sign representation of
$H_{fin}$. The main result of \cite{key-1} is a categorification
of this module (and the action of $H$ on it) as follows (c.f. \cite{key-1}
section 1 for all definitions): 

We work with the Langlands dual reductive group $\check{G}$. Let
$I^{-}$ denote the opposite Iwahori subgroup, and let $\psi$ be
a generic character on it. Then we define $D_{IW}$ to be the $(I^{-},\psi)$-equivariant
derived category of constructible sheaves on $\mathcal{F}l$; it admits
a mixed version $D_{IW,m}$. There is an action of the category $D^{b}(Perv_{I}(\mathcal{F}l))$
on $D_{IW}$, defined by convolution%
\footnote{In this case, by an action we simply mean that to each object of $D^{b}(Perv_{I}(\mathcal{F}l))$,
there is a assigned an endofunctor of $D_{IW}$, and this assignment
is natural. %
}. This works as follows: by the definition of $\mathcal{F}l$ as a
quotient, there is a map $a:\mathcal{F}l\times_{I}\mathcal{F}l\to\mathcal{F}l$
(the image of group multiplication). So, for $\mathcal{G}\in D_{IW}$
and $\mathcal{K}\in D^{b}(Perv_{I}(\mathcal{F}l))$, we take $a_{*}(\mathcal{G}\boxtimes\mathcal{K})\in D_{IW}$.
We note that, by using the same formulae, this action extends to the
mixed versions of these categories. 

Then the claim is that (the mixed version of) this action is a categorification
of the action of $H_{\mathbb{X}}$ on $M_{asp}$ in the sense of the
above sections (i.e., after taking graded $K$-groups one recovers
this action). 

Next, there is another categorification of $M_{asp}$ in terms of
coherent sheaves. We work with the group $G$, over the field $\bar{\mathbb{Q}}_{l}$.
There is a convolution action of the Steinberg variety $St=\tilde{\mathcal{N}}\times_{\mathcal{N}}\tilde{\mathcal{N}}$
on the variety $\tilde{\mathcal{N}}$. This induces an action (in
the same sense as above) of $D^{b,G\times\mathbb{G}_{m}}(Coh(St))$
on $D^{b,G\times\mathbb{G}_{m}}(Coh(\tilde{\mathcal{N}}))$ which
categorifies the action of $H_{\mathbb{X}}$ on $M_{asp}$. Then the
main result of \cite{key-1} states that there is an equivalence of
triangulated categories 
\[
D^{b,G\times\mathbb{G}_{m}}(Coh(\tilde{\mathcal{N}}))\tilde{\to}D_{IW,m}
\]
 such that the shift in grading on the left corresponds to the shift
in mixed structure on the right. The proof there relies on the main
result of \cite{key-15} and the geometric Satake equivalence. 

Below, we shall present a different approach to this result and several
{}``deformed'' versions of it, by relating both sides to relevant
categories of Soergel bimodules. In particular, the category on the
left admits two natural deformations, explained in detail in section
2 below, denoted $D^{b,G\times\mathbb{G}_{m}}(Coh(\tilde{\mathfrak{g}}))$
and $D^{b,G\times\mathbb{G}_{m}}(Mod(\tilde{D}_{h}))$. The first
is a deformation purely in the world of commutative algebraic geometry,
while the second is a noncommutative deformation quantization. We
shall study these categories in detail, and give a bimodule description
of them in section 4 below.

On the perverse sheaves side, several deformations of the category
$D_{IW}$ have been defined and studied in the paper \cite{key-12}.
There, they develop a Koszul duality formalism for these categories,
the major technical part of which is to relate these categories to
appropriate Soergel bimodules. Combining this with our description,
we arrive at equivalences of categories generalizing the one above
(see section 5 for details).

\subsection{Summary of the Major Argument}

Our strategy is to relate the three different species of categorification
of the affine Hecke algebra (or rather its standard module). We shall
indicate how the argument works. The first task is to define a \emph{tilting
collection }for this triangulated category. 

Following \cite{key-25}, we recall that an algebraic triangulated
category is one which can be constructed as the stable triangulated
category of a Frobenius exact category. We shall not recall the details
of these definitions, but the important thing for us is that all triangulated
categories arising in algebra and algebraic geometry (e.g., derived
and homotopy categories of abelian and dg categories) are algebraic
triangulated categories.

Then we begin with the: 
\begin{lem}
\cite{key-25} Let $\mathcal{C}$ be an algebraic triangulated category.
Let $T\in\mathcal{C}$ be an object such that 

1) $T$ generates $\mathcal{C}$ (in the sense that the full subcategory
of $\mathcal{C}$ containing $T$ and closed under extensions, shifts,
and direct summands is $\mathcal{C}$ itself).

2) We have $Hom_{\mathcal{C}}(T,T[n])=0$ for all $n\neq0$. 

Then the functor $M\to Hom_{\mathcal{C}}(T,M)$ is an equivalence
of categories $\mathcal{C}\to Perf(End_{\mathcal{C}}(T))$ where the
latter denotes the homotopy category of perfect complexes over $End_{\mathcal{C}}(T)$. 
\end{lem}
In this case, the object $T$ is called a tilting object of $\mathcal{C}$. 

Unfortunately, our categories will not admit such a simple description.
However, they will be generated by infinite collections of such objects.
Therefore we make the 
\begin{defn}
Let $\mathcal{C}$ be an algebraic triangulated category. Let $\mathcal{T}$
be a full subcategory such that 

1) $\mathcal{T}$ generates $\mathcal{C}$ in the following sense:
for each finite collection $I$ of objects of $\mathcal{T}$, let
$P_{I}=\bigoplus_{i\in I}T_{i}$, and let $\mathcal{C}_{I}$ be the
full subcategory generated by $\mathcal{P}_{I}$ in the sense of the
lemma above; then $\lim_{I}\mathcal{C}_{I}\tilde{\to}\mathcal{C}$
(i.e. the natural inclusion is essentially surjective). 

2) We have $Hom_{\mathcal{C}}(\mathcal{T},\mathcal{T}[n])=0$ for
all $n\neq0$. 

Then $\mathcal{T}$ is called a tilting subcategory of $\mathcal{C}$. 
\end{defn}
In the cases relevant to us, we can assume that $Ob(\mathcal{T})$
is a countable, totally ordered set. Then for each $i\in\mathbb{N}$,
we define $P_{i}=\bigoplus_{j\leq i}T_{j}$. Then we can define the
categories $\mathcal{D}_{i}=Perf(End(P_{i}))$ with the natural inclusion
functors $D_{i}\to D_{j}$ for $i\leq j$. Then the assumptions on
$\mathcal{T}$ and the previous lemma imply the
\begin{claim}
There is an equivalence of categories $\mathcal{C}\to\lim_{i}(D_{i})$
which is given by the stable image of the functors $F_{i}:\mathcal{C}\to\mathcal{D}_{i}$,
$F_{i}(M)=Hom_{\mathcal{C}}(P_{i},M)$. We shall denote this limit
by $Perf(\mathcal{B})$ where $\mathcal{B}=\oplus_{i,j}(Hom(T_{i},T_{j}))$. \end{claim}
\begin{rem}
Let us note that the objects of $\mathcal{T}$ correspond to (direct
sums of) the summands of $\mathcal{B}$. Thus we also have an equivalence
$Perf(\mathcal{B})\tilde{\to}K^{b}(\mathcal{T})$ where the category
on the right denotes the homotopy category of complexes of objects
in $\mathcal{T}$. 
\end{rem}
In the main body of the paper, we shall explicitly identify tilting
collections for each of the categories $\mathcal{C}$ of interest
to us. We shall then define a functor 
\[
\kappa:\mathcal{T}\to(Mod^{gr}(O(\mathfrak{h}^{*}\times\mathbb{A}^{1}\times\mathfrak{h}^{*}/W))
\]

(notation for algebraic groups as above) which is called the Kostant-Whittaker
reduction. This functor is based on computing the action of the center
of each of the above categories; in this sense, it is a generalization
of the fundamental work of Soergel \cite{key-30}. A prototype also
appears in the paper \cite{key-8}.

We next show that when restricted to tilting modules, $\kappa$ is
fully faithful. We shall describe explicitly the image of $\kappa$,
and show that the sheaves that appear are a version of Soergel bimodules
(for the module $M_{asp}$); thus obtaining, by the remarks above,
a description of the entire category $\mathcal{C}$. 3

On the other hand, there is already a description, contained in the
paper \cite{key-12}, of the constructible categories under consideration
in terms of sheaves on the space $\mathfrak{h}^{*}\times\mathbb{A}^{1}\times\mathfrak{h}^{*}/W$.
This description also goes by finding a tilting collection, and also
explicitly describes the image in terms of Soergel bimodules. From
these compatible descriptions, we deduce the equivalence of categories
above.

\subsection{Further Results and Future Work}

Given the results outlined above, it is natural to ask about the categorifications
of the other standard modules. In particular, we know that the regular
representation of $H_{aff}$ on itself is categorified by certain
$G$-equivariant sheaves on the Steinberg variety (as discussed in
section 1.1), and also by the usual category of Soergel bimodules.
In a future work, we shall extend the results of this paper to show
that there is a fully-faithful functor $\kappa$ which takes a certain
subcategory of $D^{b}(Coh^{G\times\mathbb{G}_{m}}(\tilde{D}_{h}\boxtimes\tilde{D}_{h}))$
onto the category of Soergel bimodules for $W_{aff}$ (see section
2 below for definitions). In addition, we shall show the comparable
results for the categories of equivariant coherent sheaves on the
deformations of partial flag varieties $\tilde{\mathfrak{g}}_{\mathcal{P}}$
(see section 2 below). The method will be to describe these categories
in terms of the ones already considered here via the Barr-Beck theorem
applied to the appropriate push-pull functors. This will allow us
to upgrade the functor $\kappa$ of this paper to functors on these
other categories; and similar methods to the ones here can then be
used to deduce the full-faithfulness on the appropriate objects. The
details will appear in a forthcoming work.

\subsection{Acknowledgements}

I would like to extend my heartfelt thanks to my advisor, Roman Bezrukavnikov,
for his help and support over many years. The work presented here
is largely an outgrowth of his insights and perspective. In addition,
I would like to thank N. Rozenblyum, D. Jordan, J. Pascaleff, N. Sheridan,
Q. Lin, B. Singh, T-H Chen, and all of my fellow graduate students
at MIT for innumerable helpful conversations and insights. Finally,
some of this work was accomplished while visiting the Hebrew University
of Jerusalem. I would like to thank the University for providing an
excellent working atmosphere.

\section{Main Players-Coherent Side}

In this section, we shall give detailed explanations of the {}``coherent''
categories that we shall use. The constructions in this section make
sense over an arbitrary algebraically closed field $k$ of characteristic
greater than the Coxeter number of $G$.

\subsection{Varieties}

We start with the varieties $\tilde{\mathfrak{g}}$ and $\tilde{\mathfrak{g}}_{\mathcal{P}}$-
details on these can be found in many places, e.g. \cite{key-10}
and \cite{key-2}.

\subsubsection{Full Flag Varieties}

The variety $\tilde{\mathfrak{g}}$ is the total space of a bundle
over $\tilde{\mathcal{N}}$, and is defined as 
\[
\tilde{\mathfrak{g}}=\{(x,\mathfrak{b})\in\mathfrak{g}^{*}\times\mathcal{B}|x|_{[\mathfrak{b},\mathfrak{b}]}=0\}
\]
the first projection defines a morphism $\tilde{\mathfrak{g}}\to\mathfrak{g}^{*}$,
and so we have a map $O(\mathfrak{g}^{*})\to O(\tilde{\mathfrak{g}})$. 

Further, we can explicitly identify the global sections $\Gamma(O(\tilde{\mathfrak{g}}))$
as follows. Recall the Harish-Chandra isomorphism 
\[
O(\mathfrak{g}^{*})^{G}\tilde{\to}O(\mathfrak{h}^{*}/W)
\]
 This makes $O(\mathfrak{h}^{*}/W)$ into a subalgebra of $O(\mathfrak{g}^{*})$.
Then we have 
\begin{equation}
\Gamma(O(\tilde{\mathfrak{g}}))\tilde{=}O(\mathfrak{g}^{*})\otimes_{O(\mathfrak{h}^{*}/W)}O(\mathfrak{h}^{*})
\end{equation}

Thus the map $\mu:\tilde{\mathfrak{g}}\to\mathfrak{g}^{*}$ factors
through another map, which by abuse of notation we shall also call
$\mu:\tilde{\mathfrak{g}}\to\mathfrak{g}^{*}\times_{\mathfrak{h}^{*}/W}\mathfrak{h}^{*}$.
We shall state an important property of this map: we consider the
locus of regular elements $\mathfrak{g}^{*,reg}$, defined as follows:
one first defines $\mathfrak{g}^{reg}$ to be the set $\{x\in\mathfrak{g}|dim(\mathfrak{c}_{\mathfrak{g}}(x))=rank(\mathfrak{g})\}$,
and then transfers this open set to $\mathfrak{g}^{*}$ via the killing
isomorphism.

Now, define $(\tilde{\mathfrak{g}})^{reg}=\mu^{-1}(\mathfrak{g}^{*,reg})$.
Then we have the following 
\begin{lem}
\label{lem:Regiso}The map $\mu:(\tilde{\mathfrak{g}})^{reg}\to\mathfrak{h}^{*}\times_{\mathfrak{h}^{*}/W}\mathfrak{g}^{*,reg}$
is an isomorphism . 
\end{lem}
The proof can be found in \cite{key-14}, section 7. We shall see
below that the regular elements play a major role in capturing the
$G$-equivariant geometry of $\tilde{\mathfrak{g}}$.

\subsubsection{Partial Flag Varieties}

Now we wish to extend this definition by having parabolic subalgebras
of different types play the role of the Borel subalgebra. Thus we
let $\mathcal{P}$ be a given partial flag variety. For a parabolic
$\mathfrak{p}\in\mathcal{P}$, we let $\mathfrak{u}(\mathfrak{p})$
denote its nilpotent radical. Then we define 
\[
\tilde{\mathfrak{g}}_{\mathcal{P}}\tilde{=}\{(x,\mathfrak{p})\in\mathfrak{g}^{*}\times\mathcal{P}|x|_{\mathfrak{u}(\mathfrak{p})}=0\}
\]

In the case $\mathcal{P}=\mathcal{B}$, this recovers $\tilde{\mathfrak{g}}$.
In case $\mathcal{P}=pt$ (i.e., $\mathfrak{p}=\mathfrak{g}$), this
is simply $\mathfrak{g}^{*}$. 

As before, we have a map $\tilde{\mathfrak{g}}_{\mathcal{P}}\to\mathfrak{g}^{*}$,
and we now have the isomorphism 
\begin{equation}
\Gamma(O(\tilde{\mathfrak{g}}_{\mathcal{P}}))\tilde{=}O(\mathfrak{g}^{*})\otimes_{O(\mathfrak{h}^{*}/W)}O(\mathfrak{h}^{*}/W(\mathcal{P}))
\end{equation}

Where $W(\mathcal{P})$ is the Weyl group of parabolic type associated
to this partial flag variety. 

Let us note that there are certain natural projection maps between
these varieties. In particular, let $\mathfrak{p}\subset\mathfrak{q}$
be two parabolic subalgebras. Then there is a natural projection map
of partial flag varieties $\pi:\mathcal{P}\to\mathcal{Q}$. This then
induces a map 
\[
\pi_{\mathcal{P}\mathcal{Q}}:\tilde{\mathfrak{g}}_{\mathcal{P}}\to\tilde{\mathfrak{g}}_{\mathcal{Q}}
\]
 defined by sending $(x,\mathfrak{p})$ to $(x,\pi(\mathfrak{p}))$. 

Some special cases will be of interest to us. First, let us note that
for any $\mathfrak{p}$, taking $\mathfrak{q}=\mathfrak{g}$ yields
the natural map $\tilde{\mathfrak{g}}_{\mathcal{P}}\to\mathfrak{g}^{*}$
considered above. 

Next, we want to consider the case of the projection $\mathcal{B}\to\mathcal{P}_{s}$,
where $s$ is a simple reflection. We shall record two important properties
of the map $\pi_{\mathcal{B}\mathcal{P}_{s}}=\pi_{s}$. First, we
have an isomorphism 
\begin{equation}
\pi_{s*}(O(\tilde{\mathfrak{g}}))=O(\tilde{\mathfrak{g}}_{s})\otimes_{O(\mathfrak{h}^{*}/W(s))}O(\mathfrak{h}^{*})
\end{equation}

thus this pushforward is a locally free sheaf of rank $2$. The second,
related fact, is that one can consider the restriction of this map
to the regular locus as follows: the map $\mu=\pi_{\mathcal{B}pt}:\tilde{\mathfrak{g}}\to\mathfrak{g}^{*}$
factors as 
\[
\tilde{\mathfrak{g}}\to\mathfrak{\tilde{g}}_{s}\to\mathfrak{g}^{*}
\]
 and so, taking the inverse image of the regular locus, we have a
map 
\[
(\tilde{\mathfrak{g}})^{reg}\to(\mathfrak{\tilde{g}}_{s})^{reg}
\]
 Then this map is a two sheeted covering map, with fibres are naturally
isomorphic to $W/W(s)$. 

In fact, we shall state the somewhat more general 
\begin{lem}
\label{lem:Regiso2}For all parabolic subgroups, the projection map
$\tilde{\mathfrak{g}}_{\mathcal{P}}^{reg}\to\mathfrak{g}^{*,}{}^{reg}$
induces an isomorphism 
\[
\tilde{\mathfrak{g}}_{\mathcal{P}}^{reg}\tilde{=}\mathfrak{h}^{*}/W_{\mathcal{P}}\times_{\mathfrak{h}^{*}/W}\mathfrak{g}^{reg}
\]
 
\end{lem}
The proof of this is similar to that of the case $\mathcal{P}=\mathcal{B}$
discussed above.

\subsubsection{Equivariance}

Here we would like to note that there is a natural action of the reductive
group $G\times\mathbb{G}_{m}$ on all the varieties we have considered.
The $G$ action comes essentially from the construction of the varieties-
for any $\tilde{\mathfrak{g}}_{\mathcal{P}}$, we can set 
\[
g\cdot(x,\mathfrak{p})=(ad^{*}(g)(x),ad(g)(\mathfrak{p}))
\]
where $ad^{*}$ and $ad$ are the coadjoint and adjoint action, respectively. 

The $\mathbb{G}_{m}$-action comes from the natural dilation action
on the lie algebra, i.e., 
\[
c\cdot(x,\mathfrak{p})=(cx,\mathfrak{p})
\]
 which obviously commutes with the action of $G$. 

Let us note that 
\[
\Gamma(O(\tilde{\mathfrak{g}}_{\mathcal{P}}))^{G}=O(\mathfrak{h}^{*}/W(\mathcal{P}))
\]
 simply by taking $G$-invariants on both sides of equation 2 above.

With this action defined, there are now abelian categories of equivariant
coherent sheaves $Coh^{G}(\tilde{\mathfrak{g}}_{\mathcal{P}})$ and
$Coh^{G\times\mathbb{G}_{m}}(\tilde{\mathfrak{g}}_{\mathcal{P}})$
(c.f. \cite{key-13}, chapter 5 for general information about equivariant
coherent sheaves) , and their derived categories, which will be some
of the main players in the paper. As it turns out, these categories
admit a very nice {}``affine'' description, as in the following: 
\begin{lem}
\label{lem:Coh-generation}The line bundles $O_{\tilde{\mathfrak{g}}}(\lambda)$
generate%
\footnote{In the sense that the smallest subcategory containing the line bundles
and closed under shifts, extensions, and direct factors is the entire
category%
} the category $D^{b}Coh^{G}(\tilde{\mathfrak{g}})$. The analogous
statement holds for the graded version.\end{lem}
\begin{proof}
We first recall that by definition $\tilde{\mathfrak{g}}=G\times_{B}\mathfrak{b}^{*}$.
Therefore there are equivalences of categories 
\[
i^{*}:Coh^{G}(\tilde{\mathfrak{g}})\tilde{\to}Coh^{B}(\mathfrak{b}^{*})
\]
and 
\[
i^{*}:Coh^{G\times\mathbb{G}_{m}}(\tilde{\mathfrak{g}})\tilde{\to}Coh^{B\times\mathbb{G}_{m}}(\mathfrak{b}^{*})
\]
given by restriction to the fibre over the base point of $\mathcal{B}$.
The inverse of this functor is given by taking the associated sheaf
of a $B$-module $M$ (c.f. \cite{key-21}, chapter 5), which yields
a quasicoherent sheaf on $\mathcal{B}$, and then noting that the
additional compatible structure of a $Sym(\mathfrak{b})$-module on
$M$ is equivalent to an action of $p_{*}(\tilde{\mathfrak{g}})$
on the associated sheaf.

So, to prove the lemma, we consider any finitely generated $B$-equivariant
module $M$ over $O(\mathfrak{b}^{*})$, and show that $M$ is in
the triangulated category generated by $i^{*}O_{\tilde{\mathfrak{g}}}(\lambda)$.
We choose a finite dimensional $B$-stable generating space for $M$,
called $V$. Recalling that $B=N\ltimes T$, we reduce to the case
that $N$ acts trivially on $V$ by considering a filtration of $V$
such that $N$ acts trivially on the subquotients. But then the proof
comes down to the statement is just that if we have a multigraded
polynomial ring, then any module has a finite resolution by graded
projective modules. The result for the entire bounded derived category
follows by induction on the length of complex. 
\end{proof}
Finally, we would like to end by describing one crucial property possessed
by equivariant coherent sheaves (in a general context), as explained
in \cite{key-22}. In particular, any equivariant coherent sheaf $M$
comes with a morphism of Lie algebras 
\[
L_{v}:\mathfrak{g}\to End_{k}(M)
\]
 obtained, essentially, by {}``differentiating the $G$-action''
(c.f. \cite{key-22} pg. 23 for details on the algebraic definition);
thus $M$ can be considered a sheaf of $\mathfrak{g}$-modules, and
hence a sheaf of $U(\mathfrak{g})$-modules. Further, for $A\in\mathfrak{g}$,
the operator $L_{v}(A)$ is a derivation on $M$.

\subsection{Deformations}

Now we shall consider certain non-commutative deformations of the
various varieties and maps considered above. Again these objects are
more or less well known, c.f. \cite{key-2}; in addition \cite{key-10}
consider the version in positive characteristic, and \cite{key-27}
has much of the material of the first subsection.

\subsubsection{Full Flag Varieties}

We shall start with $\tilde{\mathfrak{g}}$. By definition, $\tilde{\mathfrak{g}}$
is a vector bundle over $\mathcal{B}$. At a given point $\mathfrak{b}\in\mathcal{B}$,
the fibre of this bundle, $\tilde{\mathfrak{g}}_{\mathfrak{b}}$,
is equal to $\{x\in\mathfrak{g}^{*}|x|_{[\mathfrak{b},\mathfrak{b}]}=0\}$.
This is a vector space that can naturally be considered the dual of
$\mathfrak{b}$. 

We can quantize this situation, following \cite{key-27}. We start
with the sheaf $U^{0}=U(\mathfrak{g})\otimes O(\mathcal{B})$- a trivial
sheaf on $\mathcal{B}$. Let us note that the multiplication in this
sheaf is not the obvious one, but is instead given by the formula
\[
(f\otimes\xi)(g\otimes\eta)=f(\xi\cdot g)\otimes\eta+fg\otimes\xi\eta
\]
 where $\xi\cdot g$ denotes the action of a vector field on a function. 

The PBW filtration on $U(\mathfrak{g})$ gives a filtration on this
sheaf. It is clear that with respect to this filtration we have 
\[
gr(U(\mathfrak{g})\otimes O(\mathcal{B}))\tilde{=}O(\mathfrak{g}^{*})\otimes O(\mathcal{B})
\]
Further, the sheaf on the right is equal to $p_{*}(O(\mathfrak{g}^{*}\times\mathcal{B}))$,
where $p:\mathfrak{g}^{*}\times\mathcal{B}\to\mathcal{B}$ is the
obvious projection. Thus we can consider $U(\mathfrak{g})\otimes O(\mathcal{B})$
as a quantization of $\mathfrak{g}^{*}\times\mathcal{B}$. 

Now, for a given point $\mathfrak{b}\in\mathcal{B}$, we can consider
$\mathfrak{n}(\mathfrak{b})=[\mathfrak{b},\mathfrak{b}]$, the nilpotent
radical of $\mathfrak{b}$. We can define $\mathfrak{n}^{0}$ to be
the ideal sheaf generated at each point $\mathfrak{b}$ by the subalgebra
$\mathfrak{n}(\mathfrak{b})$. Then we can form the quotient sheaf
$U^{0}/\mathfrak{n}^{0}$. This sheaf inherits the PBW filtration
from $U^{0}$, and it is immediate from the definitions that 
\[
gr(U^{0}/\mathfrak{n}^{0})\tilde{=}p_{*}O(\tilde{\mathfrak{g}})
\]
where we have here used $p:\tilde{\mathfrak{g}}\to\mathcal{B}$ to
denote the projection. 

To give the quantization in its final form, let us recall that to
any filtered sheaf of $k$-algebras $\mathcal{A}$ on a space, we
can associate the Rees algebra, as was done, e.g., in \cite{key-35}.
In particular, $Rees(\mathcal{A})$ is a graded sheaf of $k[h]$-algebras,
such that $Rees(\mathcal{A})/h\tilde{=}gr(\mathcal{A})$ (the associated
graded algebra of $\mathcal{A}$). 

So, we finally make the 
\begin{defn}
The sheaf $\tilde{D}_{h}$ on $\mathcal{B}$ is $Rees(U^{0}/\mathfrak{n}^{0})$. 
\end{defn}
Thus we have that $\tilde{D}_{h}/h\tilde{=}p_{*}(O(\tilde{\mathfrak{g}}))$
by construction. 

We wish to consider the global sections of this object. To that end,
we note that the algebra 
\[
U_{h}(\mathfrak{g}):=Rees(U(\mathfrak{g}))
\]
 (where the PBW filtration is used) maps naturally to $\Gamma(\tilde{D}_{h})$,
simply by following the chain of filtration preserving maps: 
\[
U(\mathfrak{g})\to\Gamma(U^{0})\to\Gamma(U^{0}/\mathfrak{n}^{0})
\]
 Then, we also have a natural map 
\[
O(\mathbb{A}^{1}\times\mathfrak{h}^{*})\tilde{\to}k[h]\otimes U(\mathfrak{h})\to\Gamma(\tilde{D}_{h})
\]
simply by the fact that $\mathfrak{h}\subset\mathfrak{g}$. 

Further, there are embeddings $O(\mathbb{A}^{1}\times\mathfrak{h}^{*}/W)\to U_{h}(\mathfrak{g})$
and $O(\mathbb{A}^{1}\times\mathfrak{h}^{*}/W)\to O(\mathbb{A}^{1}\times\mathfrak{h}^{*})$;
the first as the inclusion of the center, the second as the natural
inclusion. Then, from \cite{key-27}, page 21, we have the 
\begin{claim}
The natural maps $U_{h}(\mathfrak{g})\to\Gamma(\tilde{D}_{h})$ and
$O(\mathbb{A}^{1}\times\mathfrak{h}^{*})\to\Gamma(\tilde{D}_{h})$
agree upon restriction to $O(\mathbb{A}^{1}\times\mathfrak{h}^{*}/W)$. 
\end{claim}
Thus, we in fact have a morphism 
\[
O(\mathbb{A}^{1}\times\mathfrak{h}^{*})\otimes_{O(\mathbb{A}^{1}\times\mathfrak{h}^{*}/W)}U_{h}(\mathfrak{g})\to\Gamma(\tilde{D}_{h})
\]
 which is actually an isomorphism-this is proved in \cite{key-27},
Theorem 5, page 37. Upon taking $h\to0$, we get isomorphism $1$
(c.f. section 2.2.4 below for details about {}``taking $h\to0$'').

\subsubsection{Partial Flag Varieties}

Now we wish to quantize the varieties $\tilde{\mathfrak{g}}_{\mathcal{P}}$,
by a similar explicit strategy. So we start with the sheaf 
\[
U_{\mathcal{P}}^{0}:=U(\mathfrak{g}^{*})\otimes O(\mathcal{P})
\]
with the multiplication as for $U^{0}$ above. For each point $\mathfrak{p}\in\mathcal{P}$,
we have the sub-lie algebra $\mathfrak{u}(\mathfrak{p})\subset\mathfrak{g}$.
We can now define the sheaf of ideals $\mathfrak{u}^{0}$ to be the
sheaf generated at each point $\mathfrak{p}$ by $\mathfrak{u}(\mathfrak{p})$.
Then we have the sheaf $U_{\mathcal{P}}^{0}/\mathfrak{u}^{0}$, and
it is immediate from the definition that 
\[
gr(U_{\mathcal{P}}^{0}/\mathfrak{u}^{0})\tilde{=}p_{*}O(\tilde{\mathfrak{g}}_{\mathcal{P}})
\]
where $p:\tilde{\mathfrak{g}}_{\mathcal{P}}\to\mathcal{P}$ is the
natural projection. So we define 
\[
\tilde{D}_{h,\mathcal{P}}:=Rees(U_{\mathcal{P}}^{0}/\mathfrak{u}^{0})
\]
Of course, we have that $\tilde{D}_{h}=\tilde{D}_{h,\mathcal{B}}$.
It also follows from the definition that in the case $\mathcal{P}=pt$,
$\tilde{D}_{h,\mathcal{P}}=U_{h}(\mathfrak{g})$. 

Next, we can explain the behavior of these sheaves under the natural
pushforward maps. In particular, let $\pi_{s}:\mathcal{B}\to\mathcal{P}_{s}$
be the natural projection morphism (this is a slight abuse of notation
from the previous section). We wish to calculate $\pi_{s*}(\tilde{D}_{h})$,
following \cite{key-10}, \cite{key-2} (the answer will be a deformation
of equation $2$). 

To proceed, let $\mathfrak{p}\in\mathcal{P}_{s}$, and let $\mathfrak{p}^{-}$
be the opposite parabolic, with levi decomposition 
\[
\mathfrak{p}^{-}=\mathfrak{u}(\mathfrak{p}^{-})\oplus\mathfrak{j}^{-}
\]

Under our assumptions, we have that $\mathfrak{j}^{-}\tilde{=}\mathfrak{sl}_{2}\oplus\mathfrak{h}^{s}$. 

Then we have the open subset 
\[
J^{-}\cdot\mathfrak{p}\subset\mathcal{P}
\]
 (and $\mathcal{P}$ is covered by such subsets). Further we have
that 
\[
\pi_{s}^{-1}(J^{-}\cdot\mathfrak{p})\tilde{=}P_{s}/B\times(J^{-}\cdot\mathfrak{p})=\mathbb{P}^{1}\times(J^{-}\cdot\mathfrak{p})
\]
 and the map $\pi_{s}$ becomes the projection to the second factor.

So, the above decompositions imply that we see that 
\[
\pi_{s*}(\tilde{D}_{h})|_{(J^{-1}\cdot\mathfrak{p})}\tilde{=}\Gamma(\tilde{D}_{h}(P/B))\otimes_{\mathbb{C}}O(\mathcal{P})\otimes_{\mathbb{C}}U_{h}[(\mathfrak{h}^{*})^{s}\oplus\mathfrak{u}^{-}(\mathfrak{p})]
\]
 where $\tilde{D}_{h}(P/B)$ denotes $\tilde{D}_{h}$ in the case
of the reductive group $SL_{2}$, with flag variety $\mathbb{P}^{1}$.
But we already know the global sections of this sheaf: 
\[
\Gamma(\tilde{D}_{h}(P/B))\tilde{=}U_{h}(\mathfrak{sl}_{2})\otimes_{O(\mathfrak{t}^{*}/<s>)}O(\mathfrak{t}^{*})
\]
where $\mathfrak{t}$ denotes the Cartan subalgebra for this $\mathfrak{sl}_{2}$,
whose Weyl group is $<s>$. So we see that 
\begin{equation}
\pi_{s*}(\tilde{D}_{h})\tilde{=}\tilde{D}_{h,\mathcal{P}}\otimes_{O(\mathfrak{h}^{*}/<s>)}O(\mathfrak{h}^{*})
\end{equation}
 which becomes equation 2 after letting $h\to0$ (c.f. section 2.2.4
below).

\subsubsection{Equivariance}

We would like to now explain how the $G\times\mathbb{G}_{m}$ action
discussed above can be quantized. We start with the action of $G$
on $\mathcal{P}$, which is of course a map 
\[
a:G\times\mathcal{P}\to\mathcal{P}
\]
 such that for each $g\in G$, $a(g):\mathcal{P}\to\mathcal{P}$ is
an isomorphism, yielding an isomorphism of sheaves $a(g)^{*}:O(\mathcal{P})\to O(\mathcal{P})$.
This collection of isomorphisms satisfies the unit, associativity,
and inverse properties, as with any group action. 

Speaking in loose terms, we would like a $G$-equivariant $\tilde{D}_{h,\mathcal{P}}$-module
to be a $\tilde{D}_{h,\mathcal{P}}$-module $M$ equipped with isomorphisms
\[
a(g)^{*}M\tilde{\to}M
\]
(where this is the quasicoherent pullback), which satisfy these compatibilities,
and which {}``depend algebraically'' on $g\in G$. 

Formally speaking, we shall give the definition of \cite{key-22}.
Firstly, we define 
\[
O_{G}\boxtimes\tilde{D}_{h,\mathcal{P}}:=O_{G\times\mathcal{P}}\otimes_{pr^{-1}O_{\mathcal{P}}}pr^{-1}\tilde{D}_{h,\mathcal{P}}
\]
 where $pr:G\times\mathcal{P}\to\mathcal{P}$ is the second projection.
This is naturally a subsheaf of 
\[
D_{G}\boxtimes\tilde{D}_{h,\mathcal{P}}:=pr_{2}^{-1}(D_{G})\otimes pr^{-1}(\tilde{D}_{h,\mathcal{P}})
\]

Next, let us recall that the maps $a$ and $pr$ induce pullback functors
\[
a^{*},pr^{*}:Mod(\tilde{D}_{h,\mathcal{P}})\to Mod(D_{G}\boxtimes\tilde{D}_{h,\mathcal{P}})
\]
These functors are simply the quasicoherent pullback of sheaves, but
one endows them with the action of vector fields on $G$ by pushforward
of vector fields as usual (c.f. \cite{key-19}, chapter 1). Given
this, we make the 
\begin{defn}
The category of quasi-$G$ equivariant-coherent $\tilde{D}_{h,\mathcal{P}}$-modules,
$Mod^{G}(\tilde{D}_{h,\mathcal{P}})$ has consists of finitely generated
$\tilde{D}_{h,\mathcal{P}}$-modules $M$ equipped with an isomorphism
of $O_{G}\boxtimes\tilde{D}_{h,\mathcal{P}}$-modules 
\[
a^{*}(M)\tilde{\to}pr^{*}(M)
\]
 Further, we demand the usual cocycle compatibility spelled out, e.g.,
in \cite{key-22}. 

The morphisms in this category are those which respect all structures. 
\end{defn}
We note that $\tilde{D}_{h,\mathcal{P}}$ has the structure of a quasi-equivariant
coherent module by the simple computation 
\[
a^{*}(\tilde{D}_{h,\mathcal{P}})\tilde{=}O_{G}\boxtimes\tilde{D}_{h,\mathcal{P}}\tilde{=}pr^{*}(\tilde{D}_{h,\mathcal{P}})
\]
 In addition to the formal definition, it will be extremely useful
for us to use one of the basic properties of equivariant coherent
$D$-modules, following the discussion in \cite{key-22}. Since any
$M\in Mod^{G}(\tilde{D}_{h,\mathcal{P}})$ is a quasi-coherent equivariant
$\mathcal{P}$-module, we have the natural map 
\[
L_{v}:\mathfrak{g}\to End_{k}(M)
\]
 as described above. However, we also have another map 
\[
\alpha:\mathfrak{g}\to End_{k}(M)
\]
 given by using the natural map $\mathfrak{g}\to\Gamma(\tilde{D}_{h,\mathcal{P}})$.
These can be considered as the {}``adjoint action'' and {}``left
action'' of $\mathfrak{g}$. Thus we can define a third action 
\[
\gamma=h\cdot L-\alpha:\mathfrak{g}\to End_{k}(M)
\]
which will in fact commute with the action of $\tilde{D}_{h,\mathcal{P}}$,
i.e., 
\[
\gamma:\mathfrak{g}\to End_{\tilde{D}_{h,\mathcal{P}}}(M)
\]
can be considered a {}``right action'' of $\mathfrak{g}$ (and hence
of $U(\mathfrak{g})$). 

We can extend this definition to define the category of $G\times\mathbb{G}_{m}$-equivariant
coherent modules $Mod^{G\times\mathbb{G}_{m}}(\tilde{D}_{h,\mathcal{P}})$
simply by demanding that the modules be graded, and the action respect
the grading (we note that $\tilde{D}_{h,\mathcal{P}}$ is graded by
virtue of being a Rees algebra, but that we put $h$ in degree $2$).
These categories and their derived versions will be the other major
players in our story. We note that the subcategory of modules in $Mod^{G}(\tilde{D}_{h,\mathcal{P}})$
(resp.$Mod^{G\times\mathbb{G}_{m}}(\tilde{D}_{h,\mathcal{P}})$) where
$h$ acts as zero is precisely the category $Mod^{G}(\tilde{\mathfrak{g}}_{\mathcal{P}})$
(resp. $Mod^{G\times\mathbb{G}_{m}}(\tilde{\mathfrak{g}}_{\mathcal{P}})$). 

To end this subsection, we shall state a result analogous to \prettyref{lem:Coh-generation}
at the end of the previous subsection:
\begin{lem}
The functor $i_{e}^{*}$ gives an equivalence of categories 
\[
Mod^{G}(\tilde{D}_{h})\to Mod^{B}(U_{h}(\mathfrak{b}))
\]
 where the category on the right consists of $U_{h}(\mathfrak{b})$-modules
equipped with an algebraic action of $B$ satisfying the natural compatibilities.
The same is true of the graded versions of these categories. 
\end{lem}
The proof is exactly the same as that of the coherent case; the inverse
is the induction functor. There is also a result concerning generation
of this category, but it shall have to wait until the next section
where we define the natural deformations of the sheaves $O_{\tilde{\mathfrak{g}}}(\lambda)$.

\subsubsection{Cohomology results}

In this subsection we shall gather several results that we need concerning
cohomology of modules in $Mod(\tilde{D}_{h,\mathcal{P}})$. First
of all, there is the following base change result: 
\begin{lem}
\label{lem:Cohomology1}Let $M\in Mod^{\mathbb{G}_{m}}(\tilde{D}_{h,\mathcal{P}})$.
Then we have an isomorphism in the derived category 
\[
R\Gamma(M)\otimes_{k[h]}^{L}k_{0}\tilde{\to}R\Gamma(M\otimes_{k[h]}^{L}k_{0})
\]
where $k_{0}$ denotes the trivial $k[h]$-module. \end{lem}
\begin{proof}
First, note that there is a natural base change map 
\[
R\Gamma(M)\otimes_{k[h]}^{L}k_{0}\to R\Gamma(M\otimes_{k[h]}^{L}k_{0})
\]
which comes from the map of sheaves $M\to M/hM$. We shall spit the
problem of showing this is an isomorphism into two cases.

Let $M_{tors}$ denote the subsheaf of $h$-torsion sections of $M$.
Then we have an exact sequence
\[
0\to M_{tors}\to M\to M/M_{tors}\to0
\]
 we claim that $M/M_{tors}$ is actually a flat $k[h]$-sheaf. To
see this, consider $M/M_{tors}(U)$ where $U$ is affine. Choose any
finite $k[h]$-submodule, $V$. Then by the usual classification of
modules over a PID, $V$ is the direct sum of free and finite-dimensional
components. The existence of a component of the form $k[h]/(h-\lambda)^{n}$
implies that $h$ has eigenvalue $\lambda$ somewhere in $V$. Since
$V$ is $h$-torsion free, we have $\lambda\neq0$. However, since
$h$ acts as a graded operator of degree $2$ on $M$, one sees that
there are no nonzero eigenvalues for $h$ either. So $V$ is a free
$k[h]$ module, and we get that $M/M_{tors}(U)$ is a direct limit
of flat $k[h]$-modules, and hence flat. 

So we must consider two cases- $M$ is $h$-torsion, and $M$ is flat.
Suppose $M$ is flat. Then $R\Gamma(M)$ is equivalent to the Cech
complex for $M$ for a given covering $\{U_{i}\}$, which is thus
a complex of $h$-flat modules. So we have 
\[
R\Gamma(M)\otimes_{k[h]}^{L}k_{0}\tilde{\to}C^{\cdot}(\{U_{i}\},M)\otimes_{k[h]}k\tilde{\to}C^{\cdot}(\{U_{i}),M/hM)
\]
 where the last isomorphism is from the definition of the Cech complex.
This takes care of the flat case. 

For the torsion case, we note that there is a natural finite filtration
of any torsion sheaf by $O(\tilde{\mathfrak{g}})$-modules (i.e. modules
$M$ where $h$ acts trivially). In this case, we have an isomorphism
\[
M\otimes_{k[h]}^{L}k_{0}\tilde{=}M\oplus M[1]
\]

and the same for the global sections $\Gamma(M)$, since $h$ also
acts trivially on them. The result for such sheaves follows immediately,
as does the general result by walking up the filtration.
\end{proof}
Let us note that the same result holds if $M$ is any bounded complex,
since we can reduce to the case where $M$ is concentrated in a single
degree by using cutoff functors and exact triangles. Now we can give
the important 
\begin{cor}
\label{cor:Cohomology2}Let $M\in D^{b,\mathbb{G}_{m}}(\tilde{D}_{h})$.
Suppose that $R\Gamma(M\otimes_{k[h]}^{L}k_{0})$ is concentrated
in degree zero. Then the same is true of $R\Gamma(M)$. \end{cor}
\begin{proof}
$R\Gamma(M)$ is a complex of graded, finitely generated $O(\mathfrak{h}^{*}\times\mathbb{A}^{1})$-modules,
whose reduction modulo an ideal of positively graded elements is concentrated
in a single degree. Given this, the result follows from the graded
Nakayama lemma for complexes- an appropriate version of which is the
next proposition.\end{proof}
\begin{prop}
\label{pro:GrNak}Let $M$ be a bounded complex in the category of
graded $k[h]$ modules ($deg(h)>0$) whose cohomology sheaves all
have grading bounded below. Let $k_{0}$ denote the trivial graded
$k[h]$-module. Then we have 

1) If $M\otimes_{k[h]}^{L}k_{0}=0$, then $M=0$. 

2) If $M\otimes_{k[h]}^{L}k_{0}$ is concentrated in a single degree,
then the same is true of $M$. \end{prop}
\begin{proof}
Since $k[h]$ has global dimension one, any bounded complex $M$ is
quasi-isomorphic as a complex of $k[h]$-modules to the direct sum
of its appropriately shifted cohomology sheaves. Since our complex
consists of graded modules and graded morphisms, the cohomology sheaves
are graded as well. We write 
\[
M\tilde{=}\oplus H^{i}(M)[-i]
\]

Then the complex $M\otimes_{k[h]}^{L}k_{0}$ is quasi-isomorphic to
\[
\oplus(H^{i}(M)/h)[-i]\oplus Tor_{k[h]}^{1}(k_{0},H^{i}(M))[-i+1]
\]
 simply because, for any object $N\in k[h]-mod$ 
\[
N\otimes_{k[h]}^{L}k_{0}\tilde{=}(N/h)\oplus Tor_{k[h]}^{1}(k_{0},N)[1]
\]
 (a quasi-isomorphism of complexes). Now, if we are in the situation
of $1$, the assumption implies $H^{i}(M)/h=0$ for all $i$, and
so the graded Nakayama lemma for modules yields $H^{i}(M)=0$ for
all $i$, hence $M$ is equivalent to the trivial complex. 

Next, suppose we are in the situation of $2$, and shift so that $M\otimes_{k[h]}^{L}k_{0}$
is concentrated in degree zero. Then $H^{i}(M)/h$ must be trivial
for all $i\neq0$, so the same is true for $H^{i}(M)$ as required. 
\end{proof}
We shall also have occasion to consider the functor $R\Gamma^{G}$
of $G$-invariant cohomology. The results above go through unchanged
in this setting, as is easy to see by the fact that the functor $V\to V^{G}$
is exact on the category of algebraic $G$-modules. 

We should like to end the section with some general remarks about
the significance of these results for us. The three main {}``coherent''
categories that appear in this work are $D^{b,G\times\mathbb{G}_{m}}(\tilde{\mathcal{N}})$,
$D^{b,G\times\mathbb{G}_{m}}(\tilde{\mathfrak{g}})$, and $D^{b}(Mod^{G\times\mathbb{G}_{m}}(\tilde{D}_{h}))$.
The above results will be used to show that, for objects $M,N$ in
a certain tilting subcategory $\mathcal{T}$ of $D^{b}(Mod^{G\times\mathbb{G}_{m}}(\tilde{D}_{h}))$,
we have 
\[
Hom_{D^{b}(Mod^{G\times\mathbb{G}_{m}}(\tilde{D}_{h}))}(M,N)\otimes_{k[h]}^{L}k_{0}\tilde{=}Hom_{D^{b,G\times\mathbb{G}_{m}}(\tilde{\mathfrak{g}})}(M\otimes_{k[h]}^{L}k_{0},N\otimes_{k[h]}^{L}k_{0})
\]
 making precise the notion that $D^{b}(Mod^{G\times\mathbb{G}_{m}}(\tilde{D}_{h}))$
can be considered a one-parameter deformation of $D^{b,G\times\mathbb{G}_{m}}(\tilde{\mathfrak{g}})$. 

On the other hand, we have the identification $\tilde{\mathcal{N}}\tilde{=}\tilde{\mathfrak{g}}\times_{\mathfrak{h}}\{0\}$;
where $\tilde{\mathfrak{g}}\to\mathfrak{h}$ is Grothendieck's morphism
described above, coming from the maps 
\[
O(\mathfrak{h}^{*})\tilde{\to}O(\tilde{\mathfrak{g}})^{G}\to O(\tilde{\mathfrak{g}})
\]
which follows from the well known fact that $\mathcal{N}\tilde{=}\mathfrak{g}^{*}\times_{\mathfrak{h}^{*}/W}\{0\}$
(c.f. \cite{key-13} chapter 3). As is also well known, the morphism
$\tilde{\mathfrak{g}}\to\mathfrak{h}^{*}$ is flat (c.f. \cite{key-27}).
Therefore the flat base change theorem implies that for $M,N\in D^{b}(Coh(\tilde{\mathfrak{g}}))$
there is an isomorphism 
\[
Hom_{D^{b}(Coh(\tilde{\mathfrak{g}}))}(M,N)\otimes_{O(\mathfrak{h}^{*})}^{L}k_{0}\tilde{=}Hom_{D^{b}(Coh(\tilde{\mathcal{N}}))}(Li^{*}M,Li^{*}N)
\]
 where $i:\tilde{\mathcal{N}}\to\tilde{\mathfrak{g}}$ denotes the
inclusion. Thus $D^{b}(Coh(\tilde{\mathfrak{g}}))$ (and its equivariant
and graded versions) can be considered a $dim(\mathfrak{h})$-parameter
deformation of $D^{b}(Coh(\tilde{\mathcal{N}}))$. The exact same
set-up holds for the three main categories of perverse sheaves under
consideration, and this will turn out to be a key point in proving
the main equivalences.

\section{Structure of Coherent Categories}

In this chapter we discuss two interrelated and crucial pieces of
structure: the braid group action and tilting generation of the categories
defined in the previous chapter.

\subsection{Braid Group Action}

In this subsection we recall the main results of the papers \cite{key-28}
and \cite{key-11}. 

First of all, let us recall that for any Coxeter system $(W,S)$,
there is associated the braid group $\mathbb{B}(W,S)$, which is the
group on generators $S$ satisfying only the braid relations 
\[
s_{i}s_{j}s_{i}\cdot\cdot\cdot=s_{j}s_{i}s_{j}\cdot\cdot\cdot
\]
where the number of factors on each side is the $(i,j)$ entry in
the associated Coxeter matrix. 

The case of interest to us, as usual, is the case of the affine Weyl
group $W_{aff}$. Of course, there is also the isomorphism $W_{aff}\tilde{=}W\rtimes\mathbb{Z}\Phi$,
and, as in the case of the affine Hecke algebra, there is a presentation
of the affine braid group based upon this isomorphism, which is actually
slightly more general (c.f. the appendix to \cite{key-28}). So we
make the
\begin{defn}
Let $(W,S_{fin})$ be a finite Weyl group, with its root lattice $\mathbb{Z}\Phi$
and its weight lattice $\mathbb{X}$. The extended affine braid group
$\mathbb{B}_{aff}^{'}$ is the group with generators $\{T_{s}\}_{s\in S_{fin}}$,
and $\{\theta_{x}\}_{x\in\mathbb{X}}$, and relations: 

$\cdot$$T_{s_{i}}T_{s_{j}}T_{s_{i}}\cdot\cdot\cdot=T_{s_{j}}T_{s_{i}}T_{s_{j}}\cdot\cdot\cdot$
(the finite braid relations) 

$\cdot$ $\theta_{x}\theta_{y}=\theta_{x+y}$ for all $x,y\in\mathbb{X}$. 

$\cdot$ $T_{s}\theta_{x}=\theta_{x}T_{s}$ whenever $<x,\check{\alpha}_{s}>=0$. 

$\cdot$ $\theta_{x}=T_{s}\theta_{s(x)}T_{s}$ whenever $<x,\check{\alpha}_{s}>=1$. 
\end{defn}
If we replace the lattice $\mathbb{X}$ with the root lattice $\mathbb{\mathbb{Z}\Phi}$
in this presentation, then the resulting group is just isomorphic
to the usual affine braid group $\mathbb{B}_{aff}$; in particular
$\mathbb{B}_{aff}$ is a subgroup of finite index in $\mathbb{B}_{aff}^{'}$.
We also note that the extended affine Weyl group $W_{aff}^{'}$ is
then the quotient of $\mathbb{B}_{aff}^{'}$ by the relations
\[
s_{i}^{2}=1
\]
 for all $s_{i}$. The group $W_{aff}$ is the analogous quotient
for $\mathbb{B}_{aff}$. 

This presentation is useful for explaining how $\mathbb{B}_{aff}^{'}$
will act on categories of coherent sheaves. To set this up, we shall
briefly recall the notion of a Fourier-Mukai functor. 

Let $X$ and $Y$ be algebraic varieties, and let $\mathcal{F}\in D^{b}Coh(Y\times X)$
be a complex of coherent sheaves whose support is proper over both
$X$ and $Y$. Then we have a well defined functor on $D^{b}Coh(X)\to D^{b}Coh(Y)$
\[
F_{\mathcal{F}}(M)=Rp_{2*}(\mathcal{F}\otimes_{O_{Y\times X}}^{L}Lp_{1}^{*}(M))
\]
The sheaf $\mathcal{F}$ is called the kernel of this functor. For
example, the diagonal sheaf $O_{\Delta X}$ corresponds to the identity
functor, while for a proper morphism $f$ the standard (derived) functors
$f^{*}$ and $f_{*}$ can be realized via the sheaf of functions on
the graph of $f$ (c.f. \cite{key-20}, page 114). 

Further, we shall need the fact that the composition of functors corresponding
to two kernels $\mathcal{F}$ and $\mathcal{G}$ can be realized as
the {}``composition'' of the kernels, as follows: suppose $\mathcal{F}\in D^{b}(Coh(X\times Y))$
and $\mathcal{G}\in D^{b}(Coh(Y\times Z))$. We define 
\[
\mathcal{F}\star\mathcal{G}:=p_{XZ_{*}}(p_{XY}^{*}(\mathcal{F})\otimes p_{YZ}^{*}(\mathcal{G}))\in D^{b}(Coh(X\times Z))
\]
 (all supports assumed proper, all functors derived). Then we have
the 
\begin{prop}
\label{pro:composition=00003Dconvolution}There is an isomorphism
$F_{\mathcal{G}}\circ F_{\mathcal{F}}\tilde{=}F_{\mathcal{F}\star\mathcal{G}}$
of functors $D^{b}(Coh(X))\to D^{b}(Coh(Z))$. 
\end{prop}
This is proved in \cite{key-20}, page 114. 

The braid group action we shall present is given by Fourier-Mukai
kernels. To explain the sorts of varieties we shall need, let us recall
that $\tilde{\mathfrak{g}}$ has an open subvariety $\tilde{\mathfrak{g}}^{rs}$
(defined above!) which has a natural action of the finite Weyl group
$W$. For $s_{\alpha}\in S_{fin}$, we define the variety $S_{\alpha}\subset\tilde{\mathfrak{g}}\times\tilde{\mathfrak{g}}$
to be the closure of the graph of $s_{\alpha}$ in $\tilde{\mathfrak{g}}^{rs}\times\tilde{\mathfrak{g}}^{rs}$.
We further define $S_{\alpha}^{'}$ to be the variety $S_{\alpha}\cap(\tilde{\mathcal{N}}\times\tilde{\mathcal{N}})$
(c.f. \cite{key-28}, section 4).

Then, we have 
\begin{thm}
There is an action%
\footnote{By an action in this case we mean a weak action. We shall discuss
an extension of this to a stronger structure later in the paper.%
} of the group $\mathbb{B}_{aff}^{'}$ on the category $D^{b}Coh(\tilde{\mathfrak{g}})$
which is specified by 

$\cdot$ The action of $\theta_{\lambda}$ is given by the kernel
$\Delta_{*}(O_{\tilde{\mathfrak{g}}}(\lambda))$ (where $\Delta$
is the diagonal inclusion) 

$\cdot$ The action of $s_{\alpha}\in S_{fin}$ is given by the kernel
$O_{S_{\alpha}}$. 

This action restricts to an action on the category $D^{b}Coh(\tilde{\mathcal{N}})$,
in the sense that if we define kernels $\Delta_{*}(O_{\mathcal{\tilde{N}}}(\lambda))$,
and $O_{S_{\alpha}^{'}}$, we get an action of $\mathbb{B}_{aff}^{'}$
on $D^{b}Coh(\tilde{\mathcal{N}})$ which agrees with the previous
action under the inclusion functor $i_{*}$. 

Further, these same kernels also define braid group actions on the
equivariant categories $D^{b,G}(\tilde{\mathfrak{g}})$, $D^{b,G\times\mathbb{G}_{m}}(\tilde{\mathfrak{g}})$;
and the same for $\tilde{\mathcal{N}}$. 
\end{thm}
We shall show later on that this action also extends to an action
on the category $D^{b}(Mod(\tilde{D}_{h}))$. 

We note, for later use, the following: 
\begin{claim}
\label{cla:W-fin-is-trivial}For each finite root $s_{\alpha}$, we
have $s_{\alpha}\cdot O_{\tilde{\mathfrak{g}}}\tilde{=}O_{\tilde{\mathfrak{g}}}$. 
\end{claim}
This is proved in \cite{key-11,key-9}.

\subsection{Highest Weight Structure}

In this section, we discuss a crucial piece of structure on the categories
$D^{b}Coh^{G}(\tilde{\mathcal{N}})$ and $D^{b}Coh^{G\times\mathbb{G}_{m}}(\tilde{\mathcal{N}})$,
which appears in the paper \cite{key-6}. In that paper the author
defines a t-structure, known as the perversely exotic t-structure,
which corresponds under the equivalence of \cite{key-1} to the perverse
t-structure on the category $D_{I^{u},\chi}^{b}(G((t))/I)$ (see \cite{key-9,key-6}
for a proof of this fact). 

A nice feature of this t-structure is that it can be defined in relatively
elementary terms, using the braid group action. To give the statement,
we shall have to recall some general facts, starting with the 
\begin{defn}
Let $\mathcal{D}$ be a triangulated category, linear over a field
$k$. Suppose that $\mathcal{D}$ is of finite type, so that $Hom^{\cdot}(N,M)$
is always finite dimensional over $k$. Let $\nabla=\{\nabla^{i}|i\in I\}$
be an ordered set of objects in $\mathcal{D}$, which generate $\mathcal{D}$
as a triangulated category. This set is called exceptional if $Hom^{\cdot}(\nabla^{i},\nabla^{j})=0$
whenever $i<j$, and if $Hom^{\cdot}(\nabla^{i})=k$ for all $i$. 
\end{defn}
The classic example of such a set is the collection of Verma modules
in the principal block of the BGG category $\mathcal{O}$. 

Of course, part of the advantage of Verma modules is that there are
dual Verma modules, and natural maps $M(\lambda)\to M^{*}(\lambda)$,
whose image is the irreducible module $L(\lambda)$. It turns out
that there is a general version of this fact as well. So let us make
the 
\begin{defn}
Let $\mathcal{D}$ be a triangulated category with a given exceptional
set $\nabla$. We let $\mathcal{D}_{<i}$ denote the triangulated
subcategory generated by $\{\nabla^{j}|j<i\}$. Then another set of
objects $\{\Delta_{i}|i\in I\}$ is called a dual exceptional set
if it satisfies $Hom^{\cdot}(\Delta_{n},\nabla^{i})=0$ for $n>i$,
and if we have isomorphisms 
\[
\Delta_{i}\tilde{=}\nabla^{i}\, mod\,\mathcal{D}_{<i}
\]
for all $i$. 
\end{defn}
Whenever the dual exceptional set exists, it is unique (c.f. \cite{key-6}).

\subsubsection{Existence of a $t$-structure}

Now we would like to recall, from the paper \cite{key-7}, the existence
of a $t$-structure which is compatible with the standard and costandard
objects. In particular, under the above assumptions, with the additional
assumption that the order set $I$ is either finite or (a finite union
of copies of) $\mathbb{Z}_{>0}$, we have the following
\begin{thm}
a) There exists a unique $t$-structure, $(\mathcal{D}^{\geq0},\mathcal{D}^{<0})$,
which satisfies $\nabla^{i}\in\mathcal{D}^{\geq0}$ and $\Delta_{i}\in\mathcal{D}^{\leq0}$
for all $i$. 

b) This $t$-structure is bounded.

c) For any $X\in Ob(\mathcal{D})$, we have that $X\in\mathcal{D}^{\geq0}$
iff $Hom^{<0}(\Delta_{i},X)=0$ for all $i$, and similarly, $X\in\mathcal{D}^{<0}$
iff $Hom^{\leq0}(X,\nabla^{i})=0$. 

d) We let $\mathcal{A}=\mathcal{D}^{\geq0}\cap\mathcal{D}^{\leq0}$
denote the heart of this $t$-structure. Every object of $\mathcal{A}$
has finite length. For each $i$ there is a canonical arrow 
\[
\tau_{\geq0}(\Delta_{i})\to\tau_{\geq0}(\nabla^{i})
\]
 whose image is an irreducible object of $\mathcal{A}$, called $L_{i}$.
The set $\{L_{i}\}_{i\in I}$ is a complete, pairwise non-isomorphic
set of irreducibles in $\mathcal{A}$. 
\end{thm}
Let us remark that by part c, if our exceptional collection satisfies
$Hom^{<0}(\nabla^{i},\nabla^{j})=0$ and $Hom^{<0}(\Delta_{i},\Delta_{j})=0$
for all $i$ and $j$, then in fact the collections $\nabla$ and
$\Delta$ are actually in the heart $\mathcal{A}$. This will be the
case in all of the examples we consider. 

We should also note that the theorem holds in the slightly modified
situation of a graded triangulated category. In particular, we suppose
that $\mathcal{D}$ is equipped with a triangulated autoequivalence
$M\to M(1)$, which is the {}``shift in grading.'' In this instance,
we can define the graded hom $Hom_{gr}^{i}(X,Y)=\oplus_{n}Hom^{i}(X,Y(n))$. 

In this case, we say that a collection of objects $\{X_{i}\}$ generates
$\mathcal{D}$ if $\{X_{i}(n)\}_{n\in\mathbb{Z}}$ generates $\mathcal{D}$.
Then a graded exceptional set is defined as above but using $Hom_{gr}$
instead of $Hom$, and using this looser sense of {}``generate''.
Then there is a graded analogue of theorem 9 where one replaces all
instances of $Hom$ with $Hom_{gr}$. See \cite{key-6} for details.

\subsubsection{Perversely Exotic $t$-structure. }

Now we are ready to describe the perversely exotic $t$-structures
on $D^{b}(Coh^{G}(\tilde{\mathcal{N}}))$ and on $D^{b}(Coh^{G\times\mathbb{G}_{m}}(\tilde{\mathcal{N}}))$.
In fact, we shall, following \cite{key-6,key-9}, write down the exceptional
and coexceptional sets explicitly. 

Our indexing set $I$ will be the character lattice $\mathbb{X}$.
We first consider $\mathbb{X}$ with the Bruhat partial ordering-
this is the partial ordering induced from considering $\mathbb{X}$
as a subset of the affine extended Weyl group $W_{aff}^{'}$. 

More explicitly, we can define the order as follows: let $\lambda$
and $\nu$ be two elements of $\mathbb{X}$. Choose $w(\lambda)$
and $w(\nu)$ in the group $W_{fin}$ so that $w(\lambda)\cdot\lambda$
sits in the dominant cone, and the same for $\nu$. Then $\lambda\leq\nu$
iff $w(\lambda)\cdot\lambda$ is below $w(\nu)\cdot\nu$ in the usual
dominance ordering. 

We note from this description that there are finitely many elements
which are absolute minima under this ordering; these are precisely
the set of minimal representatives in $\mathbb{X}$ of the finite
group $\mathbb{X}/\mathbb{Z}\Phi=\Omega$. So, we complete $\leq$
to a complete ordering on $\mathbb{X}$, which we choose to be isomorphic
to a finite union of copies of $\mathbb{Z}_{>0}$%
\footnote{In \cite{key-6}, he works with the adjoint group, and so assumes
that the ordering is isomorphic to a single copy of $\mathbb{Z}_{>0}$.
However, the results we need go over to our case without any difficulty.%
}. 

Now we can define our exceptional and coexceptional sets as follows:
we let $\mathbb{B}_{aff}^{+}$ denote the subsemigroup of the affine
braid group generated by $\{s_{\alpha}\}_{\alpha\in I_{aff}}$, and
$\mathbb{B}_{aff}^{-}$ the subsemigroup generated by the inverses.
Then our exceptional set is the collection of $\{b^{-}\cdot\omega O_{\tilde{\mathcal{N}}}\}_{\omega\in\Omega}$
and our coexceptional set is the collection of $\{b^{+}\cdot\omega O_{\tilde{\mathcal{N}}}\}_{\omega\in\Omega}$.
Then, one can in fact show c.f. \cite{key-6,key-9}, that these sets
are indexed by $\mathbb{X}$, by sending an element $b^{+}\omega$
its action on $0\in\mathbb{X}$; and thus we can also label them $\{\Delta_{\lambda}\}$
and $\{\nabla^{\lambda}\}$ for $\lambda\in\mathbb{X}$.

These indexing sets have several nice properties- at the bottom of
the ordering, the objects $\{\omega O_{\tilde{\mathcal{N}}}\}$ are
both standard and costandard. In addition, for any $\lambda$ which
is in the dominant cone of $\mathbb{X}$, we can choose the representative
$\theta_{\lambda}\in\mathbb{B}_{aff}^{'}$ as an element of the form
$b^{+}\omega$. We know from the explicit presentation of the braid
group action given above that the action of this element is given
by tensoring by the line bundle $O_{\tilde{N}}(\lambda)$. Thus we
have that the set of dominant coexceptional objects is $\{O_{\tilde{\mathcal{N}}}(\lambda)\}_{\lambda\in\mathbb{Y}^{+}}$,
and the exceptionals are $\{O_{\tilde{\mathcal{N}}}(-\lambda)\}_{\lambda\in\mathbb{Y}^{+}}$. 

So, we now can define the perversely exotic $t$-structure to be the
$t$-structure provided by the above theorem on $D^{b}(Coh^{G}(\tilde{\mathcal{N}}))$
and $D^{b}(Coh^{G\times\mathbb{G}_{m}}(\tilde{\mathcal{N}}))$ (using
the graded version for the latter). 

We would like to record one very important feature of this $t$-structure
right now. We recall from \cite{key-9} the
\begin{defn}
A $t$-structure on one of the categories $D^{b}(Coh(\tilde{\mathfrak{g}}))$,
$D^{b}(Coh(\tilde{\mathcal{N}}))$ (or one of the equivariant versions)
is said to be braid positive if for any affine root $\alpha$, the
functor $s_{\alpha}^{-1}\cdot$ is left exact with respect to this
$t$-structure. Of course, by adjointness, this implies immediately
that $s_{\alpha}\cdot$ is right exact. 
\end{defn}
Then we have the very easy 
\begin{lem}
The perversely exotic $t$-structure is braid positive. \end{lem}
\begin{proof}
By the definition of the $t$-structure and part c of the theorem,
we have that $X\in\mathcal{D}^{\geq0}$ iff $Hom^{<0}(b^{+}\omega\cdot O_{\tilde{\mathcal{N}}},X)=0$
for all $b^{+}\in\mathbb{B}_{aff}^{+}$. But then we have by adjointness
\[
Hom^{<0}(b^{+}\omega\cdot O_{\tilde{\mathcal{N}}},s_{\alpha}^{-1}\cdot X)=Hom^{<0}(s_{\alpha}b^{+}\omega\cdot O_{\tilde{\mathcal{N}}},X)
\]
 and the term on the right vanishes because $s_{\alpha}$ is positive,
so $s_{\alpha}b^{+}$ is a positive element of the braid group also. 
\end{proof}
In fact, a similar argument shows something a bit stronger: 
\begin{lem}
\label{lem:filteredobjects}Suppose that $X\in\mathcal{A}$ is filtered
(in $\mathcal{A}$) by standard objects $\nabla^{\lambda}$. Then
for all simple affine roots, $s_{\alpha}^{-1}\cdot X$ is in $\mathcal{A}$.
Similarly, if $X$ is filtered by costandard objects, then $s_{\alpha}\cdot X$
is in $\mathcal{A}$. \end{lem}
\begin{proof}
By definition, $X\in\mathcal{A}$ iff $X\in\mathcal{D}^{\geq0}$ and
$X\in\mathcal{D}^{\leq0}$. So, to show the first claim, we must show
that, under the assumptions, we have 
\[
Hom^{<0}(b^{+}\omega\cdot O_{\tilde{\mathcal{N}}},s_{\alpha}^{-1}\cdot X)=0=Hom^{<0}(s_{\alpha}^{-1}X,b^{-}\omega\cdot O_{\tilde{\mathcal{N}}})=0
\]
for all $b^{+}\in\mathbb{B}_{aff}^{+}$. The first equality holds
simply because $X\in\mathcal{A}$. For the second, we shall walk up
a standard filtration of $X$: if $X$ is itself standard, then by
definition $s_{\alpha}^{-1}X$ is also standard, and hence in $\mathcal{A}$. 

So suppose that we have the exact sequence 
\[
0\to Y\to X\to b_{1}^{-}\omega\cdot O_{\tilde{\mathcal{N}}}\to0
\]
 in $\mathcal{A}$, where $Y$ has a filtration by standard objects
of length $n-1$. Hitting this sequence with $s_{\alpha}^{-1}$ gives
the triangle 
\[
s_{\alpha}^{-1}Y\to s_{\alpha}^{-1}X\to s_{\alpha}^{-1}b_{1}^{-}\omega\cdot O_{\tilde{\mathcal{N}}}
\]
 whose left and right terms are in $\mathcal{A}$, by induction. Then,
by the long exact sequence for $Hom$, we have for all $b^{-}\omega$
the sequence 
\[
Hom^{-i}(s_{\alpha}^{-1}b_{1}^{-}\omega O_{\tilde{\mathcal{N}}},b^{-}\omega\cdot O_{\tilde{\mathcal{N}}})\to Hom^{-i}(s_{\alpha}^{-1}X,b^{-}\omega\cdot O_{\tilde{\mathcal{N}}})\to Hom^{-i}(s_{\alpha}^{-1}Y,b^{-}\omega\cdot O_{\tilde{\mathcal{N}}})
\]
 and the left and right terms are zero for $i>0$; so the middle one
is as well, proving the first claim. The second claim follows in exactly
the same way.\end{proof}
\begin{rem}
\label{rem:filteredobjects}The proof actually shows that the object
$s_{\alpha}^{-1}X$ is filtered by standard objects in $\mathcal{A}$:
since the exact triangles 
\[
s_{\alpha}^{-1}Y\to s_{\alpha}^{-1}X\to s_{\alpha}^{-1}b_{1}^{-}\omega\cdot O_{\tilde{\mathcal{N}}}
\]
 are actually exact sequences in $\mathcal{A}$, this is shown by
the same inductive argument. Clearly the analogous fact is true for
$s_{\alpha}X$ if $X$ is filtered by costandard objects. 
\end{rem}

\subsection{Reflection Functors}

In this subsection we define the reflection functors- they will come
naturally out of the braid group action, and will allow us to construct
explicitly the tilting objects of our category. We shall also see
in the next section that they are the key to lifting the braid group
action from coherent sheaves to $\tilde{D}_{h}$-modules. 

To motivate the definition, we need to recall a bit of geometry from
the paper \cite{key-28}. Recall that we have defined above the kernel
$O_{S_{\alpha}}$ to be the structure sheaf of the variety $S_{\alpha}$,
which in turn is defined as the closure of the graph of the of action
of the Weyl group element $s_{\alpha}$ acting on $\tilde{\mathfrak{g}}^{rs}$.
Let us recall also that we have defined varieties $\tilde{\mathfrak{g}}_{\mathcal{P}}$
associated to any partial flag variety $\mathcal{P}$. In the case
$\mathcal{P}=G/P_{\alpha}$, we shall denote this variety $\tilde{\mathfrak{g}}_{\alpha}$,
and the natural map $\pi_{\alpha}:\tilde{\mathfrak{g}}\to\tilde{\mathfrak{g}}_{\alpha}$. 

Now, we let us consider the algebraic variety $\tilde{\mathfrak{g}}\times_{\tilde{\mathfrak{g}}_{\alpha}}\tilde{\mathfrak{g}}$.
This is not an irreducible variety, but instead has two components:
the first is the diagonal $\Delta\tilde{\mathfrak{g}}$, and the second
is $S_{\alpha}$. The natural restriction morphism leads to a short
exact sequence of kernels: 
\[
\mathcal{K}^{\cdot}\to O_{\tilde{\mathfrak{g}}\times_{\tilde{\mathfrak{g}}_{\alpha}}\tilde{\mathfrak{g}}}\to O_{\Delta\tilde{\mathfrak{g}}}
\]
 and it is checked in \cite{key-28} that $\mathcal{K}^{\cdot}$ is
the kernel of functor inverse to $s_{\alpha}$. Further, when we consider
the action on $D^{b,\mathbb{G}_{m}}(Coh(\tilde{\mathfrak{g}}))$,
we get that $\mathcal{K}^{\cdot}(2)$ is inverse to $s_{\alpha}$. 

So we should like to understand the kernel $O_{\tilde{\mathfrak{g}}\times_{\mathfrak{\tilde{g}}_{\alpha}}\tilde{\mathfrak{g}}}$.
Fortunately, it is easy to describe, following \cite{key-28}: 
\begin{lem}
There is an isomorphism of functors 
\[
F_{O_{\tilde{\mathfrak{g}}\times_{\mathfrak{\tilde{g}}_{\alpha}}\tilde{\mathfrak{g}}}}\tilde{=}\pi_{s}^{*}\pi_{s*}
\]
 where the functor on the right is taken in the derived sense. Further,
the natural adjunction $\pi_{s}^{*}\pi_{s*}\to Id$ comes from the
natural map of sheaves $O_{\tilde{\mathfrak{g}}\times_{\mathfrak{\tilde{g}}_{\alpha}}\tilde{\mathfrak{g}}}\to O_{\Delta\tilde{\mathfrak{g}}}$
(the restriction to a subvariety quotient map). 
\end{lem}
These facts lead us to the following 
\begin{defn}
For a finite root $\alpha$, we define the reflection functor $\mathcal{R}_{\alpha}$
to be the functor of the kernel $O_{\tilde{\mathfrak{g}}\times_{\mathfrak{\tilde{g}}_{\alpha}}\times\tilde{\mathfrak{g}}}$. 
\end{defn}
These functors have many nice properties. As already noted, there
is a natural complex of functors $s_{\alpha}^{-1}(-2)\to\mathcal{R}_{\alpha}\to Id$.
In fact, there is also a natural adjunction morphism $Id\to\mathcal{R}_{\alpha}(2)$
defined in \cite{key-28}, section 5, and an exact sequence $Id\to\mathcal{R}_{\alpha}(2)\to s_{\alpha}$.
Thus it is possible to describe completely the finite braid actions
via the reflection functors. 

We should note that the adjunction morphism $Id\to\pi_{s}^{*}\pi_{s*}(2)$
has a natural algebro-geometric explanation. We have an isomorphism
$\pi_{s}^{!}\tilde{=}\pi_{s}^{*}$ (noted in \cite{key-10}), and
in fact $\pi_{s}^{!}$ is the right adjoint to $\pi_{s*}$ (c.f. \cite{key-18}).
In this instance it has the advantage of having been constructed by
hand in terms of Fourier-Mukai kernels.

Next we would like to define the reflection functor corresponding
to the affine root. Of course, there is no {}``affine root'' partial
flag variety, so we have to use a trick to get around it. The trick
relies on the following
\begin{claim}
In the extended affine braid group $\mathbb{B}_{aff}^{'}$, there
exists a finite root element $s_{\alpha}$ which is conjugate to the
affine root $s_{\alpha_{0}}$. 
\end{claim}
This claim is proved in \cite{key-9}, lemma 2.1.1. It is interesting
to note that in every type except $C$, the claim is true in the non-extended
affine braid group $\mathbb{B}_{aff}$. 

This is helpful for the following reason. If we take the exact sequence
of functors 
\[
Id\to\mathcal{R}_{\alpha}\to s_{\alpha}
\]
 and conjugate by an appropriate element in $\mathbb{B}_{aff}^{'}$,
$b$, we then arrive at a new sequence 
\[
Id\to b^{-1}\mathcal{R}_{\alpha}b\to s_{\alpha_{0}}
\]
and the same holds for the exact sequence for $s_{\alpha}^{-1}$.
Thus if we define the affine reflection functor 
\[
\mathcal{R}_{\alpha_{0}}:=b^{-1}\mathcal{R}_{\alpha}b
\]
then this is a functor which satisfies the same exact sequences for
$\alpha_{0}$ as the other reflection functors for their roots. This
functor will then do everything we need. Further, it will turn out
that the action of this functor is unique up to a unique isomorphism.

\subsection{Tilting Generators For Coherent Sheaves}

In this section, we shall describe collections of tilting modules
for the categories of our interest. We gave a general definition of
a tilting subcategory above, which we shall use for $D^{b}Coh^{G}(\tilde{\mathfrak{g}})$
and $D^{b}Mod^{G}(\tilde{D}_{h})$. However, in the case of $D^{b}Coh^{G}(\tilde{\mathcal{N}})$,
there is a way of constructing tilting modules just using the highest
weight structure. We shall take this to be our base case.

\subsubsection{Tilting in Highest Weight Categories}

In this section, we shall recall the general constructions and definitions
of \cite{key-33,key-4,key-6}. Let us suppose we are in the situation
of section 3.2, where we have a triangulated category with a $t-$structure,
whose heart $\mathcal{A}$ contains a given set of exceptional and
coexceptional objects. In this very special situation, we make the
\begin{defn}
A tilting object in $\mathcal{A}$ is one which possesses a filtration
(in $\mathcal{A}$) by standard objects, and a filtration (in $\mathcal{A}$)
by costandard objects. 
\end{defn}
This, as it turns out, is a very strong condition. We recall from
\cite{key-33,key-4,key-6,key-7} some properties that these objects
satisfy: 
\begin{lem}
In the above situation, we have that: 

1) To each $i\in I$, there is a unique indecomposable tilting module
$T_{i}$, which has a unique (up to scalar) surjection $T_{i}\to\nabla^{i}$
and a unique (up to scalar) injection $\Delta_{i}\to T_{i}$. 

2) Every tilting module is a direct sum of the $T_{i}$. 

3) For any two tilting modules $Hom^{\cdot}(T_{i},T_{j})=Hom^{0}(T_{i},T_{j})$. 

4) The tilting modules generate the standard and costandard objects;
thus they generate the entire triangulated category. Combining with
the above observation, we obtain an equivalence of categories 
\[
K^{b}(\mathcal{T})\to\mathcal{D}
\]
 where the left hand side is the homotopy category of complexes of
tilting modules (c.f. section 1.5). 

5) Let $X$ be any object of $\mathcal{D}$ such that $Hom^{>0}(\Delta_{i},X)=Hom^{>0}(X,\nabla^{i})$
for all $i$. Then $X$ lies in $\mathcal{A}$ and is a tilting object
therein.
\end{lem}
Thus, our goal is now to compare this abstract characterization of
tilting (which holds for the perversely exotic $t$-structure on $D^{b}Coh^{G}(\tilde{\mathcal{N}})$
and its graded version) with the concrete information listed above
about our categories. The main tool in this will be the reflection
functors.

\subsubsection{Tilting via Reflection Functors}

In this subsection, we will use the reflection functors to construct
tilting modules in our various categories. There is a minor problem:
the reflection functors only act on $Coh(\tilde{\mathfrak{g}})$,
and do not restrict to functors on $Coh(\tilde{\mathcal{N}})$- so
we have to construct our objects over $\tilde{\mathfrak{g}}$, and
then restrict. 

To start with, we should define natural lifts of our standard and
costandard objects to $\tilde{\mathfrak{g}}$. This is easy- since
the braid group action on the variety $\tilde{\mathfrak{g}}$ is consistent
via the restriction functor with the braid group action on $\tilde{\mathcal{N}}$,
we define 
\[
\Delta_{\lambda}(\tilde{\mathfrak{g}})=b^{+}\omega\cdot O_{\tilde{\mathfrak{g}}}
\]
 and 
\[
\nabla^{\lambda}(\tilde{\mathfrak{g}})=b^{-}\omega\cdot O_{\mathfrak{\tilde{g}}}
\]
where we have that $\lambda=b^{+}\omega\cdot0$. 
\begin{claim}
These are well defined objects which restrict to our given standard
and costandard objects on $\tilde{\mathcal{N}}$. \end{claim}
\begin{proof}
Evidently these objects restrict to our given standard and costandard
objects on $\tilde{\mathcal{N}}$; we shall now argue that they are
well defined. Suppose $\lambda\in\mathbb{Y}$ has two decompositions
$b_{1}^{+}\omega_{1}\cdot0=\lambda=b_{2}^{+}\omega_{2}\cdot0$, we
wish to show $b_{1}^{+}\omega_{1}\cdot O_{\tilde{\mathfrak{g}}}\tilde{=}b_{2}^{+}\omega_{2}\cdot O_{\tilde{\mathfrak{g}}}$
(the argument for the standard objects will work the same way). Then,
by the well-definedness on $\tilde{\mathcal{N}}$, we have a $G\times\mathbb{G}_{m}$-isomorphism
\[
O_{\tilde{\mathcal{N}}}\mbox{ }\tilde{\to}\mbox{ }\omega_{1}^{-1}b_{1}^{-}b_{2}^{+}\omega_{2}\cdot O_{\tilde{\mathcal{N}}}
\]
Thus $R\Gamma^{G}(\omega_{1}^{-1}b_{1}^{-}b_{2}^{+}\omega_{2}\cdot O_{\tilde{\mathcal{N}}})\tilde{=}k$,
and so $R\Gamma^{G}(\omega_{1}^{-1}b_{1}^{-}b_{2}^{+}\omega_{2}\cdot O_{\tilde{\mathfrak{g}}})\otimes_{O(\mathfrak{h}^{*})}^{L}k_{0}\tilde{=}k$;
implying the existence of a $G\times\mathbb{G}_{m}$-global section
of $\omega_{1}^{-1}b_{1}^{-}b_{2}^{+}\omega_{2}\cdot O_{\tilde{\mathfrak{g}}}$
and thus a morphism 
\begin{equation}
O_{\tilde{\mathfrak{g}}}\to\omega_{1}^{-1}b_{1}^{-}b_{2}^{+}\omega_{2}\cdot O_{\tilde{\mathfrak{g}}}\label{eq:iso1}
\end{equation}

lifting the one above. By the graded Nakayama lemma (applied locally),
this is a surjective morphism of sheaves. 

Now, if we play the same game with the inverse $\omega_{2}^{-1}b_{2}^{-}b_{1}^{+}\omega_{1}\cdot O_{\tilde{\mathcal{N}}}$,
we get a map the other way, such that the composition with the map
\ref{eq:iso1} is a $G\times\mathbb{G}_{m}$-endomorphism of $O_{\tilde{\mathfrak{g}}}$
lifting the identity of $O_{\tilde{\mathcal{N}}}$, which therefore
is the identity of $O_{\mathfrak{\tilde{g}}}$ since $R\Gamma^{G\times\mathbb{G}_{m}}(O_{\tilde{\mathfrak{g}}})=k$.
Thus the map \ref{eq:iso1} is injective, and hence an isomorphism. 
\end{proof}
Before we proceed, let us make one notational convention. Given a
collection of objects $\{D_{i}\}$ in a triangulated category $\mathcal{C}$,
we shall say that an object $X$ is filtered by the $\{D_{i}\}$ if
there is a finite sequence of objects $\{X_{j}\}_{j=1}^{n}$ with
$X_{1}\in\{D_{i}\}$, $X_{n}=X$, and for all $j$ there are exact
triangles: 
\[
X_{j-1}\to X_{j}\to Q
\]
 where $Q$ is an object in the set $\{D_{i}\}$. 

Now, the very definition of reflection functors implies the following 
\begin{claim}
Suppose that $X$ is an object in $D^{b}Coh^{G}(\tilde{\mathfrak{g}})$
(or $D^{b}Coh^{G\times\mathbb{G}_{m}}(\tilde{\mathfrak{g}})$) which
is filtered by the $\Delta_{\lambda}$ (the $\Delta_{\lambda}(i)$,
respectively). Then $\mathcal{R}_{\alpha}X$ is also so filtered. 

The same holds if the $\Delta_{\lambda}$ are replaced by the $\nabla^{\lambda}$
. \end{claim}
\begin{proof}
The proof of the first part comes from the exact sequence 
\[
X\to\mathcal{R}_{\alpha}X\to s_{\alpha}X
\]
 since by definition, if $X$ is filtered by the $\Delta_{\lambda}$,
so is $s_{\alpha}X$, and thus so is $\mathcal{R}_{\alpha}X$. The
second part follows by using the exact sequence for $s_{\alpha}^{-1}$ 
\end{proof}
Now, we already know a finite collection of objects of $D^{b}Coh^{G}(\tilde{\mathfrak{g}})$
(and also $D^{b}Coh^{G\times\mathbb{G}_{m}}(\tilde{\mathfrak{g}})$)
which are filtered by both the $\{\Delta_{\lambda}(\tilde{\mathfrak{g}})\}$
and the $\{\nabla^{\lambda}(\tilde{\mathfrak{g}})\}$: namely, the
set $\{\omega\cdot O_{\tilde{\mathfrak{g}}}\}_{\omega\in\Omega}$,
which are both standard and costandard. So it follows from the claim
that all objects of the form 
\[
\mathcal{R}_{\alpha_{1}}\mathcal{R}_{\alpha_{2}}\cdot\cdot\cdot\mathcal{R}_{\alpha_{n}}\omega\cdot O_{\tilde{\mathfrak{g}}}
\]
 for all collections of affine simple roots $\{\alpha_{i}\}_{i=1}^{n}$
are filtered by both the $\{\Delta_{\lambda}(\tilde{\mathfrak{g}})\}$
and the $\{\nabla^{\lambda}(\tilde{\mathfrak{g}})\}$. 

Now, we can consider the restriction of such an object to $\tilde{\mathcal{N}}$.
We have the 
\begin{claim}
All objects of the form 
\[
\mathcal{R}_{\alpha_{1}}\mathcal{R}_{\alpha_{2}}\cdot\cdot\cdot\mathcal{R}_{\alpha_{n}}\omega\cdot O_{\tilde{\mathfrak{g}}}|_{\tilde{\mathcal{N}}}
\]
 are tilting objects in the heart of the perversely exotic $t$-structure.
All indecomposable tilting objects for this $t$-structure are summands
of such objects. \end{claim}
\begin{proof}
We shall show that these are tilting objects in $\mathcal{A}$ by
from the definition of tilting. By the construction, these objects
are filtered (in the triangulated sense) by both standard and costandard
objects. Now, consider the exact triangles of the form 
\[
\mathcal{R}_{\alpha_{2}}\cdot\cdot\cdot\mathcal{R}_{\alpha_{n}}\omega\cdot O_{\tilde{\mathfrak{g}}}|_{\tilde{\mathcal{N}}}\to\mathcal{R}_{\alpha_{1}}\mathcal{R}_{\alpha_{2}}\cdot\cdot\cdot\mathcal{R}_{\alpha_{n}}\omega\cdot O_{\tilde{\mathfrak{g}}}|_{\tilde{\mathcal{N}}}\to s_{\alpha_{1}}\mathcal{R}_{\alpha_{2}}\cdot\cdot\cdot\mathcal{R}_{\alpha_{n}}\omega\cdot O_{\tilde{\mathfrak{g}}}|_{\tilde{\mathcal{N}}}
\]
We assume by induction that the leftmost object is in $\mathcal{A}$
and is even filtered, in $\mathcal{A}$, by standard and costandard
objects. Then the right hand object is in $\mathcal{A}$ by \prettyref{lem:filteredobjects}.
Thus the middle object is in $\mathcal{A}$ and its filtration by
standard objects is actually a filtration in $\mathcal{A}$, by remark
\prettyref{rem:filteredobjects}. Using the other exact sequence,
we obtain that the same is true for its filtration by costandard objects,
and it is a tilting object by definition. 

To see that we obtain all indecomposable tilting objects as summands,
we use the same exact sequences as in the first part- by going through
all reduced words in $W_{aff}^{'}$, we eventually can obtain objects
which have any given $b^{-}\omega\cdot O_{\tilde{\mathcal{N}}}$ at
the top of the filtration. Thus we obtain all tilting modules.
\end{proof}
Having thus obtained tilting modules explicitly as a restriction of
certain objects on $\tilde{\mathfrak{g}}$, we have obvious candidates
for the deformation of these modules to $\tilde{\mathfrak{g}}$, which
we shall use. In particular, we can immediately deduce the following 
\begin{lem}
The objects $\mathcal{R}_{\alpha_{1}}\mathcal{R}_{\alpha_{2}}\cdot\cdot\cdot\mathcal{R}_{\alpha_{n}}\omega\cdot O_{\tilde{\mathfrak{g}}}$
satisfy 
\[
End_{D^{b}Coh^{G}(\tilde{\mathfrak{g}})}^{\cdot}(\mathcal{R}_{\alpha_{1}}\mathcal{R}_{\alpha_{2}}\cdot\cdot\cdot\mathcal{R}_{\alpha_{n}}\omega\cdot O_{\tilde{\mathfrak{g}}})=End_{D^{b}Coh^{G}(\tilde{\mathfrak{g}})}^{0}(\mathcal{R}_{\alpha_{1}}\mathcal{R}_{\alpha_{2}}\cdot\cdot\cdot\mathcal{R}_{\alpha_{n}}\omega\cdot O_{\tilde{\mathfrak{g}}})
\]
 The same is true in $D^{b}Coh^{G\times\mathbb{G}_{m}}(\tilde{\mathfrak{g}})$
for the objects $\mathcal{R}_{\alpha_{1}}\mathcal{R}_{\alpha_{2}}\cdot\cdot\cdot\mathcal{R}_{\alpha_{n}}O_{\tilde{\mathfrak{g}}}(i)$. \end{lem}
\begin{proof}
By the deformation arguments of section 2, 
\begin{equation}
End_{D^{b}Coh^{G}(\tilde{\mathfrak{g}})}^{\cdot}(\mathcal{R}_{\alpha_{1}}\mathcal{R}_{\alpha_{2}}\cdot\cdot\cdot\mathcal{R}_{\alpha_{n}}\omega\cdot O_{\tilde{\mathfrak{g}}})\otimes_{O(\mathfrak{h}^{*})}^{L}k_{0}\tilde{=}End_{D^{b}Coh^{G}(\tilde{\mathcal{N}})}^{\cdot}(\mathcal{R}_{\alpha_{1}}\mathcal{R}_{\alpha_{2}}\cdot\cdot\cdot\mathcal{R}_{\alpha_{n}}\omega\cdot O_{\tilde{\mathfrak{g}}}|_{\tilde{\mathcal{N}}})
\end{equation}
as complexes of graded modules. 

Now, the complex on the right of equation 3.2 has no terms outside
of degree zero, because the objects there are tilting modules for
a $t$-structure. But then the same must be true for the complex 
\[
End_{D^{b}Coh^{G}(\tilde{\mathfrak{g}})}^{\cdot}(\mathcal{R}_{\alpha_{1}}\mathcal{R}_{\alpha_{2}}\cdot\cdot\cdot\mathcal{R}_{\alpha_{n}}\omega\cdot O_{\tilde{\mathfrak{g}}})
\]

by the graded Nakayama lemma (c.f. \prettyref{pro:GrNak}- to make
the arguments given there apply, one should filter the bundle $\tilde{\mathfrak{g}}\to\tilde{\mathcal{N}}$
by line bundles indexed by a basis of $\mathfrak{h}^{*}$). Hence
the first claim, and the second immediately follows since the complexes
in question are summands of the first ones. \end{proof}
\begin{rem}
We should note that a much stronger version of this lemma (without
any $G$-equivariance) is proved in \cite{key-9}. The proof there
uses reduction to characteristic $p$, where the relation with modular
representations of lie algebras is exploited.
\end{rem}
Now, we would like to show the key property of tilting modules, namely,
that $\mathcal{T}$, the full subcategory on the objects $\mathcal{R}_{\alpha_{1}}\mathcal{R}_{\alpha_{2}}\cdot\cdot\cdot\mathcal{R}_{\alpha_{n}}\omega\cdot O_{\tilde{\mathfrak{g}}}$,
generates (in the sense of section 1.5) the whole category. We recall
that all objects of the form $O_{\tilde{\mathfrak{g}}}(\lambda)$
generate (both the graded and ungraded versions of) our category.
It follows immediate that the collection $\{b\cdot O_{\tilde{\mathfrak{g}}}\}$
(where $b$ is an element of the extended affine braid group) also
generates our category. 

With this in hand, we can state and prove the 
\begin{cor}
The tilting objects generate (both the graded and ungraded versions
of) our category. \end{cor}
\begin{proof}
Let $\mathcal{C}$ denote the full triangulated subcategory containing
$\mathcal{T}$ and closed under extensions, shifts, and direct summands.
Then for any sequence of reflections and any $\omega\in\Omega$, we
have the exact sequences 
\[
\mathcal{R}_{\alpha_{2}}\cdot\cdot\cdot\mathcal{R}_{\alpha_{n}}\omega\cdot O_{\tilde{\mathfrak{g}}}\to\mathcal{R}_{\alpha_{1}}\mathcal{R}_{\alpha_{2}}\cdot\cdot\cdot\mathcal{R}_{\alpha_{n}}\omega\cdot O_{\tilde{\mathfrak{g}}}\to s_{\alpha_{1}}\mathcal{R}_{\alpha_{2}}\cdot\cdot\cdot\mathcal{R}_{\alpha_{n}}\omega\cdot O_{\tilde{\mathfrak{g}}}
\]
 and 
\[
s_{\alpha_{1}}^{-1}\mathcal{R}_{\alpha_{2}}\cdot\cdot\cdot\mathcal{R}_{\alpha_{n}}\omega\cdot O_{\tilde{\mathfrak{g}}}\to\mathcal{R}_{\alpha_{1}}\mathcal{R}_{\alpha_{2}}\cdot\cdot\cdot\mathcal{R}_{\alpha_{n}}\omega\cdot O_{\tilde{\mathfrak{g}}}\to\mathcal{R}_{\alpha_{2}}\cdot\cdot\cdot\mathcal{R}_{\alpha_{n}}\omega\cdot O_{\tilde{\mathfrak{g}}}
\]
 which imply that all objects of the form $s_{\alpha_{1}}^{\pm}\mathcal{R}_{\alpha_{2}}\cdot\cdot\cdot\mathcal{R}_{\alpha_{n}}\omega\cdot O_{\tilde{\mathfrak{g}}}$
are in $\mathcal{C}$. But now consider the sequence 
\[
s_{\alpha_{1}}^{\pm}\mathcal{R}_{\alpha_{3}}\cdot\cdot\cdot\mathcal{R}_{\alpha_{n}}\omega\cdot O_{\tilde{\mathfrak{g}}}\to s_{\alpha_{1}}^{\pm}\mathcal{R}_{\alpha_{2}}\cdot\cdot\cdot\mathcal{R}_{\alpha_{n}}\omega\cdot O_{\tilde{\mathfrak{g}}}\to s_{\alpha_{1}}^{\pm}s_{\alpha_{2}}\mathcal{R}_{\alpha_{3}}\cdot\cdot\cdot\mathcal{R}_{\alpha_{n}}\omega\cdot O_{\tilde{\mathfrak{g}}}
\]
 By the previous implication, we obtain that all objects of the form
$s_{\alpha_{1}}^{\pm}s_{\alpha_{2}}\mathcal{R}_{\alpha_{3}}\cdot\cdot\cdot\mathcal{R}_{\alpha_{n}}\omega\cdot O_{\tilde{\mathfrak{g}}}$
are in $\mathcal{C}$; and considering the other exact sequence for
$\mathcal{R}_{\alpha_{2}}\cdot\cdot\cdot\mathcal{R}_{\alpha_{n}}\omega\cdot O_{\tilde{\mathfrak{g}}}$
yields in fact that all objects of the form $s_{\alpha_{1}}^{\pm}s_{\alpha_{2}}^{\pm}\mathcal{R}_{\alpha_{3}}\cdot\cdot\cdot\mathcal{R}_{\alpha_{n}}\omega\cdot O_{\tilde{\mathfrak{g}}}$
are in $\mathcal{C}$. 

Continuing in this way, we eventually see that all objects of the
form $s_{\alpha_{1}}^{\pm}s_{\alpha_{2}}^{\pm}\cdot\cdot\cdot s_{\alpha_{n}}^{\pm}\omega\cdot O_{\tilde{\mathfrak{g}}}$
are in $\mathcal{C}$. But the previous comments then imply that $\mathcal{C}$
is the entire category.
\end{proof}

\subsection{$\tilde{D}_{h}$-Modules}

In this section we would like to show that the above considerations
extend to the category $D^{b}(Mod^{G}(\tilde{D}_{h}))$ and its graded
version. The point is to show that all of the definitions lift from
the coherent case.

\subsubsection{Line Bundles}

In this section, we shall show that the notion of twisting by a line
bundle lifts in a canonical way. This material follows easily from
standard knowledge of twisted differential operators. 

Let us start by recalling that for each $\lambda\in\mathfrak{h}^{*}$,
we define the sheaf of $\lambda$-twisted quantized differential operators,
$D_{h}^{\lambda}$, to be the quotient of $\tilde{D}_{h}$ by the
ideal sheaf generated by $\{v-h\lambda(v)|v\in\mathfrak{h}\}$. 

The sheaf $D_{h}^{\lambda}/(h-1)=D^{\lambda}$ is the well known sheaf
of twisted differential operators considered, e.g., in \cite{key-3}
(c.f. also \cite{key-27}, Chapter 1). When $\lambda$ is in the weight
lattice, we have an isomorphism 
\[
D^{\lambda}\tilde{=}D(O(\lambda))
\]
 where the object on the right is the sheaf of differential operators
of the equivariant line bundle $O(\lambda)$ on $\mathcal{B}$. This
isomorphism is even an isomorphism of filtered algebras, where the
algebra on the left has the filtration induced from $\tilde{D}_{h}$,
and the algebra on the right has the filtration by order of differential
operators. The associated graded of these algebras is clearly $O(\tilde{\mathcal{N}})$
(considered as a sheaf on $\mathcal{B}$). Further, the Rees algebra
of $D^{\lambda}$ is then isomorphic to $D_{h}^{\lambda}$ (more or
less by definition). 

Now, the algebra $D(O(\lambda))$ comes with a natural action on the
line bundle $O(\lambda)$- since by definition an element of $D(O(\lambda))$
is an operator on $O(\lambda)$. Let us equip $O(\lambda)$ with the
trivial filtration- all terms in degree zero. Then the natural action
of $D^{\lambda}$ on $O(\lambda)$ becomes a filtration respecting
action, and we can therefore deform it to an action of $D_{h}^{\lambda}$
on $Rees(O(\lambda))\tilde{=}O(\lambda)[h]$. 

Under the natural morphism $\tilde{D}_{h}\to D_{h}^{\lambda}$, we
obtain that $O(\lambda)[h]$ is a graded $\tilde{D}_{h}$-module.
The natural $G$-equivariant structure on $O(\lambda)$ lifts to $O(\lambda)[h]$
(by letting $G$ act trivially on $h$) to make it an element of $Mod^{G\times\mathbb{G}_{m}}(\tilde{D}_{h})$. 

Finally, this allows us to define the $\lambda$-twist of any $\tilde{D}_{h}$
module as follows: there is an isomorphism of sheaves 
\[
M\otimes_{O(\mathcal{B})}O(\lambda)\tilde{=}M\otimes_{O(\mathcal{B}\times\mathbb{A}^{1})}O(\lambda)[h]
\]
 and we can define the action of any $\xi\in\tilde{D}_{h}$ on the
right hand side by the usual formula: 
\[
\xi\cdot(m\otimes v)=\xi m\otimes v+m\otimes\xi v
\]
 We shall denote by $\tilde{D}_{h}(\lambda)$ the module $\tilde{D}_{h}\otimes_{O(\mathcal{B})}O(\lambda)$
with this action. By the projection formula, $\tilde{D}_{h}(\lambda)/h$
is the sheaf $O(\tilde{\mathfrak{g}})(\lambda)$ considered as a sheaf
on $\mathcal{B}$.

\subsubsection{Deforming Reflection Functors}

Now we wish to define the deformation of our reflection functors,
starting with the finite roots. Recall that for each finite reflection
$s$ we have a map 
\[
\pi_{s}:\mathcal{B}\to\mathcal{P}_{s}
\]
 and also an extension of this map (which we also called $\pi_{s}$)
$\tilde{\mathfrak{g}}\to\tilde{\mathfrak{g}}_{s}$. 

We can realize this latter map as follows: define the variety $\tilde{\mathfrak{g}}(s)$
as the incidence variety 
\[
\{(x,\mathfrak{b})\in\mathfrak{g}^{*}\times\mathcal{B}|x|_{\mathfrak{u}(\pi_{s}(\mathfrak{b}))}=0\}
\]
Equivalently, we have 
\[
\tilde{\mathfrak{g}}(s)=\tilde{\mathfrak{g}}_{s}\times_{\mathcal{P}_{s}}\mathcal{B}
\]
 where $\tilde{\mathfrak{g}}_{s}\to\mathcal{P}_{s}$ is the projection,
and $\mathcal{B}\to\mathcal{P}_{s}$ is $\pi_{s}$. Thus we see there
are natural maps 
\[
\tilde{\mathfrak{g}}\to\tilde{\mathfrak{g}}(s)\to\tilde{\mathfrak{g}}_{s}
\]
 (inclusion and projection, respectively), and the composition is
nothing but our standard map $\tilde{\mathfrak{g}}\to\mathfrak{\tilde{g}}_{s}$. 

The benefit of writing things this way is that it makes deformation
easy. In particular, we can define a sheaf (on $\mathcal{B}$), called
$\tilde{D}_{h}(s)$, quantizing $\tilde{\mathfrak{g}}_{s}$, as follows:
recall the sheaf of algebras $U^{0}$ on $\mathcal{B}$ from section
2, and define $\mathfrak{u}^{0}(s)$ as the ideal sheaf generated
by $\mathfrak{u}(\pi_{s}(\mathfrak{b}))$ at each point. Then the
quotient 
\[
\tilde{D}(s)=U^{0}/\mathfrak{u}^{0}(s)
\]
 is naturally a filtered sheaf of algebras which deforms $\tilde{\mathfrak{g}}(s)$;
and we set $\tilde{D}_{h}(s)=Rees(\tilde{D}(s))$. 

Then there is the obvious quotient map $\tilde{D}_{h}(s)\to\tilde{D}_{h}$
which deforms the inclusion $\tilde{\mathfrak{g}}\to\tilde{\mathfrak{g}}_{s}$.
The pushforward under this inclusion is then deformed by the functor
which regards a $\tilde{D}_{h}$-module as a $\tilde{D}_{h}(s)$ module. 

Further, if $M$ is a $\tilde{D}_{h}(s)$-module, then clearly $(\pi_{s})_{*}(M)$
is a $\tilde{D}_{h,\mathcal{P}_{s}}$-module. Thus we have defined
a natural functor (which on the level of sheaves is simply $(\pi_{s})_{*}$)
from $Mod(\tilde{D}_{h})$ to $Mod(\tilde{D}_{h,\mathcal{P}_{S}})$.
Thus functor clearly respects $G$-equivariance and grading. Further,
this is a deformation of the pushforward map on coherent sheaves,
in the sense of the following
\begin{prop}
We have an isomorphism 
\[
(\pi_{s})_{*}(M)\otimes_{k[h]}^{L}k_{0}\tilde{=}(\pi_{s})_{*}(M\otimes_{k[h]}^{L}k_{0})
\]
 where $M\in Mod^{\mathbb{G}_{m}}(\tilde{D}_{h})$ and the $(\pi_{s})_{*}$
on the right is that of coherent sheaves. 
\end{prop}
The proof follows immediately from \prettyref{lem:Cohomology1} and
\prettyref{cor:Cohomology2}.

Next, we would like to define a functor 
\[
\pi_{s}^{*}:\tilde{D}_{h,\mathcal{P}_{s}}\to\tilde{D}_{h,\mathcal{B}}
\]
This we can also do, following the definition for coherent sheaves.
So, we first deform the pullback along the map 
\[
\tilde{\mathfrak{g}}(s)\to\tilde{\mathfrak{g}}_{s}
\]
by defining a functor on $Mod(\tilde{D}_{h,\mathcal{P}_{s}})$ as
\[
M\to O(\mathcal{B})\otimes_{\pi_{s}^{-1}(O(\mathcal{P}))}\pi_{s}^{-1}(M)
\]
 which is clearly an object of $Mod(\tilde{D}_{h}(s))$ (by the definition
of $\tilde{D}_{h}(s)$). 

Next we deform the pullback along the map $\tilde{\mathfrak{g}}\to\tilde{\mathfrak{g}}(s)$
by defining a functor on $Mod(\tilde{D}_{h}(s))$ as 
\[
M\to\tilde{D}_{h}\otimes_{\tilde{D}_{h}(s)}M
\]
 where $\tilde{D}_{h}$ is a module over $\tilde{D}_{h}(s)$ by the
natural surjection of algebras; so in fact this functor is just 
\[
M\to M/\mathcal{I}
\]
 where $\mathcal{I}$ is the ideal sheaf kernel of $\tilde{D}_{h}(s)\to\tilde{D}_{h}$. 

Taking the composition of these two functors yields a functor $\pi_{s}^{*}:Mod(\tilde{D}_{h,\mathcal{P}_{s}})\to Mod(\tilde{D}_{h})$.
As above, one checks immediately that this functor preserves graded
and Equivariant versions of the category, and it is very easy to see
that this functor deforms the pullback of coherent sheaves (as in
the proposition above). 

Let us note that there is a natural adjunction $\pi_{s}^{*}\pi_{s*}\to Id$
for the usual reasons. Further, we note that these functors extend
naturally to derived functors, and this adjunction continues to hold
at that level.

Thus we can now make the 
\begin{defn}
To any finite simple root $\alpha$, we associate the functor $\mathcal{R}_{\alpha_{s}}=L\pi_{s}^{*}R(\pi_{s})_{*}$
on $D^{b}(Mod(\tilde{D}_{h}))$. Thus functor also preserves the equivariant
and graded versions of this category. 
\end{defn}
This is a crucial step in deforming tilting modules. We shall need
to gain some further insight into the behavior of these functors.
To do so, we shall reformulate them in terms of a {}``Fourier-Mukai''
type set up.

\subsection{Fourier-Mukai functors for Noncommutative rings. }

In this section, we would like to set up a version of Fourier-Mukai
theory which applies to certain sheaves of non-commutative rings on
nice spaces. 

We start with the {}``affine'' case of noncommutative rings themselves;
more precisely, let $A$ and $B$ be two noncommutative flat noetherian
$k$-algebras, and let $F^{\cdot}$ be a complex in $D^{b}(A^{opp}\otimes_{k}B-mod)$.
To $F^{\cdot}$ we shall associate a functor 
\[
\mathcal{F}_{F^{\cdot}}:D^{b}(A-mod)\to D^{b}(B-mod)
\]
 as follows: 

For any $M^{\cdot}\in D^{b}(A-mod)$, we can form the tensor product
$M^{\cdot}\otimes_{k}^{L}B$. Since $B$ is right $B$-module, we
can consider this tensor product as an element of \linebreak{}
$D^{b}(A\otimes_{k}B^{opp}-mod$); in addition, it carries a left
action of $B$ via $b\to1\otimes b$. So the complex 
\[
(M^{\cdot}\otimes_{k}^{L}B)\otimes_{A\otimes_{k}B^{opp}}^{L}F^{\cdot}
\]
 which is a priori just a complex of $k$-vector spaces, is in fact
a complex of left $B$-modules also; this defines the functor $\mathcal{F}_{F^{\cdot}}$. 

There are two main cases of interest. First up, we have the
\begin{example}
Let us suppose that $A=B$. 

Then, the identity functor on $D^{b}(A-mod)$ can be realized via
the kernel $A$ considered as an $A^{opp}\otimes_{k}A$-module. To
see this, we note that for any \linebreak{}
$M^{\cdot}\in D^{b}(A-mod)$, there is a morphism of complexes 
\[
M^{\cdot}\to(M^{\cdot}\otimes_{k}^{L}A)\otimes_{A\otimes_{k}A^{opp}}^{L}A
\]

given at the level of objects by $m\to m\otimes1\otimes1$. However,
the complex on the right is equivalent to another where $M^{\cdot}$
is replaced by a finite complex of projective $A$ modules. Hence,
to show that this morphism is an isomorphism, we can in fact assume
that $M^{\cdot}$ is a complex of free $A$-modules. But in that case
the complex on the right is evidently a complex of free $A$-modules
of the same rank, and the map just becomes the identity. 
\end{example}
We shall also need to consider the cases provided by the following
\begin{example}
Now let us suppose that there is an algebra map $f:A\to B$ between
two noetherian flat $k$-algebras. We shall express the Fourier-Mukai
kernels for the functors $M^{\cdot}\to M^{\cdot}\otimes_{A}^{L}B$
(here $B$ is considered as a right $A$-module via $f$, and the
resulting complex is a left $B$-module via the left action of $B$
on itself) and $N^{\cdot}\to Res_{A}^{B}(N^{\cdot})$ (where the restriction
is over the map $f$). 

For the first functor, which is an arrow $D^{b}(A-mod)\to D^{b}(B-mod)$,
we consider the object $B\in D^{b}(A^{opp}\otimes_{k}B)$- here letting
$A$ act on the right via $f$ and $B$ act on the left. Then we have
a morphism of complexes of $B$-modules 
\[
M^{\cdot}\otimes_{A}^{L}B\to(M^{\cdot}\otimes_{k}^{L}B)\otimes_{A\otimes_{k}B^{opp}}^{L}B
\]
 which comes from the map on objects sending $m\otimes b\to m\otimes1\otimes b$.
As above, we can actually replace $M^{\cdot}$ by a complex of free
$A$-modules to see that this is an isomorphism. 

For the second functor, we use the object $B\in D^{b}(B^{opp}\otimes_{k}A)$
considered as a left $A$-module via $f$ and a right $B$-module.
A similar argument says that this gives the functor $Res$. 
\end{example}

\subsubsection{Fourier Mukai Kernels for deformed reflection functors. }

Now we shall give the sheaf-theoretic versions of the above functors
in the cases which are relevant to us. The general set-up simply copies
the affine case: let $X$ and $Y$ be $k$- varieties, and let $\mathcal{A}$
and $\mathcal{B}$ be quasi-coherent sheaves of non-commutative rings
on $X$ and $Y$, respectively. We assume $\mathcal{A}$ and $\mathcal{B}$
are noetherian and flat over $k$. 

Then, the product variety $X\times Y$ carries the sheaf of rings
$\mathcal{A}\boxtimes\mathcal{B}$. We consider a complex of sheaves
$\mathcal{F}^{\cdot}$ belonging to $D^{b}((\mathcal{A}^{opp}\boxtimes\mathcal{B})-mod)$.
Then $\mathcal{F}^{\cdot}$ defines a functor 
\[
M^{\cdot}\to Rp_{*}((M^{\cdot}\boxtimes\mathcal{B})\otimes_{\mathcal{A}\boxtimes\mathcal{B}^{opp}}^{L}\mathcal{F}^{\cdot})
\]
 which goes from $D^{b}(\mathcal{A}-mod)$ to $D^{b}(\mathcal{B}-mod)$;
here $p$ denotes the projection $X\times Y\to Y$. As above, the
left action of $\mathcal{B}$ comes from the additional action of
$O_{X}\boxtimes\mathcal{B}$ on the sheaf $(M^{\cdot}\boxtimes\mathcal{B})\otimes_{\mathcal{A}\boxtimes\mathcal{B}^{opp}}^{L}\mathcal{F}^{\cdot}$. 

In the case of interest to us, the categories under consideration
are $Mod(\tilde{D}_{h,\mathcal{P}})$ for the various flag varieties
$\mathcal{P}$, which will mainly be $\mathcal{B}$ or $\mathcal{P}_{s}$.
We shall use the principles of the above section to construct the
functors $Id$, $(\pi_{s})_{*}$, and $\pi_{s}^{*}$ (from now on
we only consider the derived version of these functors; every functor
in sight is taken to be derived). 
\begin{example}
We start with the functor $Id$. Consider the sheaf of algebras \linebreak{}
$\tilde{D}_{h}\boxtimes\tilde{D}_{h}^{opp}$ on the scheme $\mathcal{B}\times\mathcal{B}$,
and let $M^{\cdot}$ be a complex in $D^{b}(Mod(\tilde{D}_{h})).$
We shall form the Fourier-Mukai functor associated to the diagonal
bimodule $\tilde{D}_{h}\in Mod(\tilde{D}_{h}^{opp}\boxtimes\tilde{D_{h}})$.
This is the functor 
\[
M^{\cdot}\to(p_{2})_{*}((M^{\cdot}\boxtimes\tilde{D}_{h})\otimes_{\tilde{D}_{h}\boxtimes\tilde{D}_{h}^{opp}}^{L}\tilde{D}_{h})
\]
 (where $p_{2}:\mathcal{B}\times\mathcal{B}\to\mathcal{B}$ is the
second projection), which will therefore be an endofunctor of the
category $D^{b}(Mod(\tilde{D}_{h}))$. 

To see that this is just the identity functor, we can reduce to the
affine case discussed above by noting that the $(\tilde{D}_{h}\boxtimes\tilde{D}_{h}^{opp})$-module
$\tilde{D}_{h}$ is concentrated on the diagonal of $\mathcal{B}\times\mathcal{B}$,
which easily implies that we can reduce to the affine case by working
with a covering of the form $\{U\times U:U\subseteq\mathcal{B}\mbox{ \ensuremath{}is affine}\}$. 
\end{example}
Next, we take care of the push and pull functors. 
\begin{example}
We have the sheaves of algebras $\tilde{D}_{h}^{opp}\boxtimes\tilde{D}_{h,\mathcal{P}}$
on the scheme $\mathcal{B}\times\mathcal{P}$, and $\tilde{D}_{h,\mathcal{P}}^{opp}\boxtimes\tilde{D}_{h}$
on $\mathcal{P}\times\mathcal{B}$. Via the map $\pi_{s}:\mathcal{B}\to\mathcal{P}_{s}$,
we get the graph subschemes $\Gamma_{\pi_{s}}\subseteq\mathcal{B}\times\mathcal{P}_{s}$
and $\Gamma_{\pi_{s}}^{'}\subseteq\mathcal{P}_{s}\times\mathcal{B}$
(the latter is the flip of the former). 

We shall now construct certain modules supported on these graph subschemes:
let $i:\mathcal{B}\to\Gamma_{\pi_{s}}$ be the natural isomorphism.
We define $(\tilde{D}_{h})_{\pi_{s}}$ as the $\tilde{D}_{h}^{opp}\boxtimes\tilde{D}_{h,\mathcal{P}}$-module
which is simply $i_{*}\tilde{D}_{h}$ as a sheaf, with the obvious
right action of $\tilde{D}_{h}$. The action by $\tilde{D}_{h,\mathcal{P}}$
then comes via the natural algebra morphism $\pi_{s}^{*}(\tilde{D}_{h,\mathcal{P}})\to\tilde{D}_{h}$. 

Similarly, we can define $_{\pi_{s}}(\tilde{D}_{h})$ over $\tilde{D}_{h,\mathcal{P}}^{opp}\boxtimes\tilde{D}_{h}$
as follows: we have the subscheme $\Gamma_{\pi_{s}}^{'}\subseteq\mathcal{P}_{s}\times\mathcal{B}$,
and a corresponding isomorphism $i^{'}:\mathcal{B}\to\Gamma_{\pi_{s}}^{'}$.
Then we can consider the module $i_{*}^{'}(\tilde{D}_{h})$ as a left
$\tilde{D}_{h}$-module, which also inherits a right action of $\tilde{D}_{h,\mathcal{P}}$
via the morphism $\pi_{s}$; this will be our $_{\pi_{s}}(\tilde{D}_{h})$. 

Now, let $M^{\cdot}\in D^{b}(Mod(\tilde{D}_{h,\mathcal{P}}))$. Then
we have the complex \linebreak{}
$M^{\cdot}\boxtimes\tilde{D}_{h}\in D^{b}(Mod(\tilde{D}_{h,\mathcal{P}}\boxtimes\tilde{D}_{h}^{opp}))$,
and so we can form the tensor product 
\[
(M^{\cdot}\boxtimes\tilde{D}_{h})\otimes_{\tilde{D}_{h,\mathcal{P}}\boxtimes\tilde{D}_{h}^{opp}}^{L}{}_{\pi_{s}}(\tilde{D}_{h})
\]
 which carries an additional action of the sheaf $O_{\mathcal{P}}\boxtimes\tilde{D}_{h}$
(via the left action of $\tilde{D}_{h}$ on itself in the first factor).
Thus the complex 
\[
(p_{2})_{*}((M^{\cdot}\boxtimes\tilde{D}_{h})\otimes_{\tilde{D}_{h,\mathcal{P}}\boxtimes\tilde{D}_{h}^{opp}}^{L}{}_{\pi_{s}}(\tilde{D}_{h}))
\]
 is naturally in $D^{b}(Mod(\tilde{D}_{h}))$. We claim that this
complex is functorially isomorphic to our complex $L\pi_{s}^{*}(M^{\cdot})$
defined above. The morphism of complexes is the same one as was used
in the affine case, and in fact if we locally replace $M^{\cdot}$
by a complex of free $\tilde{D}_{h,\mathcal{P}}$-modules, we can
use the same argument to prove that this map is an isomorphism (noting
that $_{\pi_{s}}(\tilde{D}_{h})$ is acyclic for $(p_{2})_{*}$ by
the results of section 2). 

In a very similar way, one shows that for $M^{\cdot}\in D^{b}(Mod(\tilde{D}_{h}))$,
the functor 
\[
M^{\cdot}\to(p_{2})_{*}((M^{\cdot}\boxtimes\tilde{D}_{h,\mathcal{P}})\otimes_{\tilde{D}_{h}\boxtimes\tilde{D}_{h,\mathcal{P}}^{opp}}(\tilde{D}_{h})_{\pi_{s}})
\]
is isomorphic to $(R\pi_{s})_{*}$. 
\end{example}

\subsubsection{Convolution}

With these preliminaries out of the way, we shall develop several
important properties of our reflection functors. To reach our aims,
we shall first explain how convolution works in a general setting
setting. 

We suppose varieties $X$, $Y$, and $Z$, with flat noetherian sheaves
of algebras $\mathcal{A}$, $\mathcal{B}$, $\mathcal{C}$, respectively.
We let $M^{\cdot}\in D^{b}(\mathcal{A}\boxtimes\mathcal{B}-mod)$
and $N^{\cdot}\in D^{b}(\mathcal{B}^{opp}\boxtimes\mathcal{C}-mod)$.
We shall construct an object $M^{\cdot}\star N^{\cdot}\in D^{b}(\mathcal{A}\boxtimes\mathcal{C}-mod)$,
as follows: 

We have the object $M^{\cdot}\boxtimes\mathcal{C}\in D^{b}(\mathcal{A}\boxtimes\mathcal{B}\boxtimes\mathcal{C}^{opp}-mod)$
(by looking at $\mathcal{C}$ as a right module over itself); this
object admits an additional action of $O_{X}\boxtimes O_{Y}\boxtimes\mathcal{C}$
via the left action of $\mathcal{C}$ on itself. In the same vein,
we can consider the object $\mathcal{A}^{opp}\boxtimes N^{\cdot}\in D^{b}(\mathcal{A}^{opp}\boxtimes\mathcal{B}^{opp}\boxtimes\mathcal{C}-mod)$,
with the additional action of $\mathcal{A}\boxtimes O_{Y}\boxtimes O_{Z}$.
Since $M^{\cdot}\boxtimes\mathcal{C}$ and $\mathcal{A}^{opp}\boxtimes N^{\cdot}$
are modules over opposed algebras, we can consider the tensor product
\[
(M^{\cdot}\boxtimes\mathcal{C})\otimes_{\mathcal{A}\boxtimes\mathcal{B}\boxtimes\mathcal{C}^{opp}}^{L}(\mathcal{A}^{opp}\boxtimes N^{\cdot})
\]
 which is then naturally a module over $\mathcal{A}\boxtimes O_{Y}\boxtimes\mathcal{C}$
via the additional actions discussed above. Thus the complex 
\[
(Rp_{13})_{*}((M^{\cdot}\boxtimes\mathcal{C})\otimes_{\mathcal{A}\boxtimes\mathcal{B}\boxtimes\mathcal{C}^{opp}}^{L}(\mathcal{A}^{opp}\boxtimes N^{\cdot})):=M^{\cdot}\star N^{\cdot}
\]
is thus an element of $D^{b}(\mathcal{A}\boxtimes\mathcal{C}-mod)$. 

On the other hand, by the general discussion of the previous sections,
the object $M^{\cdot}$ defines a functor $\mathcal{F}_{M^{\cdot}}:D^{b}(\mathcal{A}^{opp}-mod)\to D^{b}(\mathcal{B}-mod)$,
while $N^{\cdot}$ defines a functor $\mathcal{F}_{N^{\cdot}}:D^{b}(\mathcal{B}-mod)\to D^{b}(\mathcal{C}-mod)$,
and $M^{\cdot}\star N^{\cdot}$ defines a functor $\mathcal{F}_{M^{\cdot}\star N^{\cdot}}:D^{b}(\mathcal{A}^{opp}-mod)\to D^{b}(\mathcal{C}-mod)$.
Then we have the 
\begin{lem}
There is an isomorphism of functors 
\[
\mathcal{F}_{N^{\cdot}}\circ\mathcal{F}_{M^{\cdot}}\tilde{=}\mathcal{F}_{M^{\cdot}\star N^{\cdot}}
\]

\end{lem}
This lemma is simply the statement that {}``convolution becomes composition''
for Fourier-Mukai kernels (see section 3 above). The proof in the
classical case (spelled out in great detail in \cite{key-20}) works
perfectly well in our situation.

Let us spell out exactly how this works in the case of interest to
us: consider $(\tilde{D}_{h})_{\pi_{s}}\boxtimes\tilde{D}_{h}$ as
a $\tilde{D}_{h}^{opp}\boxtimes\tilde{D}_{h,\mathcal{P}}\boxtimes\tilde{D}_{h}^{opp}$-module.
This module admits an additional action by $O_{\mathcal{B}}\boxtimes O_{\mathcal{B}}\boxtimes\tilde{D}_{h}$.
Similarly, consider $\tilde{D}_{h}\boxtimes{}_{\pi_{s}}(\tilde{D}_{h})$
as a $\tilde{D}_{h}\boxtimes\tilde{D}_{h,\mathcal{P}}^{opp}\boxtimes\tilde{D}_{h}$-module;
this module admits an additional action by $\tilde{D}_{h}^{opp}\boxtimes O_{\mathcal{B}}\boxtimes O_{\mathcal{B}}$.
Then we have: 
\begin{lem}
For a finite root $\alpha$, the sheaf 
\[
\mathcal{G}_{\alpha}:=(p_{13})_{*}(((\tilde{D}_{h})_{\pi_{s}}\boxtimes\tilde{D}_{h})\otimes_{\tilde{D}_{h}^{opp}\boxtimes\tilde{D}_{h,\mathcal{P}}\boxtimes\tilde{D}_{h}^{opp}}^{L}(\tilde{D}_{h}\boxtimes{}_{\pi_{s}}(\tilde{D}_{h})))
\]
 on $\mathcal{B}\times\mathcal{B}$, which is naturally a $\tilde{D}_{h}^{opp}\boxtimes\tilde{D}_{h}$-module,
is a kernel which gives the functor $\mathcal{R}_{\alpha}$. 
\end{lem}
Now we can compare formally this functor with its classical version.
Firstly, let us note that, a priori, $\mathcal{G}_{\alpha}$ has three
actions of $h$ coming from the fact that is is the pushforward of
a sheaf defined on a three-fold product. However, these actions of
$h$ must all agree, as can be seen from the fact that the left and
right actions of $h$ on $(\tilde{D}_{h})_{\pi_{s}}$ and $_{\pi_{s}}(\tilde{D}_{h})$
agree. Then we have the
\begin{prop}
We have an isomorphism (of sheaves on $\mathcal{B}\times\mathcal{B}$)
\[
(\mathcal{G}_{\alpha})\otimes_{k[h]}^{L}k_{0}\tilde{=}O_{\tilde{\mathfrak{g}}\times_{\mathfrak{\tilde{g}}_{\alpha}}\tilde{\mathfrak{g}}}
\]
\end{prop}
\begin{proof}
The results of \cite{key-28}, section 5, show that we have an isomorphism
\[
(Rp_{13})_{*}(O_{\tilde{\mathfrak{g}}\times_{\tilde{\mathfrak{g}}_{\alpha}}\tilde{\mathfrak{g}}_{\alpha}\times\tilde{\mathfrak{g}}}\otimes_{\tilde{\mathfrak{g}}\times\tilde{\mathfrak{g}}_{\alpha}\times\tilde{\mathfrak{g}}}^{L}O_{\tilde{\mathfrak{g}}\times\tilde{\mathfrak{g}}_{\alpha}\times_{\tilde{\mathfrak{g}}_{\alpha}}\tilde{\mathfrak{g}}})\tilde{=}O_{\tilde{\mathfrak{g}}\times_{\mathfrak{\tilde{g}}_{\alpha}}\tilde{\mathfrak{g}}}
\]
 where all the products are taken over the natural map $\pi_{s}$.
But the sheaf on the left, considered as a quasicoherent sheaf on
$\mathcal{B}\times\mathcal{B}$, is isomorphic to $(\mathcal{G}_{\alpha})\otimes_{k[h]}^{L}k_{0}$
by \prettyref{lem:Cohomology1},\prettyref{cor:Cohomology2} and the
fact that the functor $\otimes_{k[h]}^{L}k_{0}$ commutes with all
tensor products. 
\end{proof}
From this proposition and the graded Nakayama lemma, it follows that
$\mathcal{G}_{\alpha}$ is a sheaf (i.e., concentrated in a single
cohomological degree). 

Now we can deduce from these facts the following crucial: 
\begin{cor}
The adjunction morphism of kernels $O_{\Delta\tilde{\mathfrak{g}}}\to O_{\tilde{\mathfrak{g}}\times_{\tilde{\mathfrak{g}}_{\alpha}}\tilde{\mathfrak{g}}}(2)$
lifts to a morphism of $\tilde{D}_{h}^{opp}\boxtimes\tilde{D}_{h}$-modules
$\tilde{D}_{h}\to\mathcal{G}_{\alpha}(2)$. \end{cor}
\begin{proof}
By the basic results of section 2 for the flag variety $\mathcal{B}\times\mathcal{B}$,
we have that the space $R\Gamma^{G}(\mathcal{G}_{\alpha})$ admits
an action of the algebra $O(\mathfrak{h}^{*}\times\mathbb{A}^{1})\otimes_{k}O(\mathfrak{h}^{*}\times\mathbb{A}^{1})$.
Further, there is a global section $1\in R\Gamma^{G}(\mathcal{G}_{\alpha})$
(obtained by looking at the image of $(1\boxtimes1)\otimes(1\boxtimes1)\in\mathcal{G}_{\alpha}$);
and so the action on $1$ produces a map 
\begin{equation}
O(\mathfrak{h}^{*}\times\mathbb{A}^{1})\otimes_{k}O(\mathfrak{h}^{*}\times\mathbb{A}^{1})\to R\Gamma^{G}(\mathcal{G}_{\alpha})\label{eq:relflectionmap}
\end{equation}

and applying $\otimes_{k[h]}^{L}k_{0}$ yields a map 
\[
O(\mathfrak{h}^{*})\otimes_{k}O(\mathfrak{h}^{*})\to R\Gamma^{G}(O_{\tilde{\mathfrak{g}}\times_{\tilde{\mathfrak{g}}_{\alpha}}\tilde{\mathfrak{g}}})
\]

It is not difficult to check that this map is a surjection, and produces
an isomorphism 
\[
O(\mathfrak{h}^{*})\otimes_{O(\mathfrak{h}^{*}/s)}O(\mathfrak{h}^{*})\to R\Gamma^{G}(O_{\tilde{\mathfrak{g}}\times_{\tilde{\mathfrak{g}}_{\alpha}}\tilde{\mathfrak{g}}})
\]
 using the Leray spectral sequence for either of the two pushforwards
$p:\tilde{\mathfrak{g}}\times\tilde{\mathfrak{g}}\to\tilde{\mathfrak{g}}$
(and the fact that $p_{*}(O_{\tilde{\mathfrak{g}}\times_{\tilde{\mathfrak{g}}_{\alpha}}\tilde{\mathfrak{g}}})\tilde{=}\mathcal{R}_{\alpha}O_{\tilde{\mathfrak{g}}}\tilde{=}O_{\tilde{\mathfrak{g}}}\oplus O_{\tilde{\mathfrak{g}}}$). 

Now, the morphism $O_{\tilde{\mathfrak{g}}\times\tilde{\mathfrak{g}}}\to O_{\Delta\tilde{\mathfrak{g}}}\to O_{\tilde{\mathfrak{g}}\times_{\tilde{\mathfrak{g}}_{\alpha}}\tilde{\mathfrak{g}}}(2)$
yields a non-trivial element $v$ in 
\[
R\Gamma^{G\times\mathbb{G}_{m}}(O_{\tilde{\mathfrak{g}}\times_{\tilde{\mathfrak{g}}_{\alpha}}\tilde{\mathfrak{g}}}(2))\tilde{=}\mathfrak{h}^{*}\oplus\mathfrak{h}^{*}
\]

Now, one can regard this element as a degree one element of $O(\mathfrak{h}^{*}\times\mathbb{A}^{1})\otimes_{k}O(\mathfrak{h}^{*}\times\mathbb{A}^{1})$.
Then its image in $R\Gamma^{G}(\mathcal{G}_{\alpha})$ under the map
\ref{eq:relflectionmap} produces a nontrivial map $\phi:\tilde{D}_{h}^{opp}\boxtimes\tilde{D}_{h}\to\mathcal{G}_{\alpha}$. 

Let $\tilde{D}_{h}\tilde{=}\tilde{D}_{h}^{opp}\boxtimes\tilde{D}_{h}/\mathcal{J}$
as $\tilde{D}_{h}^{opp}\boxtimes\tilde{D}_{h}$-modules. We wish to
show that the map $\phi$ dies on the submodule $\mathcal{J}$. We
know this is true upon setting both copies of $h=0$. However, $\mathcal{J}\subseteq ker(\phi)$
can then be seen from the fact that $\mathcal{J}$ is generated locally
by elements in grade degree $2$ which survive after killing $h$;
as well as the fact that the two actions of $h$ on $\mathcal{G}_{\alpha}$
(coming from $\tilde{D}_{h}^{opp}\boxtimes\tilde{D}_{h}$) agree by
definition of $\mathcal{G}_{\alpha}$. Thus $\phi$ is a lift of the
original map. \end{proof}
\begin{rem}
a) The choice of this lift depended only on the choice of an isomorphism
$O(\mathfrak{h}^{*}\times\mathbb{A}^{1})\tilde{=}R\Gamma^{G}(\tilde{D}_{h})$,
which in turn depends only on the general data fixed at the beginning
of the paper (c.f. section 2). 

b) We can compute explicitly the element $v$. By construction, it
is a degree $2$ element of $O(\mathfrak{h}^{*})\otimes_{O(\mathfrak{h}^{*}/s)}O(\mathfrak{h}^{*})$
which satisfies $(\alpha_{s}\otimes1)v=(1\otimes\alpha_{s})v$. It
is easy to see that, up to scalar, the only such is $1\otimes\alpha_{s}+\alpha_{s}\otimes1$.
We shall make this choice of $v$ from now on. \end{rem}
\begin{prop}
\label{pro:adjointness}The morphism \textup{$Id\to\mathcal{R}_{\alpha}(2)$}
is an adjointness of functors on $D^{b,\mathbb{G}_{m}}(Mod(\tilde{D}_{h}))$. \end{prop}
\begin{proof}
To see this, recall that we have morphisms of the sort 
\[
(\pi_{s})_{*}\to(\pi_{s})_{*}\mathcal{R}_{\alpha}\to(\pi_{s})_{*}
\]
(on $D^{b}(Mod(\tilde{D}_{h,\mathcal{P}}))$) which we want to show
are the identity. At the level of Fourier-Mukai kernels, we have a
morphism 
\[
(\tilde{D}_{h})_{\pi_{s}}\to(\tilde{D}_{h})_{\pi_{s}}
\]
 whose reduction mod $h$ is the identity. Since this morphism is
$G\times\mathbb{G}_{m}$-equivariant, we conclude from $R\Gamma(\tilde{D}_{h})=O(\mathfrak{h}^{*}\times\mathbb{A}^{1})$
that it is the identity as well; the same is true for the adjunction
morphism in the other order. 
\end{proof}
From this proposition, we conclude that $\mathcal{R}_{\alpha}$ is
a self-adjoint functor on $D^{b,G}(Mod(\tilde{D}_{h}))$ and its graded
version. 

As a next step, we shall also need to deform the functor associated
to the affine root. At the level of coherent sheaves, this functor
can be obtained from a finite-root functor by a certain conjugation.
Thus we need to lift the functors by which our finite root was conjugated.
That is the goal of the next subsection.

\subsection{Deforming braid generators. }

Our goal here is to deform the functors attached to braid generators.
The previous two sections have told us exactly how to do this: the
functor of action by a line bundle is given by 
\[
M\to M\otimes_{O(\mathcal{B})}O(\lambda)
\]
 made into a $\tilde{D}_{h}$-module as explained above. 

The functor associated to a finite root $\alpha$ can then be defined
using the reflection functors. By the natural adjointness property
and \prettyref{pro:adjointness}, we have natural morphisms 
\[
Id\to\mathcal{R}_{\alpha}(2)
\]
 and 
\[
\mathcal{R}_{\alpha}\to Id
\]
 and therefore we can simply define functors (on $D^{b}(Mod^{\mathbb{G}_{m}}(\tilde{D}_{h}))$)
to be triangles 
\[
Id\to\mathcal{R}_{\alpha}(2)\to s_{\alpha}
\]
 and 
\[
s_{\alpha}^{-1}(-2)\to\mathcal{R}_{\alpha}\to Id
\]
 (to get the ungraded version we simply ignore the shifts).

Our first order of business is to check that these functors are actually
inverse. For this, we shall use a very general lemma, which appears
in \cite{key-29}. The set-up is as follows: 

Suppose that $\mathcal{C}$ and $\mathcal{D}$ are algebraic triangulated
categories and let $F:\mathcal{C}\to\mathcal{D}$ and $G:\mathcal{D}\to\mathcal{C}$
be triangulated functors. Let $\Phi$ be a triangulated self-equivalence
of $\mathcal{C}$. Suppose we are given two adjoint pairs $(F,G)$
and $(G,F\Phi)$. 

Then we have the data of four morphisms of the adjunctions 
\[
\eta:1_{\mathcal{D}}\to F\Phi G\phantom{hellomy}\epsilon:G\Phi F\to1_{\mathcal{C}}
\]
 
\[
\eta':1_{\mathcal{C}}\to GF\phantom{hellomy}\epsilon':FG\to1_{\mathcal{D}}
\]
 Let $\Psi$ be the cocone of $\epsilon'$ and $\Psi^{'}$ be the
cone of $\eta$. Assume that 
\[
1_{\mathcal{C}}\to GF\to\Phi^{-1}
\]
 is a split exact triangle. Then one concludes
\begin{prop}
The functors $\Psi$ and $\Psi^{-1}$ are inverse self equivalences
of $\mathcal{D}$. 
\end{prop}
We shall apply this proposition in the case where $\mathcal{D}=D^{b}(Mod(\tilde{D}_{h}))$,
$\mathcal{C}=D^{b}(Mod(\tilde{D}_{h,\mathcal{P}}))$ $F=\pi_{s}^{*}$,
$G=(\pi_{s})_{*}$, and $\Phi=(2)$. The remaining issue is to show
that the triangle 
\[
1_{D^{b}(Mod(\tilde{D}_{h,\mathcal{P}}))}\to(\pi_{s})_{*}\pi_{s}^{*}\to(2)
\]
 is split exact- but in fact the adjunction formula implies immediately
that for any sheaf $M\in D^{b}(Mod(\tilde{D}_{h,\mathcal{P}}))$,
\[
(\pi_{s})_{*}\pi_{s}^{*}M\tilde{=}M\oplus M(2)
\]
 and so we can conclude that this the case. Thus our functors are
indeed inverse as required.

Next, we would like to show that these functors satisfy the (weak)
braid relations. We begin with the 
\begin{lem}
\label{lem:Braid-on-D(tilde)}Consider any braid relation satisfied
by the elements $\{T_{s_{\alpha}}^{\pm1}\}_{\alpha\in S_{fin}}$,
$\{\theta_{\lambda}\}$. For notational convenience, we consider $T_{s}\theta_{\lambda}=\theta_{\lambda}T_{s}$
(for $<\lambda,\alpha>=0$). Then there is an isomorphism in $D^{b}(Mod^{G\times\mathbb{G}_{m}}(\tilde{D}_{h}))$
\[
s_{\alpha}\cdot\theta_{\lambda}\cdot\tilde{D}_{h}\tilde{=}\theta_{\lambda}\cdot s_{\alpha}\cdot\tilde{D}_{h}
\]
 \end{lem}
\begin{rem}
In other words, the lemma says that the functors satisfy braid relations
upon application to $\tilde{D}_{h}$. \end{rem}
\begin{proof}
As all functors considered are invertible (c.f. the remarks right
above the lemma), we see that this comes down to showing 
\[
\tilde{D}_{h}\tilde{=}s_{\alpha}^{-1}\cdot\theta_{-\lambda}\cdot s_{\alpha}\cdot\theta_{\lambda}\cdot\tilde{D}_{h}
\]
 We know that, upon restriction to $h=0$, there is an isomorphism
\[
O_{\tilde{\mathfrak{g}}}\tilde{=}s_{\alpha}^{-1}\cdot\theta_{-\lambda}\cdot s_{\alpha}\cdot\theta_{\lambda}\cdot O_{\tilde{\mathfrak{g}}}
\]
Thus the complex $s_{\alpha}^{-1}\cdot\theta_{-\lambda}\cdot s_{\alpha}\cdot\theta_{\lambda}\cdot\tilde{D}_{h}$
is concentrated in degree zero by the graded Nakayama lemma. Now,
since $R\Gamma^{G}(O_{\tilde{\mathfrak{g}}})\tilde{=}O(\mathfrak{h}^{*})$
as graded modules, we deduce that 
\[
R\Gamma^{G}(s_{\alpha}^{-1}\cdot\theta_{-\lambda}\cdot s_{\alpha}\cdot\theta_{\lambda}\cdot\tilde{D}_{h})|_{h=0}\tilde{=}O(\mathfrak{h}^{*})
\]
 as graded modules, and that 
\[
R^{i}\Gamma^{G}(s_{\alpha}^{-1}\cdot\theta_{-\lambda}\cdot s_{\alpha}\cdot\theta_{\lambda}\cdot\tilde{D}_{h})=0
\]
for $i>0$ (by \prettyref{lem:Cohomology1} and \prettyref{cor:Cohomology2}).
So we see that there is a nontrivial element of 
\[
Hom_{(Mod^{G\times\mathbb{G}_{m}}(\tilde{D}_{h}))}(\tilde{D}_{h},s_{\alpha}^{-1}\cdot\theta_{-\lambda}\cdot s_{\alpha}\cdot\theta_{\lambda}\cdot\tilde{D}_{h})
\]
 but the restriction of any such map to $h=0$ is a nontrivial element
of 
\[
Hom_{Coh^{G\times\mathbb{G}_{m}}(\tilde{\mathfrak{g}})}(O_{\tilde{\mathfrak{g}}},s_{\alpha}^{-1}\cdot\theta_{-\lambda}\cdot s_{\alpha}\cdot\theta_{\lambda}\cdot O_{\tilde{\mathfrak{g}}})=k
\]
 and hence is an isomorphism. Thus we see that our map is surjective
by the graded Nakayama lemma (applied locally). 

To produce a morphism in the other direction, we simply run the same
argument for the {}``inverse'' complex $\theta_{-\lambda}\cdot s_{\alpha}^{-1}\cdot\theta_{\lambda}\cdot s_{\alpha}\cdot\tilde{D}_{h}$;
thus we can get a map 
\[
s_{\alpha}^{-1}\cdot\theta_{-\lambda}\cdot s_{\alpha}\cdot\theta_{\lambda}\cdot\tilde{D}_{h}\to\tilde{D}_{h}
\]
 such that the composition of the two is an endomorphism of $\tilde{D}_{h}$
lifting the identity on $O_{\tilde{\mathfrak{g}}}$; since any such
is the identity, we conclude that our original map is injective. 
\end{proof}
Now we proceed to the full statement: 
\begin{cor}
\label{cor:Weak-braid-for-D}The collection of functors $\{s_{\alpha}^{\pm1}\}_{\alpha\in W_{fin}},\{\theta_{\lambda}\}$
defined above satisfy the (weak) braid relations. \end{cor}
\begin{proof}
By definition, $s_{\alpha}$ is the Fourier-Mukai kernal for the complex
of bimodules $Id\to\mathcal{G}_{\alpha}(2)$, and $s_{\alpha}^{-1}$
is the Fourier-Mukai kernel for $\mathcal{G}_{\alpha}(2)\to Id(2)$.
Further, it is easy to verify that the functor $M\to M\otimes_{O(\mathcal{B})}O(\lambda)$
is represented by the bimodule $\tilde{D}_{h}\otimes_{O(\Delta\mathcal{B})}O(\lambda)$;
here $\tilde{D}_{h}$ is taken to be the diagonal bimodule, and the
action of $\tilde{D}_{h}\boxtimes\tilde{D}_{h}^{opp}$ is inherited
from the action on $\tilde{D}_{h}$. 

So, for $?=s_{\alpha}^{\pm}$ or $\lambda$, we let $M_{?}$ denote
the associated bimodule. We consider any braid relation; again we
shall take $T_{s}\theta_{\lambda}=\theta_{\lambda}T_{s}$ for notational
convenience. We wish to show the existence of an isomorphism inside
\[
Hom_{D^{b}(Mod^{G\times\mathbb{G}_{m}}(\tilde{D}_{h}\boxtimes\tilde{D}_{h}^{opp}))}(M_{s}\star M_{\lambda},M_{\lambda}\star M_{s})\tilde{=}
\]
\[
Hom_{D^{b}(Mod^{G\times\mathbb{G}_{m}}(\tilde{D}_{h}\boxtimes\tilde{D}_{h}^{opp}))}(\tilde{D}_{h},M_{-\lambda}\star M_{s^{-1}}\star M_{\lambda}\star M_{s})
\]
 However, the object $M_{-\lambda}\star M_{s^{-1}}\star M_{\lambda}\star M_{s}$
is known to satisfy $(M_{-\lambda}\star M_{s^{-1}}\star M_{\lambda}\star M_{s})/h\tilde{=}O(\Delta\tilde{\mathfrak{g}})$
by \cite{key-11}; this is precisely the fact that the braid relations
are known for $D^{b,G\times\mathbb{G}_{m}}(\tilde{\mathfrak{g}})$.
This implies that $M_{-\lambda}\star M_{s^{-1}}\star M_{\lambda}\star M_{s}$
is a sheaf (i.e., concentrated in a single degree). 

We shall use this to argue that $ $$M_{-\lambda}\star M_{s^{-1}}\star M_{\lambda}\star M_{s}$,
as a quasi-coherent sheaf on $\mathcal{B}\times\mathcal{B}$, is scheme-theoretically
supported on the diagonal $\Delta\mathcal{B}$. The fact that is is
set-theoretically supported there is immediate from the above remark.
To see the scheme=theoretic support, let $\mathcal{I}$ denote the
ideal sheaf of $\Delta\mathcal{B}$ in $\mathcal{B}\times\mathcal{B}$.
Then $(M_{-\lambda}\star M_{s^{-1}}\star M_{\lambda}\star M_{s})/\mathcal{I}$
is a graded sheaf whose grading is bounded below (because $M_{-\lambda}\star M_{s^{-1}}\star M_{\lambda}\star M_{s}$
is a finitely generated, graded $\tilde{D}_{h}\boxtimes\tilde{D}_{h}^{opp}$-
module and $\mathcal{I}$ is an ideal of degree zero elements). Further,
we have that 
\[
(M_{-\lambda}\star M_{s^{-1}}\star M_{\lambda}\star M_{s}/\mathcal{I})/h\tilde{=}(M_{-\lambda}\star M_{s^{-1}}\star M_{\lambda}\star M_{s}/h)/\mathcal{I}=0
\]
 since $(M_{-\lambda}\star M_{s^{-1}}\star M_{\lambda}\star M_{s})/h\tilde{=}O(\Delta\tilde{\mathfrak{g}})$
is scheme-theoretically supported on $\Delta\mathcal{B}$. Thus the
graded Nakayama lemma for sheaves implies that \linebreak{}
 $(M_{-\lambda}\star M_{s^{-1}}\star M_{\lambda}\star M_{s})/\mathcal{I}=0$,
which is what we wanted. 

Given this, we have an isomorphism 
\[
Hom_{D^{b}(Mod^{G\times\mathbb{G}_{m}}(\tilde{D}_{h}\boxtimes\tilde{D}_{h}^{opp}))}(\tilde{D}_{h},M_{-\lambda}\star M_{s^{-1}}\star M_{\lambda}\star M_{s})\tilde{=}
\]
\[
Hom_{D^{b}(Mod^{G\times\mathbb{G}_{m}}(\tilde{D}_{h}))}(\tilde{D}_{h},Rp_{*}(M_{-\lambda}\star M_{s^{-1}}\star M_{\lambda}\star M_{s}))
\]
 since the projection $p$ induces an equivalence of categories between
quasi-coherent sheaves on $\Delta\mathcal{B}$ and those on $\mathcal{B}$
(we need scheme-theoretic support for this, not just set theoretic;
hence the above discussion). 

But now, $Rp_{*}(M_{-\lambda}\star M_{s^{-1}}\star M_{\lambda}\star M_{s})\tilde{=}\theta_{-\lambda}\cdot s^{-1}\cdot\theta_{\lambda}\cdot s\cdot\tilde{D}_{h}$
(by {}``convolution becomes composition'' above), and so we finally
conclude 
\[
Hom_{D^{b}(Mod^{G\times\mathbb{G}_{m}}(\tilde{D}_{h}\boxtimes\tilde{D}_{h}^{opp}))}(M_{s}\star M_{\lambda},M_{\lambda}\star M_{s})\tilde{=}
\]
\[
Hom_{D^{b}(Mod^{G\times\mathbb{G}_{m}}(\tilde{D}_{h}))}(\tilde{D}_{h},\theta_{-\lambda}\cdot s^{-1}\cdot\theta_{\lambda}\cdot s\cdot\tilde{D}_{h})
\]
 and so the result follows from the lemma above. 
\end{proof}
Now let us note that it is possible to define functors associated
to $b$ and $b^{-1}$, where these were elements of $\mathbb{B}_{aff}^{'}$
such that 
\[
\mathcal{R}_{\alpha_{0}}=b^{-1}\mathcal{R}_{\alpha}b
\]
 for a finite root $\alpha$; this then defines an affine root functor
for $D^{b}(Mod(\tilde{D}_{h}))$. From the braid relations it follows
that any two such choices are isomorphic; we shall see an even stronger
uniqueness statement later.

\subsubsection{Tilting Objects}

Now it is straightforward to define the deformation of our tilting
objects. Indeed, for any sequence of finite roots $(\alpha_{1},...,\alpha_{n})$,
and any element $\omega\in\Omega$ we define an object 
\[
\mathcal{R}_{\alpha_{1}}\mathcal{R}_{\alpha_{2}}\cdot\cdot\cdot\mathcal{R}_{\alpha_{n}}\omega\cdot\tilde{D}_{h}
\]
which lives in $D^{b}(Mod^{G\times\mathbb{G}_{m}}(\tilde{D}_{h}))$.
From the definitions and the cohomological lemmas \ref{lem:Cohomology1},\ref{cor:Cohomology2}
it is clear that 
\[
(\mathcal{R}_{\alpha_{1}}\mathcal{R}_{\alpha_{2}}\cdot\cdot\cdot\mathcal{R}_{\alpha_{n}}\omega\cdot\tilde{D}_{h})|_{h=0}=\mathcal{R}_{\alpha_{1}}\mathcal{R}_{\alpha_{2}}\cdot\cdot\cdot\mathcal{R}_{\alpha_{n}}\omega\cdot O_{\tilde{\mathfrak{g}}}
\]

Further, the $G$-equivariant version of these cohomological lemmas
gives:
\begin{lem}
Let $T_{h}$ denote any tilting object in $D^{b}(Mod^{G\times\mathbb{G}_{m}}(\tilde{D}_{h}))$,
and $T$ its reduction mod $h$ as above. Then:

a) $H^{i}(R\Gamma^{G}(T_{h}))=0$ for all $i\neq0$. 

b) $H^{0}(R\Gamma^{G}(T_{h}))|_{h=0}=H^{0}(R\Gamma^{G}(T))$ 
\end{lem}
Then, by applying the self adjointness of the $\mathcal{R}_{\alpha}$
and the fact that \linebreak{}
$Hom_{D^{b}Mod^{G}(\tilde{D}_{h})}(\tilde{D}_{h},\cdot)=R\Gamma^{G}(\cdot)$,
we deduce immediately (from the graded Nakayama lemma) the 
\begin{cor}
The objects $\mathcal{R}_{\alpha_{1}}\mathcal{R}_{\alpha_{2}}\cdot\cdot\cdot\mathcal{R}_{\alpha_{n}}\omega\cdot\tilde{D}_{h}$
satisfy 
\[
End_{D^{b}Coh^{G}(\tilde{D}_{h})}^{\cdot}(\mathcal{R}_{\alpha_{1}}\mathcal{R}_{\alpha_{2}}\cdot\cdot\cdot\mathcal{R}_{\alpha_{n}}\omega\cdot\tilde{D}_{h})=End_{D^{b}Coh^{G}(\tilde{D}_{h})}^{0}(\mathcal{R}_{\alpha_{1}}\mathcal{R}_{\alpha_{2}}\cdot\cdot\cdot\mathcal{R}_{\alpha_{n}}\omega\cdot\tilde{D}_{h})
\]
 The same is true in $D^{b}Coh^{G\times\mathbb{G}_{m}}(\tilde{D}_{h})$
for the objects $\mathcal{R}_{\alpha_{1}}\mathcal{R}_{\alpha_{2}}\cdot\cdot\cdot\mathcal{R}_{\alpha_{n}}\omega\cdot\tilde{D}_{h}(i)$. 
\end{cor}
Therefore, to see that these are tilting objects in the sense that
we need, we should show that they generate the category. As in the
coherent case, one first shows that the objects $\tilde{D}_{h}(\lambda)$
generate our category, and then show that the tilting objects generate
these objects. 

To prove the generation by $\tilde{D}_{h}(\lambda)$'s we only have
to prove the equivalent for $Mod^{B}(U_{h}(\mathfrak{b}))$. But in
this case the proof is identical to the coherent version. 

Next, we can copy the coherent argument to show that the full triangulated
subcategory generated by the tilting objects contains all objects
of the form \linebreak{}
$s_{\alpha_{1}}^{\pm}\cdot\cdot\cdot s_{\alpha_{n}}^{\pm}\cdot\omega\cdot\tilde{D}_{h}$.
Since the weak braid relations are satisfied for objects acting on
$\tilde{D}_{h}$, we deduce right away that this collection contains
all objects $\tilde{D}_{h}(\lambda)$, which is what we needed.

\section{Kostant-Whittaker Reduction and Soergel Bimodules}

The aim of this section is prove our {}``combinatorial'' description
of the coherent categories via the Kostant-Whittaker reduction functor.
By the results of the above sections, all we have to do to completely
encode these categories is to give a description of the Hom's between
tilting generators. We shall show that this can be done entirely in
terms of the action of the (affine) Weyl group on its geometric representation-
hence the use of the adjective {}``combinatorial.''

\subsection{Kostant Reduction for $\mathfrak{g}$.}

In this section, we shall define our {}``functor into combinatorics.''
This definition is a generalization of the main idea of \cite{key-8}.
We shall start by making a few general remarks about the Kostant reduction-
first found in the classic paper \cite{key-26} (c.f. also \cite{key-16}).
Let $\mathfrak{g}$ be our reductive lie algebra (over $k$) with
its fixed pinning, such that $\mathfrak{n}^{-}$ is the {}``opposite''
maximal nilpotent subalgebra. We choose $\chi$ a generic character
for $\mathfrak{n}^{-}$; in other words, we choose a linear functional
on the space $\mathfrak{n}^{-}/[\mathfrak{n}^{-},\mathfrak{n}^{-}]$
which takes a nonzero value on each simple root element $F_{\alpha}$.
Our $\chi$ is the pullback of this functional to $\mathfrak{n}^{-}$,
which is then a character by definition. 

Next, we define a left ideal of the enveloping algebra $U(\mathfrak{g})$,
called $I_{\chi}$, to be the left ideal generated by 
\[
\{n-\chi(n)|n\in\mathfrak{n}^{-}\}
\]
 and we can form the quotient $U(\mathfrak{g})/I_{\chi}$ - naturally
a $U(\mathfrak{g})$-module. It is easy to check that this module
retains the adjoint action of the lie algebra $\mathfrak{n}^{-}$,
and hence we can further define the subspace 
\[
(U(\mathfrak{g})/I_{\chi})^{ad(\mathfrak{n}^{-})}
\]
 of $\mathfrak{n}^{-}$-invariant vectors. (This is equal, in characteristic
zero, to the $N^{-}$-invariant vectors). In (large enough) positive
characteristic, everything still works, but we should work with the
group instead of the algebra. 

As it turns out, this space has the structure of an algebra under
the residue of the multiplication in $U(\mathfrak{g})$ (c.f. \cite{key-16}
section 2 for a more general result). Further, we see that since the
center of $U(\mathfrak{g})$, $Z(\mathfrak{g})$, consists of all
$G$-invariant vectors in $U(\mathfrak{g})$, the natural quotient
map yields a morphism 
\[
Z(\mathfrak{g})\to(U(\mathfrak{g})/I_{\chi})^{ad(\mathfrak{n}^{-})}
\]
 Kostant's theorem assets that in fact this is an algebra isomorphism. 

We can perform the same procedure with $U(\mathfrak{g})$ replaced
by its associated graded version, $S(\mathfrak{g})=O(\mathfrak{g}^{*})$.
Then the quotient by the ideal $gr(I_{\chi})$ corresponds to the
restriction to affine subspace $\mathfrak{n}^{\perp}+\chi\subseteq\mathfrak{g}^{*}$.
Taking invariant vectors then corresponds to taking the quotient of
this affine space by the action of the group $N^{-}$. This quotient
exists (in the sense of GIT), and is isomorphic to an explicitly constructed
affine space, as follows.

We choose a principal nilpotent element in $\mathfrak{n}^{-}$, called
$F$, which can be taken to be the sum of all the $F_{\alpha}$ associated
to simple roots. Then by the well known Jacobson-Morozov theorem,
we can complete $F$ to an $\mathfrak{sl}_{2}$-triple- called $\{E,F,H\}$.
Then we can define the subspace $ker(ad(F))\subset\mathfrak{g}$,
and we can then transfer this space to $\mathfrak{g}^{*}$ via the
isomorphism $\mathfrak{g}\tilde{=}\mathfrak{g}^{*}$, and we shall
denote the resulting space $ker(ad(F))^{*}$. Finally, we define the
Kostant-Slodowy slice to be the affine space 
\[
S_{\chi}:=\chi+ker(ad(F))^{*}
\]
 This space lives naturally inside $\mathfrak{n}^{\perp}+\chi$. What's
more, we have: 
\begin{lem}
The action map 
\[
a:N^{-}\times S_{\chi}\to\mathfrak{n}^{-}+\chi
\]
 is an isomorphism of varieties. 
\end{lem}
Therefore, Kostant's theorem states that the space $S_{\chi}$ is
naturally isomorphic to $\mathfrak{h}^{*}/W$, and that in fact this
isomorphism is realized as the restriction of the natural adjoint
quotient map $\mathfrak{g}^{*}\to\mathfrak{h}^{*}/W$. This is a deep
result, and along the way he proves many interesting facts about $S_{\chi}$.
One which we shall record for later use is:
\begin{prop}
\label{pro:S in reglocus}Every point of $S_{\chi}$ is contained
in the regular locus of $\mathfrak{g}$.
\end{prop}
Let us note one more nice property of the Kostant map. The action
of the principal semisimple element $H$ equips the space $(U(\mathfrak{g})/I_{\chi})^{ad(\mathfrak{n}^{-})}$
with a grading; which for convenience, we shift up by $2$ (c.f. {[}GG{]},
this is called the Kazhdan grading). In addition, $Z(\mathfrak{g})$
is graded by considering the algebra $S(\mathfrak{h})^{W}$ as a subalgebra
of $S(\mathfrak{h})$- which, of course, is graded by putting $\mathfrak{h}$
in degree 2. Then, with these conventions the Kostant map is actually
an isomorphism of graded algebras.

\subsection{Kostant Reduction for $\tilde{\mathfrak{g}}_{\mathcal{P}}$. }

Now we would like to extend the definition of the Kostant reduction
to the varieties $\tilde{\mathfrak{g}}_{\mathcal{P}}$. In fact, we
shall work with the sheaves of algebras $\tilde{D}_{h,\mathcal{P}}$.
The Kostant reduction of these sheaves is easy to define: by using
the natural map $\mathfrak{n}^{-}\to\Gamma(\tilde{D}_{h,\mathcal{P}})$,
we define the sheaf of left ideals $\mathcal{I}_{\chi}$ to be the
left ideal sheaf generated by the image of $\{n-\chi(n)|n\in\mathfrak{n}^{-}\}$. 

Then we form the sheaf of $\tilde{D}_{h,\mathcal{P}}$-modules $\tilde{D}_{h,\mathcal{P}}/\mathcal{I}_{\chi}$,
and, using the residual adjoint action of the group $N^{-}$, we take
$\Gamma(\tilde{D}_{h,\mathcal{P}}/\mathcal{I}_{\chi})^{N^{-}}$. It
is easy to check that this object inherits a multiplication from the
algebra structure of $\tilde{D}_{h,\mathcal{P}}$ (c.f. \cite{key-16}
section 2 for a more general result). 

Further, we can consider the action of the principal semisimple element
$H$ (chosen via the Jacobson-Morozov theorem above), which makes
this into a graded algebra (the element $h$ is in degree $2$). As
above, we shift this grading by $2$. Then, the natural map from the
center
\[
O(\mathfrak{h}^{*}/W_{\mathcal{P}}\times\mathbb{A}^{1})=\Gamma(\tilde{D}_{h,\mathcal{P}})^{G}\to\Gamma(\tilde{D}_{h,\mathcal{P}}/\mathcal{I}_{\chi})^{N^{-}}
\]
(where $\mathfrak{h}^{*}$ is in degree $2$ as well) becomes a morphism
of graded algebras. 
\begin{claim}
This map is a graded algebra isomorphism. \end{claim}
\begin{proof}
Since this is clearly a morphism of flat graded algebras, one reduces
immediately to the coherent case where $h=0$. By the construction
of the spaces involved, we have the isomorphisms 
\[
O(\tilde{\mathfrak{g}}_{\mathcal{P}}/\mathcal{I}_{\chi})\tilde{=}O(\pi^{-1}(\mathfrak{n}^{\perp}+\chi))
\]
 and 
\[
O(\tilde{\mathfrak{g}}_{\mathcal{P}}/\mathcal{I}_{\chi})^{N^{-}}\tilde{=}O(\pi^{-1}(S_{\chi}))
\]

So, we really only have to show that the natural map 
\[
O(\mathfrak{h}^{*}/W_{\mathcal{P}})\to O(\pi^{-1}(S_{\chi}))
\]
 is an isomorphism. 

But we also have the isomorphism 
\[
\tilde{\mathfrak{g}}_{\mathcal{P}}^{reg}\tilde{=}\mathfrak{h}^{*}/W_{\mathcal{P}}\times_{\mathfrak{h}^{*}/W}\mathfrak{g^{*,}}^{reg}
\]
(c.f. \prettyref{lem:Regiso}). Since we already know that $S_{\chi}$
is a closed subscheme of $\mathfrak{g}^{reg}$ which is a section
of the map $\mathfrak{g}^{*}\to\mathfrak{h}^{*}/W$ (by \prettyref{pro:S in reglocus}),
the result follows.\end{proof}
\begin{rem}
From the proof it follows that $\pi^{-1}(S_{\chi})$ is an affine
variety (indeed, it is a copy of affine space). Therefore the space
$\pi^{-1}(\mathfrak{n}^{-}+\chi)$ is a copy of an affine space as
well. So the use of the global sections functor in the definition
is superfluous (see \prettyref{lem:skryabin} below for a more detailed
result in this direction). 
\end{rem}

\subsection{The Functor $\kappa^{'}$ }

We shall now proceed to define the first, naive version of our functor.
This shall be a functor 
\[
\kappa_{\mathcal{P}}^{'}:Mod^{G\times\mathbb{G}_{m}}(\tilde{D}_{h,\mathcal{P}})\to Mod^{gr}(O(\mathfrak{h}^{*}\times\mathbb{A}^{1}))
\]
defined as 
\[
\kappa_{\mathcal{P}}^{'}(M)=\Gamma(M/\mathcal{I}_{\chi})^{N^{-}}
\]
 where the taking of $N^{-}$-invariants is via the adjoint action
of the group $N^{-}\subset G$. The fact that this functor lands in
the category $Mod^{gr}(O(\mathfrak{h}^{*}/W_{\mathcal{P}}\times\mathbb{A}^{1}))$
follows immediately from the discussion in the previous section. 

We should like to consider some general properties of this functor.
First of all, let us note that the sheaf $M/\mathcal{I}_{\chi}$ retains
an $N^{-}$-Equivariance and an action of $H$, the principal semisimple
element (we shall consider the $H$-grading shifted by $2$, as above).
This sheaf has the property that $(M/\mathcal{I}_{\chi})|_{h=0}$
is supported on the variety $\pi^{-1}(\mathfrak{n}^{-}+\chi)$. It
has the further property that if we define a left action of $\mathfrak{n}^{-}$
on it by 
\[
n\cdot x=(n-\chi(n))x
\]
 then this is a nilpotent lie algebra action. This follows, essentially,
from the commutation relations in the enveloping algebra $U_{h}(\mathfrak{g})$. 

To study $\kappa^{'}$, we shall need to develop, briefly, some of
the properties of the functor $\Gamma$, as applied to objects of
the form $M/\mathcal{I}_{\chi}$. To that end, we make the
\begin{defn}
We let $\mathcal{C}_{\chi}$ be the category of modules $M\in Mod(\tilde{D}_{h})$
such that $\mathcal{I}_{\chi}\cdot M\subseteq hM$, such that the
$\chi$-twisted left action of $\mathfrak{n}$ is nilpotent , and
such that $M$ is graded by a semisimple action of the principal semisimple
element $H$. We further demand that $h$ act by degree $2$ with
respect to this grading. 

The morphisms in this category are those which respect all structures.
\end{defn}
We note that for any equivariant module $M$, the object $M/\mathcal{I}_{\chi}$
is in $\mathcal{C}_{\chi}$ (the grading is the Kazhdan grading).
Then, we have the 
\begin{prop}
The functor $\Gamma$ is exact on the subcategory $\mathcal{C}_{\chi}$. \end{prop}
\begin{proof}
This shall follow from the cohomological lemmas \ref{lem:Cohomology1},\ref{cor:Cohomology2}.
In particular, we need that for $M\in\mathcal{C}_{\chi}$, $R\Gamma(M)$
has grading bounded below, and that $M/hM$ has no higher cohomology. 

To see that the grading is bounded below, note that $R\Gamma(M)$
satisfies $\mathcal{I}_{\chi}R\Gamma(M)\subseteq hR\Gamma(M)$, and
that the space $S_{\chi}\times N^{-}$ is positively graded (c.f.
\cite{key-16}, section 2). So, we choose a PBW basis for $U_{h}(\mathfrak{g})$
from a basis of $\mathfrak{g}$ as a graded module, but with elements
$n-\chi(n)$ instead of $n$ for all $n$ with negative grading. Let
elements of the form $n-\chi(n)$ be on the right. Then since they
act nilpotently, and $R\Gamma(M)$ has cohomology consisting of finitely
generated modules, we see that $R\Gamma(M)$ is indeed bounded below. 

Finally, note that $M/hM$ is now supported on the affine variety
$\tilde{S}_{\chi}\times N^{-}$, and hence has no higher cohomology.
The lemma follows. 
\end{proof}
From this proposition, one can go a bit further. Let $A_{\chi}$ be
the subcategory of $U_{h}(\mathfrak{g})\otimes_{O(\mathfrak{h}^{*}/W)}O(\mathfrak{h}^{*})$-modules
$M$ such that $\mathcal{I}_{\chi}\cdot M\subseteq hM$, such that
the $\chi$-twisted left action of $\mathfrak{n}$ is nilpotent, and
such that $M$ admits a semisimple action of the principal semisimple
element $H$, with $h$ acting by degree $2$ elements. Then we have
the
\begin{cor}
The functor $\Gamma$ is an equivalence of categories between $C_{\chi}$
and $A_{\chi}$. \end{cor}
\begin{proof}
To see this, we only need show that $\Gamma$ is conservative; then
the result will follow from the previous proposition and standard
arguments (e.g. \cite{key-19} chapter 1.4). So, let $V\in\mathcal{C}_{\chi}$
be nonzero. Choose $W$ a nonzero coherent subsheaf of $V$ (on $\mathcal{B}$).
Then there exists a line bundle $O(\lambda)$ such that 
\[
\Gamma(W\otimes O(\lambda))\neq0
\]
 and so the same is true of $V$. Next, the exact sequence 
\[
0\to V\otimes O(\lambda)\to V\otimes O(\lambda)\to(V/hV)\otimes O(\lambda)\to0
\]
 (where the first map is multiplication by $h$) gives a surjection
$\Gamma(V\otimes O(\lambda))\to\Gamma((V/hV)\otimes O(\lambda))$
with kernel equal to $h\Gamma(V\otimes O(\lambda))$, by the exactness
of $\Gamma$. Thus the graded Nakayama lemma implies that $\Gamma((V/hV)\otimes O(\lambda))\neq0$.
But now the sheaf $(V/hV)$ lives on a copy of affine space. Thus
any tensor by a line bundle is an isomorphism of sheaves. So we deduce
$\Gamma(V/hV)\neq0$, which by the exact sequence 
\[
0\to V\to V\to(V/hV)\to0
\]
 implies $\Gamma(V)\neq0$, as required. 
\end{proof}
We wish to see what happens after taking $N^{-}$-invariants. To that
end, we state the 
\begin{lem}
\label{lem:skryabin}The functor $M\to M^{N^{-}}$ is exact on the
category $A_{\chi}$. In fact, this functor gives an equivalence from
$A_{\chi}$ to the category of graded $O(\mathfrak{h}^{*}\times\mathbb{A}^{1})$-modules.
\end{lem}
This lemma is really just a restatement of Skryabin's equivalence
in our context. See \cite{key-16}, section six, for a proof (the
same one applies here). 

Thus, we see that by taking flat resolutions, we can consider the
derived functor 
\[
L\kappa_{\mathcal{P}}^{'}:D^{b}Mod^{G\times\mathbb{G}_{m}}(\tilde{D}_{h,\mathcal{P}})\to D^{b}(Mod^{gr}(O(\mathfrak{h}^{*}/W_{\mathcal{P}}\times\mathbb{A}^{1})))
\]
which is obtained by taking the derived functor of the restriction
$M\to M/\mathcal{I}_{\chi}$ and then composing with the invariants
functor (we shall omit the functor $\Gamma$ from now on, which we
can do by the above propositions). 

To see that this is a functor is the appropriate one, we should first
show the 
\begin{prop}
We have $L\kappa_{\mathcal{P}}^{'}(\tilde{D}_{h,\mathcal{P}})=O(\mathfrak{h}^{*}/W_{\mathcal{P}}\times\mathbb{A}^{1})$ \end{prop}
\begin{proof}
The claim is simply that there are no higher derived terms. We note
that by the construction 
\[
L^{i}\kappa_{\mathcal{P}}^{'}(\tilde{D}_{h,\mathcal{P}})|_{h=0}=Tor_{O(\mathfrak{g}^{*})}^{i}(S_{\chi},O_{\tilde{\mathfrak{g}_{\mathcal{P}}}})
\]
 and the term on the right vanishes for nonzero $i$ (c.f. \cite{key-9},
chapter 1, and \cite{key-11}). So the result follows from the graded
Nakayama lemma.
\end{proof}
Below, we shall denote $\kappa_{\mathcal{B}}^{'}$ simply by $\kappa^{'}$,
and for $\mathcal{P}=\mathcal{P}_{s}$, we denote $\kappa_{\mathcal{P}}^{'}$
by $\kappa_{s}^{'}$. 

Now we can state the two main results of this section: 
\begin{lem}
\label{lem:Kappa&Reflection1}If $\alpha_{s}$ is a finite simple
root, then we have a functorial isomorphism 
\[
L\kappa^{'}(\mathcal{R}_{\alpha}M)\tilde{=}O(\mathfrak{h}^{*})\otimes_{O(\mathfrak{h}^{*}/s_{\alpha})}L\kappa_{s}^{'}(M)
\]
for any complex $M\in D^{b}Mod^{G\times\mathbb{G}_{m}}(\tilde{D}_{h})$. 
\end{lem}
and also 
\begin{lem}
\label{lem:Kappa&Reflection2}For any integral weight $\lambda$,
we have a functorial isomorphism 
\[
L\kappa^{'}(M\otimes O(\lambda))\tilde{=}O(\mathfrak{h}^{*}\times\mathbb{A}^{1})_{\lambda}\otimes_{O(\mathfrak{h}^{*}\times\mathbb{A}^{1})}L\kappa^{'}(M)
\]
 where $O(\mathfrak{h}^{*}\times\mathbb{A}^{1})$ is the $O(\mathfrak{h}^{*}\times\mathbb{A}^{1})=Sym(\mathfrak{h}\oplus k\cdot e)$-module
with the action defined as 
\[
(h,a)\cdot m=(h+a\lambda,a)m
\]
 abstractly, this is a one dimensional free $O(\mathfrak{h}^{*}\times\mathbb{A}^{1})$-module,
and so ultimately we get an isomorphism 
\[
L\kappa^{'}(M\otimes O(\lambda))\tilde{=}L\kappa^{'}(M)
\]
 
\end{lem}
We shall prove these lemmas momentarily. Let us note right away, however,
the crucial consequence that $L\kappa^{'}$ takes a tilting module
to a complex concentrated in degree zero- simply because the above
lemmas show that the action of the image of the reflection functors
is exact on $Mod(O(\mathfrak{h}^{*}\times\mathbb{A}^{1}))$, and the
tilting modules are built by applying the reflection functors to the
basic objects $\omega\cdot\tilde{D}_{h}$. 

In order to prove \prettyref{lem:Kappa&Reflection1}, we shall break
the reflection functor into its two pieces. The proof of the lemma
is immediately reducible to the following claim: 
\begin{claim}
a) We have a functorial isomorphism 
\[
L\kappa_{s}^{'}R\pi_{s}{}_{*}(M)\tilde{=}(pr_{s})_{*}(L\kappa^{'}(M))
\]
 (where $pr_{s}:\mathfrak{h}^{*}\to\mathfrak{h}^{*}/s$ is the natural
quotient map) for all $M\in D^{b}Mod^{G\times\mathbb{G}_{m}}(\tilde{D}_{h})$. 

b) We have a functorial isomorphism 
\[
L\kappa^{'}R\pi_{s}^{*}(N)\tilde{=}O(\mathfrak{h}^{*})\otimes_{O(\mathfrak{h}^{*}/s)}L\kappa_{\alpha}^{'}(N)
\]
 for all $N\in D^{b}Mod^{G\times\mathbb{G}_{m}}(\tilde{D}_{h,\mathcal{P}})$. \end{claim}
\begin{proof}
We start with a). There is the obvious natural map of sheaves 
\begin{equation}
R\pi_{s*}(M)/\mathcal{I}_{\chi}\to R\pi_{s*}(M/\mathcal{I}_{\chi})
\end{equation}

which upon taking $N^{-}$ invariants becomes a map 
\[
L\kappa_{s}^{'}(R\pi_{s*}(M))\to pr_{s*}(L\kappa^{'}(M))
\]
 We shall show that the map 4.1 is an isomorphism. First of all, let
us recall that we have the sheaf of algebras $\tilde{D}_{h}(s)$,
which by definition is the coherent pullback $\pi_{s}^{*}(\tilde{D}_{h,\mathcal{P}})$;
we recall the natural surjection $\tilde{D}_{h}(s)\to\tilde{D}_{h}$,
which allows us to regard any $\tilde{D}_{h}$-module as a $\tilde{D}_{h}(s)$-module. 

Now, we can define the ideal sheaf $\mathcal{I}_{\chi}(s)$ to be
the ideal sheaf of $\tilde{D}_{h}(s)$ generated by $\{n-\chi(n)|n\in\mathfrak{n}^{-}\}$,
and then we have an isomorphism 
\[
M/\mathcal{I}_{\chi}\tilde{=}M/\mathcal{I}_{\chi}(s)
\]

(as $\tilde{D}_{h}(s)$-modules) following immediately from the fact
that $\mathcal{I}_{\chi}$ is defined by global generators. So, we
have to compute 
\[
R\pi_{s*}(M/\mathcal{I}_{\chi}(s))
\]
over a given affine subset of $\mathcal{P}$, denoted $U$. To do
that, we should first replace $M$ by a complex of flat $\tilde{D}_{h}(s)$-modules,
$\mathcal{F}^{\cdot}$, and then quotient each term of this complex
by $\mathcal{I}_{\chi}(s)$. We compute cohomology by taking the Cech
complex of this complex. 

But now $\tilde{D}_{h}(s)$ is flat over $\tilde{D}_{h,\mathcal{P}}$
since it is obtained by (the quantization of) base change from the
$\mathbb{P}^{1}$-bundle $\mathcal{B}\to\mathcal{P}$, which implies
that $(\pi_{s})_{*}(N^{\cdot})$ is a complex of $\tilde{D}_{h,\mathcal{P}}$-flat
modules if $N^{\cdot}$ is a complex of $\tilde{D}_{h}(s)$-flat modules
. So we can compute $R\pi_{s*}(M)/\mathcal{I}_{\chi}$ by taking the
Cech complex of $\mathcal{F}^{\cdot}$ and then moding out by $\mathcal{I}_{\chi}$
(no further replacement necessary). These two procedures evidently
yield isomorphic complexes.

So this shows that 4.1 is an isomorphism, and the result we want follows
upon taking $N^{-}$ invariants.

Part b) is simpler- in fact it follows easily from the statement that
composition of pullback functors is the pullback of the composed map. 
\end{proof}
Now we proceed to \prettyref{lem:Kappa&Reflection2}. 
\begin{proof}
(of \prettyref{lem:Kappa&Reflection2}). As a first step we note that
since tensoring by a line bundle is exact, we have 
\[
(M\otimes O(\lambda))/\mathcal{I}_{\chi}\tilde{=}(M/\mathcal{I}_{\chi})\otimes O(\lambda)
\]
 (where the quotient is taken in the derived sense). Then, as a module
over 
\[
\Gamma(\tilde{D}_{h})^{G}=O(\mathfrak{h}^{*}/W\times\mathbb{A}^{1})
\]
 we have that $(M/\mathcal{I}_{\chi})\otimes O(\lambda)$ is simply
\[
O(\mathfrak{h}^{*}\times\mathbb{A}^{1})_{\lambda}\otimes_{O(\mathfrak{h}^{*}\times\mathbb{A}^{1})}M/\mathcal{I}_{\chi}
\]
 by the definition of the $\tilde{D}_{h}$-action on the tensor product
(we use the equivalence of the category of sheaves $\mathcal{C}_{\chi}$
with the {}``affine'' category $\mathcal{A}_{\chi}$ above). Now
the result follows from taking $N^{-}$ invariants. \end{proof}
\begin{rem}
\label{rem:A-natural-question}A natural question is to ask where
the morphisms of the adjunctions $Id\to\mathcal{R}_{\alpha}(2)$ and
$\mathcal{R}_{\alpha}\to Id$ go under $\kappa^{'}$. 

We claim that the former goes to the natural transformation 
\[
M\to M\oplus(\alpha_{s}\otimes M)
\]
 which sends $m\to\alpha_{s}m+(\alpha_{s}\otimes m)$, while the latter
goes to the multiplication map 
\[
M\oplus(\alpha_{s}\otimes M)\to M
\]
 which sends $m_{1}+(\alpha_{s}\otimes m_{2})\to m_{1}+\alpha_{s}m_{2}$. 

These claims shall follow from the explicit description of the adjunction
morphisms on Fourier-Mukai kernels given above. We saw that the morphism
\[
\tilde{D}_{h}\to\mathcal{G}_{\alpha}(2)
\]
 was defined by sending $1$ to the global section $\alpha_{s}\otimes1+1\otimes\alpha_{s}$.
Thus the morphism 
\[
M\tilde{=}p_{2*}((M^{\cdot}\boxtimes\tilde{D}_{h})\otimes_{\tilde{D}_{h}\boxtimes\tilde{D}_{h}^{opp}}^{L}\tilde{D}_{h})\to p_{2*}((M^{\cdot}\boxtimes\tilde{D}_{h})\otimes_{\tilde{D}_{h}\boxtimes\tilde{D}_{h}^{opp}}^{L}\mathcal{G}_{\alpha})\tilde{=}\mathcal{R}_{\alpha}(M)
\]
must send a local section $m\otimes1\otimes1$ to the section $m\otimes1\otimes(1\otimes\alpha_{s}+\alpha_{s}\otimes1)$.
After restriction to the regular elements (i.e., moding out by $\mathcal{I}_{\chi}$)
and taking $N^{-}$ invariants this is evidently the same map as written
above. The argument for the other adjunction is the same. 
\end{rem}

\subsection{The functor $\kappa$. }

Now we shall extend our functor $\kappa^{'}$ to a functor into categories
of bimodules. Let us first recall that there is an equivalence of
categories 
\[
K^{b}(\mathcal{T})\to D^{b}Mod^{G\times\mathbb{G}_{m}}(\tilde{D}_{h})
\]
 Where as above $K^{b}(\mathcal{T})$ denotes the homotopy category
of graded tilting complexes (c.f. section 1.5). So our task shall
be to extend the functors $\kappa^{'}$ to functors $\kappa$ on the
category $\mathcal{T}$. To that end, we note that any $T\in\mathcal{T}$
carries an action of $Z(\mathfrak{g})\tilde{=}O(\mathfrak{h}^{*}/W)$
which is inherited from the right $U(\mathfrak{g})$-module structure
(c.f. the definition of an equivariant $D$-module above). This action
is functorial, and hence the functor 
\[
\kappa^{'}:\mathcal{T}\to Mod^{gr}(O(\mathfrak{h}^{*}\times\mathbb{A}^{1}))
\]
 naturally carries a $Z(\mathfrak{g})$-action. Even better, since
the right $U(\mathfrak{g})$-module structure respects the grading,
we can in fact upgrade to a functor 
\[
\kappa:\mathcal{T}\to Mod^{gr}(O(\mathfrak{h}^{*}\times\mathbb{A}^{1}\times\mathfrak{h}^{*}/W))
\]
 by identifying the category on the right with the category of (graded)
$Z(\mathfrak{g})$-module objects in the category $Mod^{gr}(O(\mathfrak{h}^{*}\times\mathbb{A}^{1}))$. 

We then extend this to a functor 
\[
\kappa:K^{b}(\mathcal{T})\to K^{b}(Mod^{gr}(O(\mathfrak{h}^{*}\times\mathbb{A}^{1}\times\mathfrak{h}^{*}/W)))
\]
 in the canonical way.
\begin{rem}
\label{rem:Dervivedkappa}a)The key identification does not hold on
the level of derived categories- i.e., we cannot identify $D^{b}Mod^{gr}(O(\mathfrak{h}^{*}\times\mathbb{A}^{1}\times\mathfrak{h}^{*}/W))$
with the category of $Z(\mathfrak{g})$-objects in $D^{b}Mod^{gr}(O(\mathfrak{h}^{*}\times\mathbb{A}^{1}))$.
Thus, although we do end up with a functor on the whole derived category
$D^{b}Mod^{G\times\mathbb{G}_{m}}(\tilde{D}_{h})$, the definition
seems to require taking this circuitous route (the same problem occurs
in \cite{key-12}, chapter 3). 

b) There is, however, a certain consistency condition in the definition,
as follows: suppose $M\in D^{b}Mod^{G\times\mathbb{G}_{m}}(\tilde{D}_{h})$
is such that $L\kappa^{'}(M)$ is concentrated in degree zero. Then
$\kappa$ as we have defined it is that same as the object of $Mod^{gr}(O(\mathfrak{h}^{*}\times\mathbb{A}^{1}\times\mathfrak{h}^{*}/W))$
obtained by upgrading $L\kappa^{'}(M)$ with the natural $Z(\mathfrak{g})$-action
. This can be seen easily by taking any homotopy equivalence between
$M$ and a complex of tilting modules, and noting that such an equivalence
respects the natural $Z(\mathfrak{g})$-actions (since it is a $G$-equivariant
equivalence). \end{rem}
\begin{example}
We note right away that $\kappa(\tilde{D}_{h})=O(\mathfrak{h}^{*}\times\mathbb{A}^{1})$
considered with the right action given by the natural inclusions of
algebras 
\[
O(\mathfrak{h}^{*}/W)\to O(\mathfrak{h}^{*})\to O(\mathfrak{h}^{*}\times\mathbb{A}^{1})
\]
 simply because this morphism comes from the residual right action
of $Z(\mathfrak{g})$ on $\tilde{D}_{h}$. 
\end{example}

\subsection{The Categorical Affine Hecke Algebra}

The goal of the next few sections is to extend our key lemmas \ref{lem:Kappa&Reflection1}
and \ref{lem:Kappa&Reflection2} to the functor $\kappa$ itself.
First, we wish to give some context for the final result, extending
the remarks in section 1.2.

So, in the rest of this section, we adopt the language and notation
of section 1.2, in the specific case $W=W_{aff}$. We have $V=\mathfrak{h}^{*}\times\mathbb{A}^{1}$
and $A=O(\mathfrak{h}^{*}\times\mathbb{A}^{1})$. Let us re-state
the constructions of that section explicitly in this language. 

First we define, for a finite root $\alpha_{s}$ 
\[
R_{\alpha_{s}}:=O(\mathfrak{h}^{*}\times\mathbb{A}^{1})\otimes_{O(\mathfrak{h}^{*}\times\mathbb{A}^{1})^{s}}O(\mathfrak{h}^{*}\times\mathbb{A}^{1})
\]

and then we define, for a weight $\lambda$, the bimodule $J_{\lambda}$
to be the module of functions on the graph 
\[
\{(h_{1},h_{2},a)|h_{2}=h_{1}+a\lambda\}
\]
 We note right away that these bimodules yield exact functors on \linebreak{}
$Mod^{gr}(O(\mathfrak{h}^{*}\times\mathbb{A}^{1}\times\mathfrak{h}^{*}/W))$. 

We wish to relate these functors to the categorical constructions
of Rouquier and Soergel. The first observation here is that the geometric
action of the extended affine Weyl group $W_{aff}^{'}$ on the space
$\mathfrak{h}^{*}\times\mathbb{A}^{1}$ is given by the formulas 
\[
w\cdot(v,a)=(w\cdot v,a)
\]
 for $w\in W$, and 
\[
\lambda\cdot(v,a)=(v+a\lambda,a)
\]
 for $\lambda$ in the weight lattice. This action induces an action
of $W_{aff}^{'}$ on the category $Mod^{gr}(O(\mathfrak{h}^{*}\times\mathbb{A}^{1}\times\mathfrak{h}^{*}/W))$,
and this action can be expressed via tensoring by bimodules. In particular,
the action of $\lambda$ is given by tensoring by the $J_{\lambda}$,
while the action of $w\in W_{fin}$ is given by the module $J_{w}$
defined as the module of functions on the graph of $w$, i.e., $\{((v,a),(w\cdot v,a))|v\in\mathfrak{h}^{*},a\in\mathbb{A}^{1}\}$. 

Then it follows directly from the definitions that for finite Coxeter
generators $s_{\alpha}$ we have the exact sequence of bimodules 
\[
J_{s_{\alpha}}(-2)\to R_{\alpha}\to J_{Id}
\]
where $J_{Id}$ is the diagonal bimodule (c.f. \cite{key-29} section
3); in particular, the bimodule $\mathcal{R}_{\alpha}$ is just the
module of functions on the union of the diagonal $\Delta$ and the
graph of $s_{\alpha}$. 

Further, there is a naturally defined affine reflection functor $R_{\alpha_{0}}$
which is given by the bimodule 
\[
O(\mathfrak{h}^{*}\times\mathbb{A}^{1})\otimes_{O(\mathfrak{h}^{*}\times\mathbb{A}^{1})^{s_{\alpha_{0}}}}O(\mathfrak{h}^{*}\times\mathbb{A}^{1})
\]
 and which fits into the analogous exact sequence 
\[
J_{s_{\alpha_{0}}}(-2)\to R_{\alpha_{0}}\to J_{Id}
\]

Then we have $\mathcal{H}_{aff}=\mathcal{H}(W)$, the smallest category
of $A\otimes A$ bimodules containing $J_{Id}$ and $\{R_{\alpha}\}_{\alpha\in S_{aff}}$
and closed under direct sums, summands, and tensor product. 

Soergel's results show that the objects of this category (the Soergel
bimodules) are precisely the summands of the objects of the form $R_{\alpha_{1}}R_{\alpha_{2}}\cdot\cdot\cdot R_{\alpha_{n}}(A)$
for all sequences of simple affine roots.

It follows from this that the same isomorphism is true for the homotopy
category of complexes $K^{b}(\mathcal{H}_{aff})$; and we can even
express the element $q^{-2}T_{s}$ as the image in $K$-theory of
the complex 
\[
R_{\alpha}\to J_{Id}
\]

Now, if we consider $b\in W_{aff}^{'}$ and $\alpha\in S_{fin}$ as
in the definition of the affine reflection functor as given above
(i.e., chosen so that \textbf{$b^{-1}s_{\alpha}b=s_{\alpha_{0}}$}),
then we have the
\begin{claim}
There is an isomorphism 
\[
J_{b^{-1}}R_{\alpha}J_{b}\tilde{=}R_{\alpha_{0}}
\]

where the term on the left indicates convolution of bimodules.\end{claim}
\begin{proof}
We note first that there is a multiplication isomorphism \linebreak{}
 $\phi:J_{b^{-1}}(A\otimes A)J_{b}\tilde{\to}A\otimes A$ given by
\[
\phi(a_{1}\otimes(a_{2}\otimes a_{3})\otimes a_{4})=a_{1}a_{2}\otimes a_{3}a_{4}
\]

Further, the group relations give us isomorphisms 
\[
J_{b^{-1}}AJ_{b}\tilde{\to}A
\]
 
\[
J_{b^{-}}A_{s_{\alpha}}J_{b}\tilde{\to}A_{s_{\alpha_{0}}}
\]
(also given by multiplication) from which we conclude the following:
if we let $I_{w}$ is the kernel of the defining map $A\otimes A\to A_{w}$;
then $\phi(I_{e})=I_{e}$ and $\phi(I_{s_{\alpha}})=I_{s_{\alpha_{0}}}$.
Since $R_{\alpha}=A\otimes A/(I_{e}\cap I_{s_{\alpha}})$ (and the
same for $R_{\alpha_{0}}$), the result follows.
\end{proof}
From here, one checks right away that the natural map $R_{\alpha_{0}}\to Id$
corresponds under conjugation to the map $R_{\alpha}\to Id$. Thus
we deduce that the category of bimodules containing $\{R_{\alpha}\}_{\alpha\in S_{fin}}$
and $\{J_{b^{-1}}R_{\alpha}J_{b}\}$ and closed under direct sums,
summands, and tensor product is equivalent to the category $\mathcal{H}_{aff}$. 

With this in mind, we shall extend $\mathcal{H}_{aff}$ to a categorification
of the extended affine Hecke algebra $H_{aff}^{'}$ corresponding
to any given finite root system, with the finite group $\Omega=\mathbb{Y}/\mathbb{Z}\Phi$.
In particular, we make the 
\begin{defn}
We define $\mathcal{H}_{aff}^{'}$ to be the smallest additive category
of bimodules containing $\{J_{Id}\}$ and closed under the action
of $\mathcal{H}$ and the modules $\{J_{\omega}\}_{\omega\in\Omega}$.
So in particular this category consists of summands of objects of
the form 
\[
R_{\alpha_{1}}\cdot\cdot\cdot R_{\alpha_{n}}(J_{\omega})
\]
 where $\{\alpha_{1},...,\alpha_{n}\}$ is any sequence of affine
roots. 
\end{defn}
We shall now argue that in fact $K(\mathcal{H}_{aff}^{'})\tilde{=}H_{aff}^{'}\tilde{=}\Omega\rtimes H_{aff}$. 

First, let us recall \cite{key-17} that the action of $\Omega$ induces
an action on the set of affine simple roots, which we denote $\alpha_{i}\to\omega(\alpha_{i})$. 

Therefore, for any $\omega\in\Omega$ and any affine simple root $\alpha_{i}$,
we have the isomorphism of complexes 
\[
J_{\omega}(R_{\alpha_{i}}\to J_{Id})J_{\omega^{-1}}\tilde{=}R_{\omega(\alpha_{i})}\to J_{Id}
\]
 by the same reasoning as the claim above. Therefore the map $H_{aff}^{'}\to K(\mathcal{H}_{aff}^{'})$
which takes $\omega T_{s_{i}}$ to $[J_{\omega}]([A\otimes_{A^{s}}A]-1)$
satisfies the correct relations. The surjectivity of this map is evident
from the definition of the category $\mathcal{H}_{aff}^{'}$. In addition,
this relation also demonstrates the above claim about how to index
the objects in $\mathcal{H}_{aff}^{'}$. 

The injectivity of our map will be clear if we know that for any $\omega_{1}\neq\omega_{2}$
and any two strings of affine roots $\{\alpha_{1},...,\alpha_{n}\}$
and $\{\beta_{1},...,\beta_{m}\}$, we have 
\[
Hom(R_{\alpha_{1}}\cdot\cdot\cdot R_{\alpha_{n}}(J_{\omega_{1}}),R_{\beta_{1}}\cdot\cdot\cdot R_{\beta_{m}}(J_{\omega_{2}}))=0
\]
 this will be proved in section \ref{rem:Extended-Hecke-claim} below. 

Finally, we should consider the category of interest to us: 
\begin{defn}
Let $\mathcal{M}_{asp}$ be the smallest additive category of modules
in\linebreak{}
 $Mod^{gr}(O(\mathfrak{h}^{*}\times\mathbb{A}^{1}\times\mathfrak{h}^{*}/W))$
which contains the {}``identity'' module $O(\mathfrak{h}^{*}\times\mathbb{A}^{1})$
and which is closed under the action of $\mathcal{H}_{aff}^{'}$. 
\end{defn}
This is then a category of {}``singular Soergel bimodules'' as defined
in \cite{key-34}. From the results of that paper (and an argument
just like the one above for $\mathcal{H}_{aff}^{'}$) it follows that
$K(\mathcal{M}_{asp})\tilde{=}M_{asp}$, the polynomial representation
of $H_{aff}^{'}$ (c.f. section 1 above); and that one can describe
the objects as summands of actions of the reflection functors on $A$.
The same sorts of descriptions then hold for the homotopy category
of complexes $K^{b}(\mathcal{M}_{asp})$. 
\begin{rem}
The computation of $K(\mathcal{M}_{aff})$ which follows the arguments
of \cite{key-34} is purely algebraic- indeed, that paper follows
the lines of argument of Sorgel's paper \cite{key-32}. In fact, this
computation is not needed to prove the main results of this paper
in section 4.6 below; the equivalences there follow directly from
the definition of the category $\mathcal{M}_{asp}$. Since the computation
of this $K$-group on the geometric side is a well known result (c.f,
e.g., \cite{key-1}), we can obtain from this result on the algebraic
side. 
\end{rem}

\subsection{The Key Properties of $\kappa$. }

We have the following corollary of the previous section: 
\begin{cor}
Let $\tilde{\mathcal{T}}$ be the full subcategory of $D^{b}Mod^{G\times\mathbb{G}_{m}}(\tilde{D}_{h})$
on all objects obtained from $\tilde{D}_{h}$ via repeated application
of the reflection functors or tensoring by a line bundle. Then for
any $M\in\tilde{\mathcal{T}}$, we have functorial isomorphisms: 
\[
\kappa(\mathcal{R}_{\alpha}M)\tilde{=}R_{\alpha}(M)
\]
 and 
\[
\kappa(M\otimes O(\lambda))\tilde{=}J_{\lambda}(M)
\]
\end{cor}
\begin{proof}
The second isomorphism follows from the observation that the right
$U(\mathfrak{g})$-module structure is unaffected by tensoring by
a line bundle. This is because the left $\mathfrak{g}$ structure
is defined by the tensor product rule, as is the adjoint action; so
the action of their difference simply comes from the right $\mathfrak{g}$-action
on $M$. Combining this with the calculation of the left action in
the previous section, we see immediately the isomorphism. 

The first isomorphism is proved by showing that the action of $R\pi_{s*}$
and $\pi_{s}^{*}$ correspond to the pushforward and pullback of bimodules.
But this is an easy generalization of the proofs of \ref{lem:Kappa&Reflection1}and\ref{lem:Kappa&Reflection2}. \end{proof}
\begin{rem}
\label{rem:A-natural-question-2}It also follows right away that the
adjunctions $Id\to\mathcal{R}_{\alpha}(2)$ and $\mathcal{R}_{\alpha}\to Id$
are sent under $\kappa$ to the morphisms described in \prettyref{rem:A-natural-question},
now considered as morphisms of bimodules. These maps make the $R_{\alpha}$
into self adjoint functors on $Mod(A)$. 
\end{rem}
Therefore, we arrive at the 
\begin{cor}
For any tilting object $\mathcal{R}_{\alpha_{1}}\cdot\cdot\cdot\mathcal{R}_{\alpha_{n}}(\tilde{D}_{h})$,
we have that 
\[
\kappa(\mathcal{R}_{\alpha_{1}}\cdot\cdot\cdot\mathcal{R}_{\alpha_{n}}(\tilde{D}_{h}))\tilde{=}R_{\alpha_{1}}\cdot\cdot\cdot R_{\alpha_{n}}(O(\mathfrak{h}^{*}\times\mathbb{A}^{1}))
\]
 as $O(\mathfrak{h}^{*}\times\mathbb{A}^{1}\times\mathfrak{h}^{*}/W)$-modules. 
\end{cor}
This result will now allow us to give a complete description of the
category $K^{b}(\mathcal{T})$, in particular the following
\begin{thm}
The functor $\kappa$ is fully faithful on tilting modules, and thus
induces an equivalence of categories 
\[
D^{b}(Mod^{G\times\mathbb{G}_{m}}(\tilde{D}_{h}))\tilde{\to}K^{b}(\mathcal{T})\tilde{\to}K^{b}(\mathcal{M}_{asp})
\]

\end{thm}
Let us note right away the following 
\begin{cor}
Let 
\[
\mathcal{B}_{h}=\oplus_{i}End_{D^{b}(Mod(\tilde{D}_{h}))}(T_{i})
\]
where the sum runs over all tilting modules. Then we have equivalences
of categories 
\[
D^{b,G\times\mathbb{G}_{m}}(Mod(\tilde{D}_{h}))\tilde{\to}Perf^{gr}(\mathcal{B}_{h})
\]

\[
D^{b,G\times\mathbb{G}_{m}}(Coh(\tilde{\mathfrak{g}}))\tilde{\to}Perf^{gr}(\mathcal{B}_{h}\otimes_{k[h]}k_{0})
\]
 and 
\[
D^{b,G\times\mathbb{G}_{m}}(Coh(\tilde{\mathcal{N}}))\tilde{\to}Perf^{gr}(\mathcal{B}_{h}\otimes_{O(\mathfrak{h}^{*}\times\mathbb{A}^{1})}k_{0})
\]
Where the categories on the right stand for (direct limits of) graded
perfect complexes (c.f. section 1.5). The same statement also holds
true if we remove the $\mathbb{G}_{m}$ from the left and the $gr$
from the right of all these equivalences. 
\end{cor}
This is an immediate consequence of the theorem and the descriptions
of all three categories via tilting modules.

In order to approach the proof of the theorem, we start with an immediate
reduction: 
\begin{lem}
The theorem follows from the statement that 
\[
Hom_{D^{b}Mod^{G}(\tilde{D}_{h})}(\tilde{D}_{h},T)\tilde{=}Hom_{O(\mathfrak{h}^{*}\times\mathbb{A}^{1}\times\mathfrak{h}^{*}/W)}(O(\mathfrak{h}^{*}\times\mathbb{A}^{1}),\kappa T)
\]
(as graded modules) for any tilting module $T$. \end{lem}
\begin{proof}
The proof follows from the self-adjointness property of the reflection
functors, and the obvious adjointness for the action of the line bundles.
On the left hand side, this is discussed above, while on the right
hand side the fact that the $R_{\alpha}$ are self adjoint is explained
\cite{key-29}. Further, we wish to see that the diagram $$ \begin{CD} Hom_{D^{b}Mod^{G}(\tilde{D}_{h})}(\mathcal{R}_{\alpha}T_{1},T_{2}) @> =>> Hom_{D^{b}Mod^{G}(\tilde{D}_{h})}(T_{1},\mathcal{R}_{\alpha}T_{2})  \\ @V\kappa VV            @V\kappa VV  \\ Hom_{O(\mathfrak{h}^{*}\times\mathbb{A}^{1}\times\mathfrak{h}^{*}/W)}(R_{\alpha}\kappa T_1,\kappa T_2)  @> = >> Hom_{O(\mathfrak{h}^{*}\times\mathbb{A}^{1}\times\mathfrak{h}^{*}/W)}(\kappa T_1,R_{\alpha} \kappa T_2)  \\ \end{CD} $$

coming from functoriality commutes. This follows from the various
remarks \ref{rem:A-natural-question} and \ref{rem:A-natural-question-2}
above. 
\end{proof}
So the proof comes down to computing equivariant global sections.
To approach this, let us note that the exact sequences 
\[
s_{\alpha}(-2)\to\mathcal{R}_{\alpha}\to Id
\]
 imply that the object $T$ admits a filtration in $D^{b}(Mod^{G\times\mathbb{G}_{m}}(\tilde{D}_{h}))$
by objects of the form $b\cdot\tilde{D}_{h}(i)$ for $b\in\mathbb{B}_{aff}^{'+}$
(i.e., $b$ is a product of positive elements). Further, let us note
that there are isomorphisms of graded modules 
\[
RHom_{D^{b}(Mod^{G}(\tilde{D}_{h}))}(\tilde{D}_{h},b\cdot\tilde{D}_{h})|_{h=0}\tilde{\to}RHom_{D^{b}(Coh^{G}(\tilde{\mathfrak{g}}))}(O_{\tilde{\mathfrak{g}}},b\cdot O_{\tilde{\mathfrak{g}}})
\]
as follows from the cohomological lemmas \ref{lem:Cohomology1}and
\ref{cor:Cohomology2}. Now, the term on the right is zero for any
$b$ such that $b\cdot0\neq0$ in the representation $M_{asp}$ (this
follows from the description of standard and costandard objects in
section 3). Let us note that this is equivalent to $b\notin W$. Thus
the term on the left is zero for $b\notin W$ also. Further, when
$b\in W_{fin}$, we have $b\cdot\tilde{D}_{h}\tilde{=}\tilde{D}_{h}$
(c.f. proof of \prettyref{lem:Braid-on-D(tilde)}, and \prettyref{cla:W-fin-is-trivial}),
and so 
\[
Hom_{D^{b}(Mod^{G}(\tilde{D}_{h}))}(\tilde{D}_{h},b\cdot\tilde{D}_{h})\tilde{=}O(\mathfrak{h}^{*}\times\mathbb{A}^{1})
\]
 as graded modules.

Now, on the bimodule side, one makes exactly the same type of argument:
the exact sequences for reflection functors imply that 
\[
Hom_{O(\mathfrak{h}^{*}\times\mathbb{A}^{1}\times\mathfrak{h}^{*}/W)}(O(\mathfrak{h}^{*}\times\mathbb{A}^{1}),R_{\alpha_{1}}\cdot\cdot\cdot R_{\alpha_{n}}(O(\mathfrak{h}^{*}\times\mathbb{A}^{1})))
\]
is filtered by terms of the form 
\[
Hom_{O(\mathfrak{h}^{*}\times\mathbb{A}^{1}\times\mathfrak{h}^{*}/W)}(O(\mathfrak{h}^{*}\times\mathbb{A}^{1}),J_{\bar{b}}\cdot O(\mathfrak{h}^{*}\times\mathbb{A}^{1}))
\]
 where $\bar{b}\in W_{aff}$ is a positive element. Now, when $\bar{b}\notin W$,
the module $J_{\bar{b}}O(\mathfrak{h}^{*}\times\mathbb{A}^{1})$ is
the module of functions on an affine subspace of $\mathfrak{h}^{*}\times\mathbb{A}^{1}\times\mathfrak{h}^{*}/W$
which intersects $O(\mathfrak{h}^{*}\times\mathbb{A}^{1})$ in a proper
subspace. Thus the $Hom$'s between them are zero. When $\bar{b}\in W$,
then of course we have $J_{\bar{b}}O(\mathfrak{h}^{*}\times\mathbb{A}^{1})=O(\mathfrak{h}^{*}\times\mathbb{A}^{1})$
(since we are in $\mathfrak{h}^{*}\times\mathbb{A}^{1}\times\mathfrak{h}^{*}/W$)
and so 
\[
Hom_{O(\mathfrak{h}^{*}\times\mathbb{A}^{1}\times\mathfrak{h}^{*}/W)}(O(\mathfrak{h}^{*}\times\mathbb{A}^{1}),J_{\bar{b}}O(\mathfrak{h}^{*}\times\mathbb{A}^{1}))\tilde{=}O(\mathfrak{h}^{*}\times\mathbb{A}^{1})
\]
 as graded modules. Now the proof of the lemma follows by walking
up a standard filtration, and the following easy
\begin{claim}
For any object $b^{+}\omega\cdot\tilde{D}_{h}$, we have $\kappa(b^{+}\omega\cdot\tilde{D}_{h})=J_{\bar{b^{+}\omega}}(O(\mathfrak{h}^{*}\times\mathbb{A}^{1}))$.
The analogous result holds for objects $b^{-}\omega\cdot\tilde{D}_{h}$. \end{claim}
\begin{proof}
We already know the compatibility with the reflection functors, so
the proof follows by applying the usual exact sequences defining the
action of the braid group and the fact that the adjunction maps $\mathcal{R}_{\alpha}\to Id$
and $Id\to\mathcal{R}_{\alpha}(2)$ go to the corresponding maps for
$R_{\alpha}$ on $D^{b,gr}(A-mod)$. \end{proof}
\begin{rem}
\label{rem:Extended-Hecke-claim}Let us note that the argument given
above also proves the unproved claim (of the previous section) that
in $\mathcal{H}_{aff}^{'}$ we have 
\[
Hom(R_{\alpha_{1}}\cdot\cdot\cdot R_{\alpha_{n}}(J_{\omega_{1}}),R_{\beta_{1}}\cdot\cdot\cdot R_{\beta_{m}}(J_{\omega_{2}}))=0
\]
 for $\omega_{1}\neq\omega_{2}$ and any strings of affine roots.
By adjointness and the definition of the action of $\Omega$ on $W_{aff}$,
this comes down to showing that 
\[
Hom(O(\mathfrak{h}^{*}\times\mathbb{A}^{1}),R_{\alpha_{1}}\cdot\cdot\cdot R_{\alpha_{k}}(J_{\omega}))=0
\]
 for any $\omega\neq0$ and any sequence $\{\alpha_{1},...,\alpha_{k}\}$
of affine roots. But now the object $R_{\alpha_{1}}\cdot\cdot\cdot R_{\alpha_{k}}(J_{\omega})$
will have a filtration by objects of the form $J_{s}(i)$ where no
$s$ is $Id$. The claim follows. A similar argument works for the
category $\mathcal{M}_{asp}$. 
\end{rem}

\section{Applications}

In this section, we shall give our two main applications of the above
description. The first is the connection with perverse sheaves on
affine flag manifolds, and the second is the strictification of the
braid group action.

\subsection{Connection with perverse sheaves. }

Let us fulfill our promise from the introduction of the paper. Given
the results of the paper \cite{key-12} this is an easy consequence
of the results of the previous section. Let us recall some generalities
from \cite{key-1,key-12} (c.f. also the references therein). 

We consider the dual reductive group $\check{G}$ over a field $F=\bar{\mathbb{F}}_{p}$.
As is well known, we can associate to $\check{G}$ an ind-scheme (called
the affine flag variety) as follows: we let $F((t))$ be the field
of Laurent series in $F$, and $F[[t]]$ its ring of integers. Then
$\check{G}(F[[t]])$ is a maximal compact subgroup of the topological
group $\check{G}(F((t)))$, and there is a subgroup $I\subseteq\check{G}(F[[t]])$
called the standard Iwahori (it is the inverse image of our standard
Borel in $\check{G}$ under the evaluation map taking $\check{G}(F[[t]])\to\check{G}$).
Then we define $\mathcal{F}l:=\check{G}(F((t)))/I$ with its natural
ind-scheme structure. 

Next, we recall that the action of $I$ on the left of $\mathcal{F}l$
induces a decomposition of $\mathcal{F}l$ into orbits, each of which
is isomorphic to a copy of the affine space $\mathbb{A}^{n}$. Further,
the orbits are indexed by the standard basis of the algebra $H_{\mathbb{X}}$
which we associated to $G$ above (this is a combinatorial manifestation
of Langlands duality). In fact, we can even say that the basis element
$T_{w}$ gives an orbit of length $\mathbb{A}^{l(w)}$. Thus we have
a stratification of the variety $\mathcal{F}l$ which is given by
closures of $I$ orbits, and $I$ acts on each orbit through a finite
quotient; as in section 1 we let $j_{w}$ denote the inclusion of
an orbit into its closure. 

Thus we are in a perfect setting to consider categories of equivariant
constructible sheaves. We let $I^{-}$ be the opposite Iwahori subgroup,
and $\mathbf{I}^{-}$ its associated group scheme. We also consider
their {}``unipotent radicals'' $I_{u}^{-}$ and $\mathbf{I}_{u}^{-}$.
Let $\psi:\mathbf{I}^{-}\to\mathbb{G}_{a}$ be a generic character
(this is the affine lie algebra analogue of the situation of section
4.1, in which Kostant-Whittaker reduction was defined). Then we shall
let $D_{IW}$ denote the triangulated category of bounded complexes
of $(\mathbf{I}^{-},\psi)$-equivariant $\bar{\mathbb{Q}_{l}}$-constructable
sheaves on $\mathcal{F}l$. This category is a main player of \cite{key-1},
along with its mixed version obtained by taking into account the action
of the Frobenius by weights, and denoted $D_{IW,m}$. We note that
these categories are linear over $k=\bar{\mathbb{Q}}_{l}$.

As explained in \cite{key-12}, this category admits several natural
deformations. We first explain how to deform the spaces involved.

We can define the extended affine flag manifold to be $\tilde{\mathcal{F}l}=\check{G}(F((t)))/I_{u}$
which is a natural $\check{T}$-torsor over $\mathcal{F}l$ (here
$\check{T}$ is the maximal torus of $\check{G}$), where the morphism
is just the quotient morphism. 

In addition, we recall that the group $\check{G}(F((t)))$ admits
a one dimensional central extension $\check{\mathcal{G}}$ (this is
an example of a Kac-Moody group), and therefore we can define the
quotient $\check{\mathcal{F}l}=\check{\mathcal{G}}/I_{u}$. This is
naturally a $\check{T}\times\mathbb{G}_{m}$-torsor over $\mathcal{F}l$. 

By the general yoga of sheaves on torsors, we see that the category
$D_{IW}$ is equivalent to a certain category of $\check{T}$- equivariant
sheaves on $\tilde{\mathcal{F}l}$, and a certain category of $\check{T}\times\mathbb{G}_{m}$-equivariant
sheaves on $\check{\mathcal{F}l}$. Because these are tori, Equivariance
for a constructible sheaf is equivalent to demanding that the associated
monodromy action be trivial. We can therefore loosen this condition
by demanding that a torus act with unipotent monodromy. In this way
we obtain categories $D^{b}((I_{u}^{-},\psi)\backslash\tilde{\mathcal{F}l}\rightdash\check{T})$
and $D^{b}((I_{u}^{-},\psi)\backslash\check{\mathcal{F}l}\rightdash\check{T}\times\mathbb{G}_{m})$,
and their mixed versions. We note that taking the logarithm of monodromy
then gives us actions by the polynomial rings $Sym(\check{\mathfrak{t}})$
and $Sym(\check{\mathfrak{t}}\times\mathbb{A}^{1})$, such that the
augmentation ideal acts nilpotently. 

As explained in the appendices to \cite{key-12}, these categories
alone do not have enough objects to be suitable for the Koszul duality
formalism. This is remedied there by defining certain completions
with respect to the action of the rings $Sym(\check{\mathfrak{t}})$
and $Sym(\check{\mathfrak{t}}\times\mathbb{A}^{1})$, roughly analogous
to replacing modules such that the augmentation ideal acts nilpotently
with modules over the completed algebra. 

Therefore we obtain categories denoted $\hat{D}^{b}((I_{u}^{-},\psi)\backslash\check{\mathcal{F}l}\rightdash\check{T}\times\mathbb{G}_{m})$
(and similarly for $\tilde{\mathcal{F}l}$), and mixed versions $\hat{D}_{m}^{b}((I_{u}^{-},\psi)\backslash\check{\mathcal{F}l}\rightdash\check{T}\times\mathbb{G}_{m})$.
Let us say a few words about their structure. First, the category
$D_{IW}$ inherits the $t$-structure of the middle perversity from
$D^{b}(\mathcal{F}l,\mathbb{\bar{Q}}_{l})$, the bounded derived category
of constructible sheaves on $\mathcal{F}l$. The heart of this $t$-structure
is a highest weight category in the sense of section three. In particular,
the standard and costandard objects are given by the collections $\{j_{w!}(\bar{\mathbb{Q}}_{l,Z_{w}}[dimZ_{w}])\}_{W_{aff}/W}$
and $\{j_{w*}(\bar{\mathbb{Q}}_{l,Z_{w}}[dimZ_{w}])\}_{W_{aff}/W}$,
henceforth simply denoted $\{j_{w!}\}$ and $\{j_{w*}\}$. 

We can rephrase this in terms of the convolution action (as defined
in section 1.4); in particular, if we consider $j_{w!}$ and $j_{w*}$
as objects in $D^{b}(\mathcal{F}l,\mathbb{\bar{Q}}_{l})$, then we
see that our standard and costandard collection is given by $\{j_{w!}\star j_{e}\}$
and $\{j_{w*}\star j_{e}\}$. This is exactly the same type of formula
used to define the perversely exotic $t$-structure on coherent sheaves
above; it also works for finite dimensional flag varieties and category
$\mathcal{O}$. By the general theory of highest weight categories,
this category has a tilting collection which generates it.

In \cite{key-12} it is shown that the perverse $t$-structure and
the standard and costandard objects admit deformations to both of
the lifted categories. The deformations of $j_{w*}$ are denoted $\tilde{\nabla}_{w}$,
and $j_{w!}$ deforms to $\tilde{\Delta}_{w}$. In addition, it is
proved there that the deformed categories are generated by tilting
collections, which deform the tilting objects in $D_{IW}$. 

Then the main result is the following 
\begin{thm}
We have equivalences of triangulated categories $$ D^{b,G\times\mathbb{G}_{m}}(Mod(\tilde{D}_{h}))\tilde{\to}\hat{D}_{m}^{b}((I_{u}^{-},\psi)\backslash\check{\mathcal{F}l}\rightdash\check{T}\times\mathbb{G}_{m}) $$
$$ D^{b,G\times\mathbb{G}_{m}}(Coh(\tilde{\mathfrak{g}}))\tilde{\to} \hat{D}_{m}^{b}((I_{u}^{-},\psi)\backslash\tilde{\mathcal{F}l}\rightdash\check{T}) $$
\[
D^{b,G\times\mathbb{G}_{m}}(Coh(\tilde{\mathcal{N}}))\tilde{\to}D_{IW,m}
\]
 Where the varieties on the left are taken over the field $k=\bar{\mathbb{Q}}_{l}$.
These equivalences take tilting modules to tilting modules and they
respect the grading in the sense that the $\mathbb{G}_{m}$-shift
on the left corresponds to shifting the mixed structure on the right. \end{thm}
\begin{proof}
The proof of this theorem is almost immediate from the results of
\cite{key-12} and the previous section. 

The first statement follows from the equivalence $D^{b,G\times\mathbb{G}_{m}}(Mod(\tilde{D}_{h}))\tilde{\to}Perf^{gr}(\mathcal{B}_{h})$.
This equivalence takes tilting modules to summands of $\mathcal{B}_{h}$
and the $\mathbb{G}_{m}$-shift to the shift of grading on the right.
On the other hand, there is a description, proved in \cite{key-12},
chapter 4, of the category on the right of the first statement as
$Perf^{gr}(\mathcal{W}_{h})$ where $\mathcal{W}_{h}$ is an explicitly
defined ind-algebra. This equivalence takes tilting modules to finitely
generated summands of $\mathcal{W}_{h}$, and it takes the shift of
mixed structure to the shift of grading. 

Thus it remains to identify $\mathcal{W}_{h}$ and $\mathcal{B}_{h}$.
This is done via the functor $\mathbb{V}$ of \cite{key-12}; this
functor takes sheaves to modules over the ring $\mbox{Sym}(\check{\mathfrak{t}}\times\mathbb{A}^{1}\times\check{\mathfrak{t}}/W)$,
and is fully faithful on tilting modules. Further, the result of \cite{key-12},
appendix C, says that the image of the tilting modules in $\mbox{Sym}(\check{\mathfrak{t}}\times\mathbb{A}^{1}\times\check{\mathfrak{t}}/W)$
is exactly the category $\mathcal{M}_{asp}$ described above. The
first statement follows. 

We can deduce the other two statements from the first if we know that
$$ \hat{D}_{m}^{b}((I_{u}^{-},\psi)\backslash\tilde{\mathcal{F}l}\rightdash\check{T})\tilde{\to}Perf^{gr}(\mathcal{W}_{h}\otimes_{k[h]}k_{0})$$
and that 
\[
D_{IW,m}\tilde{\to}Perf^{gr}(\mathcal{W}_{h}\otimes_{O(\mathfrak{h}^{*}\times\mathbb{A}^{1})}k_{0})
\]
This is indeed the case, and follows from the description of the deformation
categories in \cite{key-12}, appendices A and B.
\end{proof}
As a corollary, we can give a coherent sheaf interpretation of the
non-completed versions of the categories in \cite{key-12} as well.
These will be the full subcategories of the completed categories on
objects such that $\mbox{Sym(}\check{\mathfrak{t}})$ and $\mbox{Sym}(\check{\mathfrak{t}}\times\mathbb{A}^{1})$
act nilpotently. On the coherent side, one sees from the definitions
that the corresponding full subcategories are those on objects which
are set theoretically supported on $\tilde{\mathcal{N}}$. Thus we
see 
\begin{cor}
We have equivalences of categories $$ D_{\tilde{\mathcal N}}^{b,G\times\mathbb{G}_{m}}(Mod(\tilde{D}_{h}))\tilde{\to}D_{m}^{b}((I_{u}^{-},\psi)\backslash\check{\mathcal{F}l}\rightdash\check{T}\times\mathbb{G}_{m}) $$

$$ D_{\tilde{\mathcal N}}^{b,G\times\mathbb{G}_{m}}(Coh(\tilde{\mathfrak{g}}))\tilde{\to} D_{m}^{b}((I_{u}^{-},\psi)\backslash\tilde{\mathcal{F}l}\rightdash\check{T}) $$

where the subscript $\tilde{\mathcal{N}}$ denotes the full subcategories
on objects set-theoretically supported on $\tilde{\mathcal{N}}$. 

In addition, by looking at the ungraded versions of the underlying
$DG$-algebras, we conclude that are equivalences $$ D_{\tilde{\mathcal N}}^{b,G}(Mod(\tilde{D}_{h}))\tilde{\to}D^{b}((I_{u}^{-},\psi)\backslash\check{\mathcal{F}l}\rightdash\check{T}\times\mathbb{G}_{m}) $$

$$ D^{b,G}(Coh(\tilde{\mathfrak{g}}))\tilde{\to} D^{b}((I_{u}^{-},\psi)\backslash\tilde{\mathcal{F}l}\rightdash\check{T}) $$
\end{cor}

\subsection{Strict Braid Group Action}

We shall now discuss how our results fit in with the main result of
the paper \cite{key-29}, mentioned several times above. Let us recall
some results from that paper. As discussed in section 4, the action
of $W_{aff}^{'}$ on $A$ induces a strict action of $W_{aff}^{'}$
on the category $D^{b}(A-mod)$. This can be written in terms of bimodules
as follows: for any collection of generators $\{s_{1},\ldots,s_{n},\omega\}$,
the multiplication map induces an isomorphism 
\[
A_{s_{1}}\otimes\cdot\cdot\cdot\otimes A_{s_{n}}\otimes A_{\omega}\tilde{\to}A_{w}
\]
 where $w=s_{1}\cdot\cdot\cdot s_{n}\cdot\omega$ in $W_{aff}^{'}$.
This map has a unique inverse. Thus any relation $s_{1}\cdot\cdot\cdot s_{n}\cdot\omega=s_{1}^{'}\cdot\cdot\cdot s_{m}^{'}\cdot\omega^{'}$
yields an isomorphism 
\[
A_{s_{1}}\otimes\cdot\cdot\cdot\otimes A_{s_{n}}\otimes A_{\omega}\tilde{\to}A_{w}\tilde{\to}A_{s_{1}^{'}}\otimes\cdot\cdot\cdot\otimes A_{s_{m}^{'}}\otimes A_{\omega^{'}}
\]
 and this collection of isomorphisms is compatible with the multiplication
of bimodules. Thus this collection yields a strict action of $W_{aff}^{'}$
on $D^{b}(A-mod).$ 

Rouquier's observation is that, if we consider the complexes $R_{\alpha}\to J_{Id}$
(denoted $F_{s_{\alpha}}^{-1}$), which satisfy weak braid relations
(proved in \cite{key-29}, section 3), then for any braid relation
$s_{1}\cdot\cdot\cdot s_{n}=s_{1}^{'}\cdot\cdot\cdot s_{m}^{'}$ we
have an isomorphism 
\[
Hom_{K^{b}(A\otimes A-mod^{gr})}(F_{s_{1}}^{-1}\cdot\cdot\cdot F_{s_{n}}^{-1},F_{s_{1}}^{-1}\cdot\cdot\cdot F_{s_{m}}^{-1})\tilde{\to}End_{K^{b}(A\otimes A-mod^{gr})}(A)=k
\]
 
\[
\tilde{\to}Hom_{D^{b}(A\otimes A-mod^{gr})}(F_{s_{1}}^{-1}\cdot\cdot\cdot F_{s_{n}}^{-1},F_{s_{1}}^{-1}\cdot\cdot\cdot F_{s_{m}}^{-1})
\]

where the second line is by the quasi-isomorphism $F_{s_{\alpha}}^{-1}\tilde{\to}(R_{\alpha}\to J_{Id})(-2)$.
Therefore we can lift the collection of isomorphisms given for $W_{aff}^{'}$,
and obtain strict braid relations. 

Our presentation of the group $\mathbb{B}_{aff}^{'}$ was slightly
different from Rouquier's, but let us note that the isomorphism 
\[
J_{b^{-1}}(R_{\alpha}\to Id)J_{b}\tilde{=}R_{\alpha_{0}}\to Id
\]
 comes from the multiplication map itself; therefore the same proof
shows that this presentation gives strict braid relations as well.
We further deduce the independence of the affine reflection functor
from the choice of $b$. Given this, we can deduce right away the 
\begin{thm}
There are strict actions of the group $\mathbb{B}_{aff}^{'}$ on the
categories \linebreak{}
$D^{b,G\times\mathbb{G}_{m}}(Mod(\tilde{D}_{h}))$, $D^{b,G\times\mathbb{G}_{m}}(Coh(\tilde{\mathfrak{g}}))$,
$D^{b,G\times\mathbb{G}_{m}}(Coh(\tilde{\mathcal{N}}))$, as well
as their ungraded versions. Further, there is a strict action on the
categories $D^{b}(Mod(\tilde{D}_{h}))$, $D^{b}(Coh(\tilde{\mathfrak{g}}))$
and $D^{b}(Coh(\tilde{\mathcal{N}}))$. The actions by braid generators
are given by the functors of section 3.\end{thm}
\begin{proof}
Given the main theorem, the only thing that remains to be done is
to see that the action extends to the non-equivariant categories;
it clearly suffices to consider $D^{b}(Mod(\tilde{D}_{h}))$; the
other two follow by restriction. 

As above we denote, for $?$= $s\in S$ or $\lambda\in\mathbb{X}$,
the bimodule $M_{?}\in D^{b}(\tilde{D}_{h}\boxtimes\tilde{D}_{h}^{opp})$
which represents the functor associated to $?$. For any element of
the braid group, $b$, we wish to give a preferred isomorphism between
any two bimodules given by convolutions of different decompositions
of $b$. As in section 3, we choose, for notational convenience, the
element $b=s\cdot\lambda$ for $<s,\lambda>=0$. So, we wish to give
a preferred isomorphism between $M_{s}\star M_{\lambda}$ and $M_{\lambda}\star M_{s}$. 

We have 
\[
Hom_{D^{b}(Mod^{G\times\mathbb{G}_{m}}(\tilde{D}_{h}\boxtimes\tilde{D}_{h}^{opp}))}(M_{s}\star M_{\lambda},M_{\lambda}\star M_{s})\tilde{=}
\]
\[
Hom_{D^{b}(Mod^{G\times\mathbb{G}_{m}}(\tilde{D}_{h}\boxtimes\tilde{D}_{h}^{opp}))}(\tilde{D}_{h},M_{-\lambda}\star M_{s^{-1}}\star M_{\lambda}\star M_{s})\tilde{=}
\]

\[
Hom_{D^{b}(Mod^{G\times\mathbb{G}_{m}}(\tilde{D}_{h}))}(\tilde{D}_{h},Rp_{*}(M_{-\lambda}\star M_{s^{-1}}\star M_{\lambda}\star M_{s}))\tilde{=}
\]

\[
Hom_{D^{b}(Mod^{G\times\mathbb{G}_{m}}(\tilde{D}_{h}))}(\tilde{D}_{h},\theta_{-\lambda}\cdot s^{-1}\cdot\theta_{\lambda}\cdot s\cdot\tilde{D}_{h})
\]
 exactly as in \prettyref{cor:Weak-braid-for-D}. Now, we can apply
$\kappa$ to this isomorphism and obtain 
\[
Hom_{D^{b}(Mod^{G\times\mathbb{G}_{m}}(\tilde{D}_{h}))}(\tilde{D}_{h},\theta_{-\lambda}\cdot s^{-1}\cdot\theta_{\lambda}\cdot s\cdot\tilde{D}_{h})\tilde{=}
\]
\[
Hom_{D^{b}(A\otimes A-mod^{gr})}(A,F_{-\lambda}\cdot F_{s}^{-1}\cdot F_{\lambda}\cdot F_{s}\cdot A)
\]

The strict braid group action discussed above then produces a preferred
isomorphism in the last $Hom$ space, which we transfer to the first.
It follows easily that the collection of isomorphisms obtained this
way produces a strict braid group action on $D^{b}(Mod(\tilde{D}_{h}))$. 
\end{proof}
By the results of section 5.1, we obtain a compatible action on all
the categories of perverse sheaves considered. In fact, this action
is the same as the usual one, constructed, e.g., in \cite{key-3}
(c.f. also \cite{key-29}, section 6). Let us recall this action:
for each affine simple root $s_{\alpha}$, we let $j_{s_{\alpha}*}$
and $j_{s_{\alpha}!}$ denote the standard and costandard objects,
respectively, of the $I$-orbit associated to $s_{\alpha}$. As noted
above, these objects have deformations to the category $\hat{D}_{m}^{b}((I_{u}^{-},\psi)\backslash\check{\mathcal{F}l}\rightdash\check{T}\times\mathbb{G}_{m})$,
which we denote $\tilde{\Delta}_{s_{\alpha}}$ and $\tilde{\nabla}_{s_{\alpha}}$.
Then the action of the braid generators is given by convolution with
respect to $\tilde{\nabla}_{s_{\alpha}}$, with the inverse being
given by convolution with respect to $\tilde{\Delta}_{s_{\alpha}}$;
this is of course consistent with the description of the affine Hecke
algebra in terms of perverse sheaves. 

Let us denote by $\tilde{\delta}$ the deformation of the constant
sheaf on the trivial orbit. Then, according to \cite{key-12}, appendix
C, there are exact sequences 
\[
0\to\tilde{\Delta}_{s_{\alpha}}\to\tilde{\mathcal{T}}_{s_{\alpha}}\to\tilde{\delta}(1/2)\to0
\]
 
\[
0\to\tilde{\delta}(1/2)\to\tilde{\mathcal{T}}_{s_{\alpha}}\to\tilde{\nabla}_{s_{\alpha}}\to0
\]
 where $\tilde{\mathcal{T}}_{s_{\alpha}}$ is the (free-monodromic)
tilting sheaf associated to $s_{\alpha}$, and the $(1/2)$ denotes
Tate twist. The explicit calculation done there confirms that the
action of functor $\mathbb{V}$ transforms these exact sequences into
the ones which define the braid generators on $K^{b}(A-mod)$. Further,
we recall that $\mathbb{V}$ respects convolution. Thus we conclude
that the actions coincide. 

An immediate corollary of this is that the equivalence constructed
here corresponds, at least on objects, with the one constructed in
\cite{key-1}, and therefore that the perversely exotic $t$-structure
corresponds to the heart of the perverse $t$-structure of $D_{IW}$
(as was also proved in \cite{key-9}). 

\medskip{}
Department of Mathematics, Massachusetts Institute of Technology 

cdodd@math.mit.edu

\begin{thebibliography}{References}
\bibitem[AB]{key-1}S. Arkhipov and R. Bezrukavnikov, \emph{Perverse
sheaves on affine flags and langlands dual group}, Israel Journal
of Mathematics Volume 170, Number 1, 135-183. 

\bibitem[BK]{key-2}E. Backelin, K. Kremnizer, \emph{On Singular Localization
of $\mathfrak{g}$-Modules},\emph{ }arxiv:1011.0896. 

\bibitem[BB]{key-3}A. Beilinson, J. Bernstein, \emph{A generalization
of Casselman's submodule theorem}, Progr. Math. 40, 35-52, Birkhäuzer,
Boston, Mass (1983). 

\bibitem[BBM]{key-4}A. Beilinson, R. Bezrukavnikov, I. Mirkovic,
\emph{Tilting Exercises}, Mosc. Math. J. 4 (2004), no. 3, 547-557,
782. 

\bibitem[BGS]{key-5}A. Beilinson, V. Ginzburg, W. Soergel, \emph{Koszul
Duality Patterns in Representation Theory}, J. Amer. Math. Soc. 9
(1996), 473-527.

\bibitem[B1]{key-6}R. Bezrukavnikov, \emph{Cohomology of tilting
modules over quantum groups and t-structures on derived categories
of coherent sheaves}, Invent. Math. 166 (2006), 327-357.

\bibitem[B2]{key-7}R. Bezrukavnikov, \emph{Quasi-exceptional sets
and equivariant coherent sheaves on the nilpotent cone}, Represent.
Theory 7 (2003), 1-18.

\bibitem[BF]{key-8}R. Bezrukavnikov and M. Finkelberg, \emph{Equivariant
Satake Category and Kostant\textendash{}Whittaker Reduction}, Moscow
Mathematical Journal Volume 8, Number 1, January\textendash{}March
2008, Pages 39\textendash{}72.\emph{ }

\bibitem[BFG]{key-35}R. Bezrukavnikov, M. Finkelberg, and V. Ginzburg,
Cherednik algebras and Hilbert schemes in characteristic p , with
an appendix by P. Etingof, Represent. Theory 10 (2006), 254--298. 

\bibitem[BM]{key-9}R. Bezrukavnikov, I. Mirkovic, \emph{Representations
of semisimple Lie algebras in prime characteristic and noncommutative
Springer resolution}, arXiv:1001.2562. 

\bibitem[BMRI,II ]{key-10}R. Bezrukavnikov, I. Mirkovic, D. Rumynin,
\emph{Localization of modules for a semisimple Lie algebra in prime
characteristic}, Annals of Mathematics, Vol. 167 (2008), No. 3, 945\textendash{}991;
\emph{Singular localization and intertwining functors for semisimple
Lie algebras in prime characteristic, }Nagoya Math. J. Volume 184
(2006), 1-55. 

\bibitem[BR]{key-11}R. Bezrukavnikov, S. Riche, \emph{Affine braid
group actions on derived categories of Springer resolutions}, arXiv:1101.3702. 

\bibitem[BY]{key-12}R. Bezrukavnikov, Z. Yun, \emph{On Koszul duality
for Kac-Moody groups}, arXiv:1101.1253. 

\bibitem[CG]{key-13}N. Chriss and V. Ginzburg, \emph{Representation
Theory and Complex Geometry}, Birkhauser Boston, 1997. 

\bibitem[G1]{key-14}D. Gaitsgory, \emph{Geometric Representation
Theory}, available at http://www.math.harvard.edu/\textasciitilde{}gaitsgde/267y/catO.pdf. 

\bibitem[G2]{key-15}D. Gaitsgory, \emph{Construction of central elements
in the affine Hecke algebra via nearby cycles}, Invent. Math. 144
(2001), no. 2, 253\textendash{}280. 

\bibitem[GG]{key-16}W.L. Gan and V. Ginzburg, \emph{Quantization
of Slodowy Slices}, Int. Math. Res. Not. 5 (2002), 243-255.

\bibitem[Hai]{key-17}M. Haiman, \emph{Cherednik algebras, Macdonald
polynomials and combinatorics}, Proceedings of the International Congress
of Mathematicians, Madrid 2006, Vol. III, 843-872. 

\bibitem[Ha]{key-18}R. Hartshorne, \emph{Residues and Duality}, Lecture
Notes in Mathematics, Vol. 20, Springer-Verlag, 1966. 

\bibitem[HTT]{key-19}R. Hotta, K. Takeuchi, and T. Tanisaki,\emph{
D-modules, Perverse Sheaves, and Representation Theory}, Progress
in Mathematics, 236, Birkhauser Boston, 2008.

\bibitem[Huy]{key-20}Huybrechts, D., \emph{Fourier-Mukai Transforms
in Algebraic Geometry},\emph{ }Oxford Mathematical Monographs, Oxford
University Press, 2006. 

\bibitem[J]{key-21}J. Jantzen, \emph{Representations of Algebraic
Groups, second edition}, American Mathematical Society (2003).

\bibitem[Kas]{key-22}M. Kashiwara, \emph{Equivariant derived category
and representation of real semisimple Lie groups}, in CIME Summer
school Representation Theory and Complex Analysis, Cowling, Frenkel
et al. eds. LNM 1931. 

\bibitem[KL]{key-23}D. Kazhdan and G. Lusztig, \emph{Proof of the
Deligne-Langlands conjecture for Hecke algebras}, Invent. math. Volume
87, Number 1, 153-215.

\bibitem[KL2]{key-24}D. Kazhdan and G. Lusztig, \emph{Schubert varieties
and Poincare duality}, Proc. Symp. Pure Math. XXXVI (1980), pp. 185-203.

\bibitem[Ke]{key-25}B. Keller, \emph{Derived Categories and Tilting,}
\emph{Handbook of Tilting Theory}, London Mathematical Society Lecture
Note Series, 332, Cambridge University Press, 2007. 

\bibitem[K]{key-26}B. Kostant, \emph{On Whittaker vectors and representation
theory}, Invent. Math. 48 (1978), 101-184.

\bibitem[Mil]{key-27}D. Milicic, \emph{Localization and Representation
Theory of Reductive Lie Groups}, available at http://www.math.utah.edu/\textasciitilde{}milicic.

\bibitem[Ri]{key-28}S. Riche, \emph{Geometric braid group action
on derived categories of coherent sheaves}, Represent. Theory 12 (2008),
131-169.

\bibitem[R]{key-29}R. Rouquier, \emph{Categorification of the braid
groups}, math.RT/0409593. 

\bibitem[S1]{key-30}W. Soergel, \emph{Kategorie O, perverse Garben,
und Moduln über den Koinvarianten zur Weylgruppe, }Journal of the
AMS 3, 421-445 (1990).

\bibitem[S2]{key-31}W. Soergel, \emph{The combinatorics of Harish-Chandra
bimodules}, J. Reine Angew. Math. 429 (1992) 49\textendash{}74. 

\bibitem[S3]{key-32}W. Soergel, \emph{Kazhdan-Lusztig-Polynome und
unzerlegbare Bimoduln uber Polynomringen}, Journal of the Inst. of
Math. Jussieu 6(3) (2007) 501\textendash{}525.

\bibitem[S4]{key-33}W. Soergel, \emph{Kazhdan-Lusztig-Polynome und
eine Kombinatorik fur Kipp-Moduln}, Representation Theory 1 (1997),
37\textendash{}68.

\bibitem[W]{key-34}G. Williamson, \emph{Singular Soergel bimodules},
arXiv:1010.1283. 

\end{thebibliography}
\end{document}